\documentclass{article}
\usepackage{graphicx} 
\usepackage{amsmath,amssymb,amsthm,graphicx,float}
\usepackage{caption}
\usepackage{subcaption}
\usepackage{psfrag,url}
\usepackage{enumerate}
\usepackage{hyperref}
\hypersetup{
    colorlinks=true, 
    linkcolor=blue, 
    urlcolor=black, 
    citecolor=blue,
    linktoc=all 
}
\usepackage[normalem]{ulem}
\usepackage{tikz}
\usetikzlibrary{arrows.meta,hobby}

\usepackage{makecell}

\newcommand{\calA}{\mathcal{A}}
\newcommand{\calB}{\mathcal{B}}
\newcommand{\calD}{\mathcal{D}}

\newcommand{\calH}{\mathcal{H}}
\newcommand{\calL}{\mathcal{L}}
\newcommand{\calP}{\mathcal{P}}

\newcommand{\calS}{\mathcal{S}}

\newcommand{\upcalW}[1]{\overline{\mathcal{W}_{#1}}}

\newcommand{\calC}{\mathcal{C}}

\newcommand{\C}{\mathbb{C}}

\newcommand{\hatC}{\widehat{\mathbb{C}}}

\newcommand{\D}{\mathbb{D}}

\renewcommand{\H}{\mathbb{H}}

\newcommand{\R}{\mathbb{R}}
\newcommand{\hatR}{\widehat{\mathbb{R}}}

\newcommand{\ophom}[1]{\operatorname{Homeo}^+(#1)}

\newcommand{\AC}{\ensuremath{\mathrm{AC}}}
\newcommand{\ACloc}{\ensuremath{\mathrm{AC}_{loc}}}
\newcommand{\ACp}{\ensuremath{\mathrm{AC}^+_{loc}}}
\def\Aut{\operatorname{Aut}}
\def\MOB{\operatorname{M\ddot{o}b}}

\def\T{{\mathbb T}}
\newcommand{\wC}{\widehat \C}
\newcommand{\wR}{\widehat \R}
\newcommand{\brac}[1]{\left \langle #1 \right \rangle}
\newcommand{\SSS}{\mathbb{S}^1}
\newcommand{\dd}{\mathrm{d}}

\def\PSL{\operatorname{PSL}}
\def\SLE{\operatorname{SLE}}

\def\QS{\operatorname{QS}}

\def\WP{\operatorname{WP}}
\def\Mob{M\"obius }

\newcommand{\g}{\gamma}

\renewcommand{\Im}{\textnormal{Im}}

\newcommand{\gvp}{\varphi}

\newcommand{\half}{\frac{1}{2}}
\newcommand{\Teich}{Teichm\"{u}ller }

\newcommand{\bs}{\backslash}

\newcommand{\abs}[1]{\left\lvert #1 \right \rvert}

\newtheorem{thm}{Theorem}[section]
\newtheorem{prop}[thm]{Proposition}
\newtheorem{lemma}[thm]{Lemma}
\newtheorem{cor}[thm]{Corollary}
\newtheorem{con}[thm]{Conjecture}

\theoremstyle{definition}
\newtheorem{remark}[thm]{Remark}

\newtheorem{definition}[thm]{Definition}
\newtheorem{example}[thm]{Example}

\numberwithin{equation}{section}

\renewcommand{\tabcolsep}{0.2cm}
\renewcommand{\arraystretch}{1.2}

\begin{document}
\title{On Loewner energy and curve composition}
\author{Tim Mesikepp\thanks{Email: tmesikepp@gmail.com}, 
Yaosong Yang\thanks{Institute of Mathematics, Academy of Mathematics and Systems Science, Chinese Academy of Sciences, Beijing, 100190, China;~School of Mathematical Sciences, University of Chinese Academy of Sciences, Beijing 100049, China. ~Email: yangyaosong@amss.ac.cn}}

\date{\today}

\maketitle
\begin{abstract}
The composition $\gamma \circ \eta$ of Jordan curves $\gamma$ and $\eta$ in universal \Teich space is defined through the composition $h_\gamma \circ h_\eta$ of their conformal weldings.  We show that whenever $\gamma$ and $\eta$ have finite Loewner energy $I^L$, the energy of their composition satisfies $$I^L(\gamma \circ \eta) \lesssim_K I^L(\gamma) + I^L(\eta),$$ with an explicit constant in terms of the quasiconformal $K$ of $\gamma$ and $\eta$.  We also study the asymptotic growth rate of the Loewner energy under $n$ self-compositions $\gamma^n := \gamma \circ \cdots \circ \gamma$, showing $$\limsup_{n \rightarrow \infty} \frac{1}{n}\log I^L(\gamma^n) \lesssim_K 1,$$ again with explicit constant. 

Our approach is to define a new conformally-covariant rooted welding functional $W_h(y)$, and show $W_h(y) \asymp_K I^L(\gamma)$ when $h$ is a welding of $\gamma$ and $y$ is any root (a point in the domain of $h$).  In the course of our arguments we also give several new expressions for the Loewner energy, including generalized formulas in terms of the Riemann maps $f$ and $g$ for $\gamma$ which hold irrespective of the placement of $\gamma$ on the Riemann sphere, the normalization of $f$ and $g$, and what disks $D, \overline{D}^c \subset \hatC$ serve as domains.  An additional corollary is that $I^L(\gamma)$ is bounded above by a constant only depending on the Weil--Petersson distance from $\gamma$ to the circle. 
\end{abstract}

\section{Introduction and main results}

\subsection{Motivation}

The Loewner energy, despite being defined deterministically, owes its genesis to the fecund soil of probability theory. Friz--Shekar \cite{ShekharFriz17} and Wang \cite{Yilin19} independently initiated its study, both in probabilistic settings: the former sought to give Rohde--Schramm-like theorems for drivers in Loewner's equation more general than Brownian motion, while the latter used the reversibility of Schramm--Loewner evolution SLE to show the reversibility of Loewner energy.  While both these initial papers studied Jordan arcs connecting boundary points of a simply-connected domain $\Omega \subset \C$, Rohde--Wang \cite{Loopenergy} subsequently generalized the energy to Jordan curves on the Riemann sphere $\wC$.\footnote{For further connections between Loewner energy and probability, see, for instance, \cite{Interplay,YinlinAnoteBLM,Guskov,FredrikCoulomb,LDP24EveliinaYilin,YilinNotice24,YilinFoliations24,Baverez24Loopmeasure,Yilin24CMP,FanJinwoo2025quasiinvariancesleweldingmeasures,BJ25conformalwelding}.}

\emph{Composition} of curves, however, was not immediately in sight, at least not until Wang's deep work \cite{Yilinequivalent} giving expressions for the Loewner energy in terms of  conformal maps instead of the associated driving function (see \S2 for undefined terminology and further background).  Recall that for a Jordan curve $\gamma \subset \mathbb{C}$ with Riemann maps $f:\D \rightarrow \Omega$, $g:\D^* \rightarrow \Omega^*$ to the two components of the complement $\wC \bs \gamma$ of $\gamma$, Wang showed
\begin{align}\label{Eq:LEDisk}
    I^L(\gamma) = \frac{1}{\pi}\int_{\mathbb{D}} \left| \frac{f''}{f'} \right|^2 dA + \frac{1}{\pi}\int_{\mathbb{D}^*} \left| \frac{g''}{g'} \right|^2 dA + 4 \log \left| \frac{f'(0)}{g'(\infty)} \right|,
\end{align}
provided $g(\infty)=\infty$.  When $\gamma$ passes through $\infty$, Wang similarly showed
\begin{align}\label{Eq:LEH}
    I^L(\gamma) = \frac{1}{\pi}\int_{\mathbb{H}} \left| \frac{f''}{f'} \right|^2 dA + \frac{1}{\pi}\int_{\mathbb{H}^*} \left| \frac{g''}{g'} \right|^2 dA,
\end{align}
for Riemann maps $f:\H \rightarrow \Omega$, $g:\H^* \rightarrow \Omega^*$ from the upper half plane and its complement which both fix $\infty$.
Among the many striking aspects of these formulas is the fact that \eqref{Eq:LEDisk}, up to multiplicative constant, had been previously shown by Takhtajan--Teo to be finite if and only if $\gamma$ belongs to the \emph{Weil--Petersson} class $T_0(1)$.  Weil--Petersson curves, first studied by Cui \cite{Cui} (although not under that name), are a type of $L^2$ class within the $L^\infty$ class of universal \Teich space $T(1)$.  Wang thereby showed finite-energy curves are nothing other than Weil--Petersson curves.  And, in fact, the connection is much deeper, as \eqref{Eq:LEDisk} is also a constant multiple of Takhtajan--Teo's ``universal Liouville action" functional $\mathrm{S}_1$, which is the K\"ahler potential for the Weil--Petersson metric on $T_0(1)$.\footnote{The Weil--Petersson metric is, up to constant multiple, the unique homogeneous K\"ahler metric on $T_0(1)$~\cite[Ch.II Cor. 4.2]{TTbook}. 

While introduced under a different name in \cite{Schiffer62Connections}, we also comment that $\mathrm{S}_1$ is sometimes called the ``Takhtajan-Teo functional" in the geometry literature.  Owing to Wang's work is it usually known as the Loewner energy in the probability and analysis communities.}  See \cite{TTbook,Yilinequivalent}.

In particular, the Loewner energy has some connection with the algebraic structure on $T(1)$ and $T_0(1)$, which is given by composition.  We recall that the composition $\gamma \circ \eta$ of curves is defined through the composition of the corresponding \emph{conformal weldings}, where if $\gamma$ has conformal maps $f$ and $g$ as above, the conformal welding is the circle homeomorphism $h:= g^{-1}\circ f$. Hence $\gamma \circ \eta$ is the curve with conformal welding $h_\gamma \circ h_\eta$.  While $T(1)$ is thus closed under composition, $T_0(1)$ is additionally a topological group \cite[Ch.I Thm. 3.8]{TTbook}.  It is only natural, then, to ask how $I^L(\gamma \circ \eta)$ relates to $I^L(\gamma)$ and $I^L(\eta)$.

\subsection{Two main inequalities: energy of a composition}\label{ideas}

The first elementary observation is that it is not possible for the energy to be either multiplicative or additive, i.e. it cannot satisfy
\begin{align*}
    I^L(\gamma \circ \eta) = I^L(\gamma)I^L(\eta) \qquad \text{ or } \qquad  I^L(\gamma \circ \eta) = I^L(\gamma)+I^L(\eta).
\end{align*}
To see this, recall that inversion of weldings corresponds to complex conjugation of curves, which preserves Loewner energy. On the other hand, the sub-additivity relation
\begin{align}\label{Ineq:LESubadditive}
I^L(\gamma \circ \eta) \leq I^L(\gamma)+I^L(\eta)
\end{align}
seems \emph{a priori} possible.  Our first main result is this inequality, up to a constant depending on $\gamma$ and $\eta$.
\begin{thm}\label{Thm:GeneralCompBound}
    For Jordan $\gamma$ and $\eta$ of finite Loewner energy that are also $K$-quasicircles,
    \begin{align}\label{Ineq:LESubadditiveK}
        I^L(\gamma \circ \eta) \leq (4+K^2)^2\big( I^L(\gamma) + I^L(\eta) \big).
    \end{align}
\end{thm}
\noindent When extending the Loewner energy from Jordan arcs to curves, Rohde and Wang showed that all finite-energy curves are indeed $K$-quasicircles, with the constant a function of the Loewner energy \cite[Prop. 3.6]{Loopenergy}.  Thus $K$ in \eqref{Ineq:LESubadditiveK} may be taken as the maximum of $K_\gamma$ and $K_\eta$.  


One could also inquire about the long-term growth rate of energy under self composition.  Our second main inequality studies this by means of a ``Loewner entropy."

\begin{thm}\label{Thm:LEGrowth}
    If $\g$ is a Jordan curve of finite Loewner energy and a $K$-quasicircle, then
    \begin{align}\label{ineq: entropyineq}
        \limsup_{n \to \infty} \frac{\log I^{L}(\gamma^n)}{n} 
        \leq \log(K^2+K^{-2}).
    \end{align}
\end{thm}
\noindent Hence for any $\varepsilon>0$, $I^L(\gamma^n) < \exp\big( (\log(K^2+K^{-2}) + \varepsilon)n \big)$ whenever $n \geq N(\varepsilon)$.

\subsubsection{Discussion, criticism, and a related result}\label{Sec:IntroMainInequalitiesDiscussion}
Before delving into proof strategies, we offer some brief discussion and criticism of these inequalities, and review a related result.

Overall, we do not expect \eqref{Ineq:LESubadditiveK} and \eqref{ineq: entropyineq} to be the last word on Loewner energy and curve composition, but rather hope they are a stimulating starting point.  In particular, the constants in \eqref{Ineq:LESubadditiveK} and \eqref{ineq: entropyineq} are not sharp (take $\eta = \gamma^{-1}$ in the former, for instance).

We also acknowledge that inequality \eqref{Ineq:LESubadditiveK} is admittedly rather far from \eqref{Ineq:LESubadditive}, given the dependence upon the curves.  However, note that, for each $K \geq 1$, \eqref{Ineq:LESubadditiveK} yields a fixed constant of sub-additivity on each collection 
\begin{align}\label{Eq:Gamma_K}
    \Gamma_K := \{\,\gamma \; : \; I^L(\gamma)<\infty, K_\gamma \leq K \,\}.
\end{align}
That is, $I^L(\gamma \circ \eta) \lesssim I^L(\gamma) + I^L(\eta)$ on every $\Gamma_K$.

We also note that while \eqref{ineq: entropyineq} bounds the exponential growth rate of the Loewner energy, it does not say that the limit supremum is, in general, positive.  Thus if the Loewner energy actually grows \emph{sub}-exponentially, then \eqref{ineq: entropyineq} is trivial because the left-hand side is always zero.  This would be the case, for instance, if \eqref{Ineq:LESubadditive} holds. On the other hand, if one can produce a curve $\gamma$ for which the left-hand side of \eqref{ineq: entropyineq} is positive, then \eqref{Ineq:LESubadditive} does not hold, and furthermore it is not the case that $I^L(\gamma \circ \eta) \lesssim I^L(\gamma) + I^L(\eta)$ for any constant independent of $
\gamma$ and $\eta$.

The only other result on Loewner energy and composition that we are aware of is from recent work of Alekseev--Shatashvili--Takhtajan~\cite{Takhtajan24}, who show, in our notation, that
\begin{align}\label{Eq:TStrikesAgain}
    I^L(\gamma \circ \eta) = I^L(\gamma) + I^L(\eta) - 48\log(|C_N(f_\gamma,f_\eta)|)
\end{align}
when $\gamma$ and $\eta$ are analytic Jordan curves \cite[Thm. 1.2]{Takhtajan24}.\footnote{Note the sign of the log term in the arXiv and published versions may be different.}  Here $f_\gamma$ and $f_\eta$ are associated conformal maps, and $C_N(\cdot,\cdot)$ is a $2$-cocycle defined in terms of certain Grunsky coefficients~\cite[Thm.~5.4]{Takhtajan24}.  The sign of the log term, however, does not seem to be known, and so this does not immediately yield a general inequality between $I^L(\gamma \circ \eta)$ and $I^L(\gamma)$ and $I^L(\eta)$. Note that, in contrast to \eqref{Eq:TStrikesAgain}, our results have the virtue of holding for all finite-energy curves, and not just analytic ones.

\subsubsection{Outline of the rest of the introduction}

In the remainder of the introduction we sketch our approach to proving Theorems \ref{Thm:GeneralCompBound} and \ref{Thm:LEGrowth} and highlight auxiliary contributions of our methods.  In \S\ref{Sec:IntroApproachWeldingEnergy} we introduce our ``welding energy" and inequality Theorem \ref{Comparable} relating the Loewner and welding energies, which is our primary tool for proving both the above inequalities.  We give an application of Theorem \ref{Comparable} to bounding $I^L(\gamma)$ by a function of $d_{\WP}(\gamma,0)$ in Corollary \ref{Thm: boundedcor}, and then proceed to sketch the proof of Theorem \ref{Comparable} in \S\ref{Sec:IntroProofStrategyForComparable}.  In the course of the sketch, we show in Corollary \ref{Thm:LoewnerIP} a new formula for the Loewner energy in terms of $H^\half$-inner products.

Sections \ref{Sec:IntroW} and \ref{Sec:IntroNPS} then discuss two out-workings of our methods that may be of independent interest.  Section \ref{Sec:IntroW} gives an overview of properties of our new welding energy $W_h(y)$, and introduces several M\"obius-invariant variations of $W_h(y)$.  Section \ref{Sec:IntroNPS} discusses our ``normalized pre-Schwarzian'' operator, and applies it in Theorems \ref{Thm:BoundaryLE} and \ref{Thm:InteriorLE} to give two generalized formulas for the Loewner energy.

\subsection{Proof approach for main inequalities: a rooted welding energy, and a corollary}\label{Sec:IntroApproachWeldingEnergy}
Our approach to bridging the gap between Loewner energy and conformal welding is to define a rooted functional $W_h(y)$ on conformal weldings $h$, and then prove a comparison inequality between $W_h(y)$ and $I^L(\gamma)$.  We proceed to introduce $W_h(y)$, state our main comparison inequality Theorem \ref{Comparable}, and give a corollary.

The genesis of our welding functional lies in the work of Shen, who first characterized the Weil--Petersson class $T_0(1)$  in terms of conformal welding.
\begin{prop}[\cite{ShenWP}]\label{IntroProp:ShenWeldingD}
    An element $h \in \ophom{\SSS}$ belongs to the Weil--Petersson class if and only if  $h$ is absolutely continuous with respect to the arclength measure and $\log h' \in H^{\half}(\SSS)$.
\end{prop}
\noindent Our preferred disk in $\wC$ will be the upper half-plane $\H$, for which Shen and his coauthors similarly showed:
\begin{prop}[{\cite{WPII, STWrealline}}]\label{IntroProp:ShenWeldingH}
  An increasing homeomorphism $h$ of $\R$ onto itself is in the Weil--Petersson class if and only if $h$ is locally absolutely continuous with $\log h' \in H^{\half}(\R)$.
\end{prop}
\noindent Note that we may describe such an $h$ as an orientation-preserving homeomorphism of the extended real line $\wR$ that fixes $\infty$.  That is, $h \in \ophom{\wR}$ with $h(\infty)=\infty$.  Based on these results, a natural functional on weldings $h \in \ophom{\wR}$ appears to be
\begin{align*}
    h \mapsto \| \log h' \|_{H^{1/2}(\R)} := \bigg(\frac{1}{4 \pi^2}\int_{-\infty}^\infty \int_{-\infty}^\infty \frac{\left|\log h'(x)-\log h'(y)\right|^2}{\left| x-y\right|^2} \dd x\,\dd y \bigg)^{\half}.
\end{align*}
A connection with $I^L(\cdot)$ begins to crystallize when we re-write \eqref{Eq:LEH} as
\begin{align}\label{Eq:LEHNabla}
    I^L(\gamma) = \frac{1}{\pi} \int_\H \big| \nabla \log|f'|  \big|^2 dA + \frac{1}{\pi} \int_{\H^*} \big| \nabla \log|g'|  \big|^2 dA,
\end{align}
and recall that the classical Douglas formula says the trace map $F \mapsto F|_{\R}$ behaves particularly nice with respect to the Dirichlet energy $\calD_\H$ and $H^\half$ semi-norm, namely
\begin{align}\label{Eq:IntroDouglasFormula}
    \calD_\H(F) := \frac{1}{2\pi}\| \nabla F\|^2_{L^2(\H)} = \|F|_{\R}\|^2_{H^{1/2}(\R)}.
\end{align}
Thus \eqref{Eq:LEHNabla} becomes
\begin{align}
    I^L(\gamma) &= 2\calD_\H(\nabla \log|f'|) +2\calD_{\H^*}(\nabla \log|g'|) \notag\\
    &= 2\big\|\log |f'|\big\|^2_{H^{1/2}(\R)} + 2\big\|\log |g'| \big\|^2_{H^{1/2}(\R)}. \label{Eq:LERHHalf}
\end{align}
While the welding $h$ does not yet immediately appear, this suggests the natural welding functional may actually be $h \mapsto \| \log h' \|_{H^{1/2}(\R)}^2$.  In our attempts greater symmetry seemed necessary, and we define
\begin{align}\label{Eq:IntroWhInfty}
    W_h(\infty) := \big\| \log |h'|  \big\|_{H^{1/2}(\R)}^2 + \big\| \log |(h^{-1})'| \big\|_{H^{1/2}(\R)}^2.
\end{align}
The presence of both $h$ and $h^{-1}$ is reminiscent of the fact that norms from both sides of the curve appear in \eqref{Eq:LEHNabla} and \eqref{Eq:LERHHalf}.   

The ``$\infty$'' in the notation for $W_h$ merits comment.  Following Proposition \ref{IntroProp:ShenWeldingH} above, we are still assuming $h$ fixes $\infty$.  However, we may desire to handle other \Mob re-normalizations of $h$, motivated by the conformal invariance of the Loewner energy.  For this reason, for any $y \in \R$, we define the welding energy rooted at $y$ as
\begin{align*}
    W_h(y) := \bigg\| \log\bigg| \frac{(h(x)-h(y))^2}{h'(x)(x -y)^2} \bigg| \bigg\|^2_{H^{1/2}(\R,dx)} + \bigg\| \log\bigg| \frac{(h^{-1}(x)-h^{-1}(y))^2}{(h^{-1})'(x)(x -y)^2} \bigg| \bigg\|^2_{H^{1/2}(\R,dx)}
\end{align*}
provided $h(y)$ is finite.  The energy is also defined for when one or both of $y$ and $h(y)$ is infinite, for instance becoming \eqref{Eq:IntroWhInfty} in the latter case.  We motivate these expressions in \S\ref{Sec:IntroDiscussion} and \S\ref{Sec:welding energy} below.  We use the same expression for weldings defined on other circles $C$ in the extended complex plane $\hatC$ (that is, circles or bi-infinite lines in $\C$), taking the $H^\half$-norm on $C$ instead of on $\R$.

The rooted welding energy thus defined, we find the following useful \Mob covariance, which holds for $h:C \rightarrow C$ defined on any circle $C \subset \hatC$.

\begin{thm}\label{Thm:WMobiusInvariance}
    Let $C_j$ be circles in $\hatC$, $h \in \ophom{C_2}$, and $S,T \in \PSL_2(\C)$ such that $T:C_1 \rightarrow C_2$. Then for any $y \in C_1$,
    \begin{align}\label{Eq:WMobiusInvariance}
        W_{S\circ h \circ T}(y) = W_{h}(T(y)).
    \end{align}
\end{thm}
\noindent We illustrate Theorem \ref{Thm:WMobiusInvariance} numerically in Examples \ref{Eg:WNotConstant} and \ref{Eg:WConjugatedtoS1}.  This interaction with \Mob transforms, combined with an approximation argument, as well as careful use of the Douglas formula (see \S\ref{Sec:IntroProofStrategyForComparable} below for more details), yields our main result for how $I^L(\cdot)$ relates to $W_h(y)$.
\begin{thm}\label{Comparable}
    Let $\gamma$ be a $K$-quasicircle and $h:\wR \rightarrow \wR$ any associated conformal welding.  Then for any $y \in \wR$,
\begin{align}\label{Ineq:Main}
        \frac{1}{2}\Big(3+\frac{1}{K^2+K^{-2}}\Big)I^L(\gamma) \leq W_h(y) \leq \frac{1}{2}(3+K^2+K^{-2})\,I^L(\gamma).
\end{align}
In particular, $\frac{3}{2} I^L(\gamma) \leq W_h(y)$ for all quasicircles $\gamma$ and any $y \in \wR$.
\end{thm}
\noindent We emphasize the lack of normalization requirement on $h$; the same comparison holds for any welding of $\gamma$. Of course, this ought to be plausible in light of Theorem \ref{Thm:WMobiusInvariance}.  Note also the coefficient of $I^L(\gamma)$ on the left-hand side of \eqref{Ineq:Main} is monotonically decreasing in $K$, while that on the right-hand side is monotonically increasing in $K$.  Hence $W_h(y) \asymp I^L(\gamma)$ on each $\Gamma_K$, with $\Gamma_K$ as in \eqref{Eq:Gamma_K}.

In Example \ref{Eg:WNotConstant} we give an explicit numerical example of the universal lower bound of Theorem \ref{Comparable}.

Combined with Lemma \ref{Lemma:WGeneralComp}, which bounds the welding energy $W_{h_1\circ h_2}(y)$ of a composition in terms of $W_{h_1}$ and $W_{h_2}$, Theorem \ref{Comparable} is our main tool for proving Theorems \ref{Thm:GeneralCompBound} and \ref{Thm:LEGrowth}.  In fact, with these tools in place, the proof of Theorem \ref{Thm:GeneralCompBound} is almost one line, while the argument for Theorem \ref{Thm:LEGrowth} is only slightly longer.  See \S\ref{Sec:ProofOfSubAdditivity} and \S\ref{Sec:ProofOfEntropy}, respectively. 

By leveraging the universal lower bound in \eqref{Ineq:Main} and invoking an estimate from \cite{ParametrizationpWP}, we also obtain the following.

\begin{cor}\label{Thm: boundedcor}
    For a Jordan curve $\g \subset \wC$ of finite Loewner energy with conformal welding $h: \hatR \rightarrow \hatR$ that fixes $0,1$ and $\infty$, $I^{L}(\g)$ is bounded above by a constant only depending on $d_{\textsc{WP}}(h, id)$.
\end{cor}
\noindent This corollary is a counterpart to the more explicit bound 
\begin{align*}
    I^{L}(\gamma) \geq \frac{c}{\pi}(d_{\textsc{WP}}(h, id)-Kc)
\end{align*}
already known for the other direction \cite[Thm. 6.3]{UniversalLiouvilleaction}, which holds for some $0<\delta<1$, with $K=\sqrt{2}/(1-\delta)^2$, and $c$ any value satisfying $0<c<2\delta\sqrt{4\pi/3}$.

Combined, these two bounds are reminiscent of work of  Brock~\cite{Brock03WPmetricandvolume}. Let $Q(X, Y)$ be a quasi-Fuchsian manifold with conformal structures $X$ and $Y$ on the boundary at infinity. Brock showed the Weil--Petersson distance $d_{\textsc{WP}}(X,Y)$ between $X$ and $Y$ is quasi-comparable to the renormalized volume $V(C)$ of the convex core $C$ of $Q(X,Y)$, in the sense that
\begin{align*}
    \frac{1}{C_1} d_{\textsc{WP}}(X,Y)-C_2 \leq V(C) \leq C_1 d_{\textsc{WP}}(X,Y) + C_2
\end{align*}
for some constants $C_1$ and $C_2$. Inspired by this, we might expect a more explicit version of Corollary~\ref{Thm: boundedcor} to take the form $I^L(\gamma) \leq C_1 d_{\textsc{WP}}(h, id) +C_2$.


\subsubsection{Proof strategy for Theorem \ref{Comparable}}\label{Sec:IntroProofStrategyForComparable}

As Theorem \ref{Comparable} is our main tool for proving Theorems \ref{Thm:GeneralCompBound} and \ref{Thm:LEGrowth}, we proceed to sketch the idea of its proof, noting some consequences of our approach.

Recall that the Douglas formula turns Wang's Loewner energy formula \eqref{Eq:LEHNabla} into \eqref{Eq:LERHHalf}, which gave impetus for our definition of $W_h(\infty)$ in \eqref{Eq:IntroWhInfty}.  Na\"ively taking the derivative of the composition $h = g^{-1} \circ f$ in \eqref{Eq:IntroWhInfty} and expanding the norms using the $H^\half$-inner product, we find
\begin{align}
    W_h(\infty) &= \big\|\log |(g^{-1})'\circ f| + \log|f'|\big\|_{H^{1/2}(\R)}^{2} + \big\|\log |(f^{-1})'\circ g| + \log |g'|\big\|_{H^{1/2}(\R)}^{2} \notag\\
    &= \frac{1}{2}I^L(\gamma) + \|\log |(g^{-1})'\circ f|\|_{H^{1/2}(\R)}^2 + \|\log |(f^{-1})'\circ g| \|_{H^{1/2}(\R)}^{2} \label{Eq:IntroExpandWInfty}\\
    &\hspace{13.5mm} +2 \brac{\log |(g^{-1})'\circ f|, \log|f'|}_{H^{1/2}(\R)} + 2 \brac{\log |(f^{-1})'\circ g|, \log|g'|}_{H^{1/2}(\R)}\notag\\
    &= \frac{1}{2}I^L(\gamma) + \|\log |g'\circ h|\|_{H^{1/2}(\R)}^2 + \|\log |f'\circ h^{-1}| \|_{H^{1/2}(\R)}^{2} \label{Eq:IntroWInftyDealWithThis1}\\
    &\hspace{13.5mm} -2 \brac{\log |g' \circ h|, \log|f'|}_{H^{1/2}(\R)} - 2 \brac{\log |f'\circ h^{-1}|, \log|g'|}_{H^{1/2}(\R)},\label{Eq:IntroWInftyDealWithThis2}
\end{align}
where we have used \eqref{Eq:LERHHalf} to obtain the $I^L(\gamma)$ term in \eqref{Eq:IntroExpandWInfty}.  To proceed, we need to say something about the norm terms in \eqref{Eq:IntroWInftyDealWithThis1} as well as the inner products in \eqref{Eq:IntroWInftyDealWithThis2}.  We handle the former through the following Nag--Sullivan result on the norm of the pull-back operator $\calP_h(\varphi):= \varphi \circ h$ acting on $H^\half$.

\begin{prop}[\cite{Nag-SullivanTeichmuller}] \label{const.} 
   $\calP_h$ is a bounded operator on $H^{\half}(\R)$ if and only if $h$ is a quasisymmetric homeomorphism of $\R$. Moreover, if h extends to a K-quasiconformal homeomorphism of $\mathbb{H}$, then the operator norm of $\calP_h$ satisfies $\|\mathcal{P}_h\| \leq \sqrt{K+K^{-1}}$.
\end{prop}
\noindent Combined with the fact that $h$ is the conformal welding of a $K$-quasicircle, and thus optimally has a $K^2$-quasiconformal extension to $\H$ (Lemma \ref{Lemma:K^2Optimal}), we can thereby bound the two norm terms in \eqref{Eq:IntroWInftyDealWithThis1} by
$(K^2 +K^{-2})I^L(\gamma)/2$ (again using \eqref{Eq:LERHHalf}), and thus have
\begin{align}\label{Eq:IntroWInftyDealWithThis3}
\begin{split}
    W_h(\infty) \leq \frac{1}{2}(1 &+ K^2 +K^{-2})I^L(\gamma)\\
    &-2 \brac{\log |g' \circ h|, \log|f'|}_{H^{1/2}(\R)} - 2 \brac{\log |f'\circ h^{-1}|, \log|g'|}_{H^{1/2}(\R)}.
\end{split}
\end{align}
The following lemma, another key step in our argument, handles the remaining cross terms.

\begin{lemma}\label{lem:generalcrossterm}
Let $\g \subset \hat{\mathbb{C}}$ be a Jordan curve of finite Loewner energy which passes through $\infty$, with $f: \mathbb{H} \rightarrow H$ and $g:\mathbb{H}^* \rightarrow H^*$ Riemann maps to either side of $\gamma$ which fix $\infty$, and $h := g^{-1} \circ f$ the corresponding conformal welding on $\mathbb{R}$. Then
\begin{align}\label{Eq:CrossTermGeneral}
\begin{split}
    \brac{\log |g'\circ h|,\log |f'|}_{H^{1/2}(\R)}+&\brac{\log |f'\circ h^{-1}|,\log |g'|}_{H^{1/2}(\R)}\\
    &\hspace{20mm}= - \brac{\log|f'|, \log|h'|}_{H^{1/2}(\R)} \\
    &\hspace{20mm}=-\langle\log|g'|, \log|(h^{-1})'|\rangle_{H^{1/2}(\R)}.
\end{split}
\end{align}
Furthermore, the inner products on the left side satisfy
\begin{align}
    \brac{\log |g'\circ h|,\log |f'|}_{H^{1/2}(\R)} &= -\frac{1}{2\pi}\| \nabla \log|g'| \|^2_{L^2(\H^*)}, \; \quad \text{and}\label{Eq:IPDEgGeneral}\\
    \brac{\log |f'\circ h^{-1}|,\log |g'|}_{H^{1/2}(\R)} &= -\frac{1}{2\pi}\| \nabla \log|f'| \|^2_{L^2(\H)}.\label{Eq:IPDEfGeneral}
\end{align}
\end{lemma}
\noindent Before proceeding, we note this immediately yields two novel expressions for the Loewner energy.
\begin{cor}\label{Thm:LoewnerIP}
    Given the assumptions and notation of Lemma \ref{lem:generalcrossterm},
     \begin{align}\label{Eq:LoewnerIP}
          I^L(\gamma) = 2\brac{\log|f'|, \log h'}_{H^{1/2}(\R)} = 2\langle \log|g'|, \log (h^{-1})'\rangle_{H^{1/2}(\R)}.
     \end{align}
\end{cor}
\begin{proof}
    Recall \eqref{Eq:LEHNabla} and substitute \eqref{Eq:IPDEgGeneral} and \eqref{Eq:IPDEfGeneral} into \eqref{Eq:CrossTermGeneral}. 
\end{proof}
\noindent We also come up with three other formulas for the Loewner energy; see Theorems \ref{Thm:BoundaryLE} and \ref{Thm:InteriorLE} in \S\ref{Sec:IntroNPS} below, as well as Lemma \ref{Lemma:LEWeldingFactorization}.  This last result expresses the Loewner energy in terms of a factorization $h = H_g^{-1} \circ H_f$ of the welding $h$ by two increasing homeomorphisms $H_f$ and $H_g$ of $\R$, and thus is related to Question 4.1 of \cite{Yilin2024two}.

Returning to the argument for Theorem \ref{Comparable}, by \eqref{Eq:CrossTermGeneral} and Corollary \ref{Thm:LoewnerIP} we see the cross terms in \eqref{Eq:IntroWInftyDealWithThis3} sum to $I^L(\gamma)$, which yields the claimed upper bound in \eqref{Ineq:Main}.  The argument for the lower bound is similar.  

As we expend non-trivial effort on Lemma \ref{lem:generalcrossterm}, we also say a word regarding its proof.  We begin in \S\ref{Sec:MainProof} by using the Douglas formula and the divergence theorem (in the form of Green's first identity) to prove the result for analytic $\gamma$ passing through $\infty$.  We then approximate a general finite-energy $\gamma$ by the images $\gamma_n = f(C_n)$ under $f$ of larger and larger circles $C_n \subset \H$.  Using Proposition \ref{Thm: equiv convergence H} on modes of convergence in the Weil--Petersson \Teich space WP$(\R)$, combined with some observations on the ``normalized pre-Schwarzian'' derivative that we discuss below in \S\ref{Sec:IntroNPS}, we show that $\gamma_n \rightarrow \gamma$ in WP$(\R)$, and that this implies the inner products and norms also converge.  See Lemma \ref{lem: curvesWPconvergence} and following for these details.

\subsection{By-product I: welding energies}\label{Sec:IntroW}

While we utilize the welding energy $W_h(y)$ as a tool to obtain Theorems \ref{Thm:GeneralCompBound} and \ref{Thm:LEGrowth},   
Theorems \ref{Thm:WMobiusInvariance} and \ref{Comparable} suggest it may be an intriguing object in its own right.  We therefore initiate an investigation of $W_h(y)$ by proving several basic properties, as we proceed to summarize.  

At the outset, we note from \eqref{Ineq:Main} that $\gamma$ has finite Loewner energy if and only if $W_h(y) < \infty$ for some $y$, if and only if $W_h(y) < \infty$ for all $y$.  
\begin{thm}\label{Thm:WFiniteEquiv}
    Let $C$ be a circle in $\hatC$ and $h \in \ophom{C}$ a conformal welding for a Jordan curve $\gamma \subset \hatC$.  The following are equivalent:
    \begin{enumerate}[$(a)$]
        \item $I^L(\gamma)<\infty$.
        \item \label{Thm:WFiniteEquivOneY} There exists $y \in C$ such that $W_h(y) < \infty$.
        \item \label{Thm:WFiniteEquivEveryY}For every $y \in C$, $W_h(y) < \infty$.
    \end{enumerate}
\end{thm}
\noindent While this immediately follows from Theorem \ref{Comparable}, we prove Theorem \ref{Thm:WFiniteEquiv} in \S\ref{Sec:welding energy} before we prove the former result, and thus give a different argument.  

Write $\WP(C)$ for the collection of all weldings $h:C\rightarrow C$ of Jordan curves $\gamma$ of finite energy (irrespective of normalization), and let $\PSL_2(\C) \supset \MOB(C) \simeq \PSL_2(\R)$ be the collection of \Mob transformations which preserve $C$, with $\MOB(C) \supset \MOB(C,y) \simeq \H$ the subgroup fixing a given $y \in C$.  Combining Theorems \ref{Thm:WMobiusInvariance} and \ref{Thm:WFiniteEquiv} yields that $W_h(y)$ is well-defined and finite on the double quotient
\begin{align}\label{Eq:WOnQuotientSpace}
    \MOB(C)\bs \!\WP(C) / \MOB(C,y).
\end{align}

The welding energy also has the following parallel with the Loewner energy.
\begin{thm}\label{Thm:WVanishesMobius}
    Let $C$ be a circle in $\hatC$ and $h \in \ophom{C}$.  There exists $y \in C$ such that $W_h(y)=0$ if and only if $h$ is a \Mob transformation.  In this case, $W_h(y) = 0$ for all $y \in C$.
\end{thm}
\noindent These results might even lead one to wonder whether the welding energy $W_h(y)$ \emph{is} the Loewner energy, or perhaps a multiple of it.  The question comes down to the nature of the function $y \mapsto W_h(y)$.  If the welding energy were the Loewner energy (modulo constant multiple, say), this would be a constant function, as the energy would be root independent.  We show in \S\ref{Sec:WhyNotRootInvariant} by explicit numerical example, however, that $y \mapsto W_h(y)$ is not always constant.  Thus the welding energy, while in some sense a close kin of the Loewner energy, is also truly a distinct object.  

Despite $y \mapsto W_h(y)$ not generally being constant, we do show it is continuous.

\begin{thm}\label{Thm: WisContinuousatRoot}
    For $C \subset \hatC$ a circle and $h \in \ophom{C}$, $y \mapsto W_h(y)$ is continuous on $C$.

\end{thm}
\noindent This result is not entirely trivial because $W_h(y)$ has differing expressions depending on whether $y$ and/or $h(y)$ is infinite.  Theorem \ref{Thm:WVanishesMobius} says continuity in the root always holds, irrespective of the different cases.

\subsubsection{M\"obius-invariant  welding energies}\label{Sec:IntroMobiusInvariantEnergies}

The continuity of Theorem \ref{Thm: WisContinuousatRoot} enables us to define three variants of $W_h(y)$ that are completely \Mob invariant.  

\begin{definition}\label{Def:UpperLowerW}
    Let $C$ be a circle in $\hatC$ and $h \in \ophom{C}$. The \emph{upper and lower welding energies of $h$} are, respectively,
    \begin{align*}
        \overline{W_h} := \max_{y \in C} W_h(y) \qquad \text{ and } \qquad \underline{W_h} := \min_{y \in C} W_h(y).
    \end{align*}
    When $W_h(y) < \infty$, the  \emph{welding energy gap} is 
    \begin{align*}
        \Delta W_h := \overline{W_h} - \underline{W_h}.
    \end{align*}
\end{definition}

\noindent The above properties of $W_h(\cdot)$ immediately yield the following for $\overline{W_h},  \underline{W_h}$, and $\Delta W_h$.

\begin{cor}\label{Cor:UpperLowerWProps}
    For $C$ a circle in $\hatC$ and $h \in \ophom{C}$, the following hold:
    \begin{enumerate}[$(i)$]
        \item \label{Cor:UpperLowerWPropsMobInv} $\overline{W_h}$, $\underline{W_h}$, and $\Delta W_h$ are all invariant under \Mob renormalization.  That is, for any $S,T \in \PSL_2(\C)$, 
    \begin{align*}
        \overline{W_{S\circ h \circ T}} = \overline{W_{h}} \quad \text{ and } \quad \underline{W_{S\circ h \circ T}} = \underline{W_h},
    \end{align*}
    and hence $\Delta W_{S\circ h \circ T} = \Delta W_h$ as well.
        \item\label{Cor:UpperLowerWPropsFinite} $\overline{W_h}<\infty$ if and only if $\underline{W_h} < \infty$ if and only if $h$ is the welding of a Jordan curve $\gamma$ of finite Loewner energy.
        \item\label{Cor:UpperLowerWPropsCircle} If $h$ is a welding for a circle $\gamma \subset \hatC$, then $\overline{W_h} = 0$ (and hence $\Delta W_h=0$).
    \end{enumerate}
\end{cor}

\begin{proof}
    Property $(\ref{Cor:UpperLowerWPropsMobInv})$ follows immediately from Definition \ref{Def:UpperLowerW} and Theorem \ref{Thm:WMobiusInvariance}, while $(\ref{Cor:UpperLowerWPropsFinite})$ follows immediately from Theorem \ref{Thm:WFiniteEquiv}.  Theorem \ref{Thm:WVanishesMobius} immediately yields $(\ref{Cor:UpperLowerWPropsCircle})$.
\end{proof}

The welding energy gap $\Delta W_h$ may be the most intriguing of the trio introduced in Definition \ref{Def:UpperLowerW}.  Theorem \ref{Comparable} yields the following.
\begin{cor}\label{Cor:EnergyGapBound}
    For a conformal welding $h$ of a Jordan curve $\gamma \subset \hatC$ of finite Loewner energy that is also a $K$-quasicircle, for $\tilde{K} := K^2 + K^{-2}$ we have
    \begin{align*}
        \Delta W_h \leq \frac{1}{2}\Big( \tilde{K} - \frac{1}{\tilde{K}}\Big) I^L(\gamma).
\end{align*}
\end{cor}
\noindent In particular, we recover the fact in Corollary \ref{Cor:UpperLowerWProps}$(\ref{Cor:UpperLowerWPropsCircle})$ that $\Delta W_h$ vanishes for weldings $h$ of the circle, i.e. \Mob transformations.

\subsection{By-product II: A ``normalized pre-Schwarzian derivative," and two generalized Loewner-energy formulas}\label{Sec:IntroNPS}

\subsubsection{The normalized pre-Schwarzian derivative}
Recall that to prove Lemma \ref{lem:generalcrossterm} we approximate a general finite-energy $\gamma$ with analytic equipotentials $\gamma_n = f(C_n)$, where $C_n$ are large circles in $\H$.  In the process of proving $\gamma_n \rightarrow \gamma$ in $T_0(1)$, we discovered the usefulness of the differential operator
\begin{align}\label{Eq:IntroNPSDef}
    \calB_f(z,w) := \frac{f'(z)}{f(z)-f(w)} - \frac{1}{z-w} - \frac{1}{2} \frac{f''(z)}{f'(z)} = \frac{1}{2}\partial_z \log\bigg( \frac{(f(z)-f(w))^2}{f'(z)(z-w)^2} \bigg),
\end{align}
which we call the \emph{normalized pre-Schwarzian derivative}.  Despite its potentially-cumbersome appearance, $\calB_f$ turns out to have very elegant properties, so much so that it appears to be the appropriate 1-form analog to the Schwarzian derivative 
\begin{align}\label{Def:Schwarzian}
    \calS_f := \frac{f'''}{f'} - \frac{3}{2} \Big(\frac{f''}{f'}\Big)^2,
\end{align}
rather than the classical pre-Schwarzian 
\begin{align}\label{Def:Preschwarzian}
    \calA_f := \frac{f''}{f'}.
\end{align}
Table \ref{Table:ABS}, for instance, shows four respects in which $\calB_f$ has behavior entirely parallel to $\calS_f$, whereas in each instance $\calA_f$ does not.

\subsubsection{Boundary-normalized Loewner energy}
In particular, the first two rows of Table \ref{Table:ABS} yield the \Mob covariance 
\begin{align*}
         \calB_{S\circ f\circ T}(z,w) =\calB_{f}\big(T(z),T(w)\big)T'(z)
\end{align*} 
for $S,T \in \PSL_2(\C)$, which is very useful for $L^2$ integrals.  This leads to a formula for the Loewner energy that holds for any placement of $\gamma$ on $\hatC$, and any normalization of conformal maps that have domains that are the two sides of any disk on the sphere. A disk, as usual, is a planar ball $B_r(w) = \{\,z\,:\, d(z,w)<r\,\}$, where $d$ is either the Euclidean or spherical metric (in particular, half-planes in $\C$ are disks in $\hatC$).  For $\Omega \subsetneq \hatC$, set $\Omega^* := \hatC \bs \overline{\Omega}$.  

\begin{table}
\begin{align*}
\begin{array}{|c|c|c|c|}
    \hline
     & \calA & \calB & \mathcal{S}\\
     \hline
     f \mapsto f\circ T &  \calA_f(T(z))T'(z) + \calA_T(z) & \begin{array}{@{}c@{}} \calB_f(T(z),T(w))T'(z) \\  \end{array} &  \calS_f(T(z))(T'(z))^2\\
     \hline
     f \mapsto T\circ f & \calA_T(f(z))f'(z) + \calA_f(z) & \begin{array}{@{}c@{}} \calB_f(z,w) \\  \end{array} & \calS_f(z)\\
     \hline
     f \in \mathcal{M}(\Omega) & \calA_f  \in \mathcal{M}(\Omega) & \begin{array}{@{}c@{}} \calB_f(\cdot,w) \in \calH(\Omega) \\  \end{array} & \calS_f \in \calH(\Omega)\\
     \hline
     \text{Vanishes} & \text{Iff $f$ affine} & \begin{array}{@{}c@{}} \text{Iff $f$ \Mob} \\  \end{array} & \text{Iff $f$ \Mob}\\
     \hline
\end{array}
\end{align*}
\caption{\small
Four ways in which, in contrast to the pre-Schwarzian $\calA$, the normalized pre-Schwarzian $\calB$ of \eqref{Eq:IntroNPSDef} behaves analogously to the Schwarzian $\calS$.  Here $T \in \PSL(\C)$ is a \Mob transformation, and $\mathcal{M}(\Omega)$ and $\mathcal{H}(\Omega)$ are the collections of meromorphic and holomorphic functions on a domain $\Omega$, respectively. The top row of the table, for instance, says $\calA_{f \circ T}(z) = \calA_f(T(z))T'(z) + \calA_T(z)$, whereas $\calB_{f \circ T}(z,w) = \calB_f(T(z),T(w))T'(z)$, a 1-form analogue of the Schwarzian composition rule $\calS_{f \circ T}(z) = \calS_f(T(z))(T'(z))^2$.  See Lemma \ref{Lemma:NPSMobiusInvariance1} for this and the post-composition invariance of $\calB$, and Corollary \ref{Lemma:BFinite} and Lemma \ref{Lemma:BMobiusZero} for the properties in the third and fourth rows, respectively.  See Table \ref{Table:NPSProps} for other basic properties of $\calB$.}
\label{Table:ABS}
\end{table}

\begin{thm}[Boundary-normalized Loewner energy]\label{Thm:BoundaryLE}
    Let $\gamma \subset \widehat{\C}$ be any Jordan curve, $D \subsetneq \widehat{\C}$ any disk, and $f:D \rightarrow \Omega$ and $g: D^* \rightarrow \Omega^*$ any conformal maps to the two complementary components of $\widehat{\C}\bs \gamma = \Omega \cup \Omega^*$.  Then for any $w_0 \in \partial D$ with $w_1 := g^{-1}\circ f(w_0)$,
    \begin{align}\label{Eq:IntroBoundaryLE}
      I^L(\gamma) = &\frac{4}{\pi}\int_D |\calB_f(z,w_0)|^2 dA(z)+  \frac{4}{\pi}\int_{D^*} |\calB_g(z,w_1)|^2 dA(z).
    \end{align}
\end{thm}
\noindent Note that since $\partial D$ is the boundary of $D$ with respect to $\widehat{\C}$, $w_0$ or $w_1$ (or both) may be $\infty$.  Wang's formula \eqref{Eq:LEH} is the special case of $D = \H$, $f$ and $g$ normalized to both fix $\infty$, and $w_0 = w_1 = \infty$.\footnote{While our formula \eqref{Eq:IntroNPSDef} for $\calB_f$ assumes both $w$ and $f(w)$ are elements of $\C$, the normalized pre-Schwarzian is also defined when one or both is $\infty$; the corresponding term in the sum simply vanishes.  For instance, $\calB_f(z,\infty) = - \frac{1}{2} \frac{f''(z)}{f'(z)}$ if $w = f(w) = \infty$.  See \S\ref{Sec:NPS} for more details.}  The fresh insight provided by the theorem is how the integrand changes if we normalize at points $w_0 \in \R$ instead, and/or no longer require $f$ and/or $g$ to fix $\infty$.  For instance, in the same setting of $D = \H$ but when $f(\infty) = g(\infty) \in \C$, \eqref{Eq:IntroBoundaryLE} says
\begin{align*}
    I^L(\gamma) = &\frac{4}{\pi}\int_\H |\calB_f(z,\infty)|^2 dA(z)+  \frac{4}{\pi}\int_{\H^*} |\calB_g(z,\infty)|^2 dA(z)\\
    &= \frac{4}{\pi}\int_\H \bigg|\frac{f'(z)}{f(z)- f(\infty)} - \frac{1}{2}\frac{f''(z)}{f'(z)} \bigg|^2 dA(z)+  \frac{4}{\pi}\int_{\H^*} \bigg|\frac{g'(z)}{g(z)- g(\infty)} - \frac{1}{2}\frac{g''(z)}{g'(z)} \bigg|^2 dA(z).
\end{align*}

As alluded to above, the $\calB_f$ operator will also help us prove that certain analytic equipotentials $\gamma_n$ converge to a general Weil--Petersson curve $\gamma$ in the Weil--Petersson \Teich space.  Indeed, Theorem \ref{Thm:BoundaryLE} will give us an efficient approach to control several integrals by Loewner energies.  See the proof of Lemma \ref{lem: curvesWPconvergence} for these details.

While the identity \eqref{Eq:LEH} is an instance of \eqref{Eq:IntroBoundaryLE} with $D= \H$, what of Wang's other formula \eqref{Eq:LEDisk} for the disk?  The presence of the interior points $0 \in \D$ and $\infty \in \D^*$ in \eqref{Eq:LEDisk} suggests it may not neatly fall into the framework of Theorem \ref{Thm:BoundaryLE}. And indeed, a distinct formula handles this situation.   

\subsubsection{Interior-normalized Loewner energy}
In contrast to the boundary case of Theorem \ref{Thm:BoundaryLE}, normalizing in the interior gives us two $\C$-degrees of freedom $u \in D$ and $v \in D^*$.  We also need the Schwarz-reflection $u^* \in D^*$ of $u$ across $\partial D$, as well as two additional differential operators, namely
\begin{align}
    \calB_{f,g}^*(z,u,v) :=& \frac{f'(z)}{f(z)-g(v)} - \frac{1}{z-u^*} - \frac{1}{2}\frac{f''(z)}{f'(z)} 
    \label{Eq:IntroB*}
\end{align}
and
\begin{align}
    \calC_{f,g}(z,u,v) :=& \calB_f(z,u) - \calB^*_{f,g}(z,u,v)\notag \\
    =& \frac{f'(z)}{f(z)-f(u)} - \frac{f'(z)}{f(z)-g(v)} - \bigg(\frac{1}{z-u} - \frac{1}{z-u^*}  \bigg). \label{Intro:C}
\end{align}
Given the presence of the reflected point $u^*$, note that each of $\calB^*$ and $\calC$ implicitly depends on the ambient domain $D$, and thus we could write $\calB_{f,g,D}^*(z,u,v)$ and $\calC_{f,g,D}(z,u,v)$ for additional clarity. Also, as in the case of the $\calB_f$ formula \eqref{Eq:IntroNPSDef}, the expressions in \eqref{Eq:IntroB*} and \eqref{Intro:C} assume all the quantities in play are elements of $\C$.  However, the operators remain defined when any of $u,v,u^*,f(u)$, and $g(v)$ is infinite, in which case the corresponding term vanishes (see Example \ref{Eg:InteriorLE} immediately below, for instance).  
\begin{thm}[Interior-normalized Loewner energy]\label{Thm:InteriorLE}
Let $\gamma \subset \widehat{\C}$ be any Jordan curve, $D \subsetneq \widehat{\C}$ any disk, and $f:D \rightarrow \Omega$ and $g: D^* \rightarrow \Omega^*$ any conformal maps to the two complementary components of $\widehat{\C}\bs \gamma = \Omega \cup \Omega^*$.  Then for any $u \in D$ and any $v \in D^*$,
\begin{align}\label{Eq:IntroInteriorLE1}
\begin{split}
    I^L(\gamma) = \frac{4}{\pi}\int_D | \calB_{f}(z,u) |^2&dA(z) + \frac{4}{\pi}\int_{D^*} | \calB^*_{g,f}(z,v,u) |^2dA(z)\\ &- \frac{2}{\pi}\int_D | \calC_{f,g}(z,u,v) |^2dA(z) - \frac{2}{\pi}\int_{D^*} | \calC_{g,f}(z,v,u) |^2dA(z).
\end{split}
\end{align}
\end{thm}
\noindent As in Theorem \ref{Thm:InteriorLE}, we emphasize the theorem's total absence of requirements on $\gamma$, the domain (beyond being a disk), the normalization of the conformal maps, and the choice of $u \in D$ and $v \in D^*$.  The lack of symmetry in the first two integrals in \eqref{Eq:IntroInteriorLE1} may be surprising, but observe that the roles of $D$ and $D^*$ are interchangeable.

\begin{example}\label{Eg:InteriorLE}
Take $D = \D$, assume $g(\infty) = \infty$, and normalize at $u = 0$ and $v = u^* = \infty$.  Recalling the corresponding terms in our formulas for $\calB$, $\calB^*$ and $\calC$ vanish when an $\infty$ appears, Theorem \ref{Thm:InteriorLE} yields
\begin{align*}
    I^L(\gamma) &= \frac{4}{\pi}\int_D | \calB_{f}(z,0) |^2dA(z) + \frac{4}{\pi}\int_{D^*} | \calB^*_{g,f}(z,\infty,0) |^2dA(z)\\
    &\hspace{28mm} - \frac{2}{\pi}\int_D | \calC_{f,g}(z,0,\infty) |^2dA(z) - \frac{2}{\pi}\int_{D^*} | \calC_{g,f}(z,\infty,0) |^2dA(z)\\
    &= \frac{4}{\pi}\int_\D \bigg| \frac{f'(z)}{f(z)-f(0)} - \frac{1}{z} - \frac{1}{2}\frac{f''(z)}{f'(z)} \bigg|^2dA + \frac{4}{\pi}\int_{\D^*} \bigg| \frac{g'(z)}{g(z)-f(0)} - \frac{1}{z} - \frac{1}{2} \frac{g''(z)}{g'(z)} \bigg|^2dA\\
    &\hspace{28mm}-\frac{2}{\pi}\int_\D \bigg| \frac{f'(z)}{f(z)-f(0)} - \frac{1}{z} \bigg|^2dA - \frac{2}{\pi}\int_{\D^*} \bigg| \frac{g'(z)}{g(z)-f(0)}-\frac{1}{z} \bigg|^2dA.
\end{align*}
Recall from the second row of Table \ref{Table:ABS} (Lemma \ref{Lemma:NPSMobiusInvariance1}) that $\calB$, like the Schwarzian derivative, is invariant under post-composition by \Mob transformations.  We prove the same for $\calB^*$ and $\calC$ in Lemma \ref{Lemma:NPS*MobiusInvariance}, and thus post composing the conformal maps by $z \mapsto z-f(0)$ yields
\begin{align*}
    I^L(\gamma) &= \frac{4}{\pi}\int_\D \bigg| \frac{f'(z)}{f(z)} - \frac{1}{z} - \frac{1}{2}\frac{f''(z)}{f'(z)} \bigg|^2dA + \frac{4}{\pi}\int_{\D^*} \bigg| \frac{g'(z)}{g(z)} - \frac{1}{z} - \frac{1}{2} \frac{g''(z)}{g'(z)} \bigg|^2dA\\
    &\hspace{28mm}-\frac{2}{\pi}\int_\D \bigg| \frac{f'(z)}{f(z)} - \frac{1}{z} \bigg|^2dA - \frac{2}{\pi}\int_{\D^*} \bigg| \frac{g'(z)}{g(z)}-\frac{1}{z} \bigg|^2dA\\
    &= \frac{4}{\pi}\int_\D \bigg| \frac{f'(z)}{f(z)} - \frac{1}{z} - \frac{1}{2}\frac{f''(z)}{f'(z)} \bigg|^2dA + \frac{4}{\pi}\int_{\D^*} \bigg| \frac{g'(z)}{g(z)} - \frac{1}{z} - \frac{1}{2} \frac{g''(z)}{g'(z)} \bigg|^2dA + 4\log \frac{|f'(0)|}{|g'(\infty)|},
\end{align*}
where the last equality follows from the generalized Grunsky equality (see \cite[Ch.II Rmk. 2.2]{TTbook}), and $g'(\infty) := \lim\limits_{z \rightarrow \infty}g'(z)$.  If we reverse the roles of $\D$ and $\D^*$, calling the latter $D$ and the former $D^*$ instead, we find
\begin{align*}
    I^L(\gamma) &= \frac{4}{\pi}\int_{\D} | \calB_{f,g}^*(z,0,\infty) |^2dA(z) + \frac{4}{\pi}\int_{\D^*} | \calB_{g}(z,\infty) |^2dA(z)\\
    &\hspace{28mm} - \frac{2}{\pi}\int_{\D} | \calC_{f,g}(z,0,\infty) |^2dA(z) - \frac{2}{\pi}\int_{\D^*} | \calC_{g,f}(z,\infty,0) |^2dA(z)\\
    &= \frac{1}{\pi}\int_\D \bigg|\frac{f''(z)}{f'(z)} \bigg|^2dA + \frac{1}{\pi}\int_{\D^*} \bigg|\frac{g''(z)}{g'(z)} \bigg|^2dA + 4\log \frac{|f'(0)|}{|g'(\infty)|},
\end{align*}
showing Wang's formula \eqref{Eq:LEDisk} is a special case of Theorem \ref{Thm:InteriorLE}, as \eqref{Eq:LEH} was of Theorem \ref{Thm:BoundaryLE}.  

In Corollary \ref{Cor:InteriorLELog} we generalize the identity
\begin{align}\label{Eq:GrunskyID}
    - \frac{2}{\pi}\int_{\D} | \calC_{f,g}(z,0,\infty) |^2dA(z) - \frac{2}{\pi}\int_{\D^*} | \calC_{g,f}(z,\infty,0) |^2dA(z) = 4\log \frac{|f'(0)|}{|g'(\infty)|}
\end{align}
to give a version of the interior-normalized formula \eqref{Eq:IntroInteriorLE1} that looks more like \eqref{Eq:LEDisk}.

\end{example}

\subsubsection{Discussion and criticism}\label{Sec:IntroNPSDiscussion}

We conclude \S\ref{Sec:IntroNPS} by distilling the main contributions of our results on the $\calB$, $\calB^*$, and $\calC$ operators, as well as by answering an objection.

\begin{enumerate}
    \item Theorems \ref{Thm:BoundaryLE} and \ref{Thm:InteriorLE} show that difference between Wang's two Loewner-energy formulas \eqref{Eq:LEDisk} and \eqref{Eq:LEH} lies not in the placement of $\gamma$ on $\hatC$, nor in the choice of domains, nor in anything regarding the conformal maps, but rather in the fact that \eqref{Eq:LEDisk} normalizes at the interior points $u=0$ and $v=\infty$, while \eqref{Eq:LEH} normalizes at the boundary point $w_0=\infty$.
    
    \quad For the boundary-normalized expression \eqref{Eq:LEH}, we could just as easily, for instance, take $D = \D$ and assume $\gamma \subset \C$, as in Wang's version \eqref{Eq:LEDisk} of the interior-normalized formula. Choosing $w_0 = i$, say, and supposing $f(i) = g(i)$, Theorem \ref{Thm:BoundaryLE} then says
    \begin{align*}
        I^L(\gamma) = &\frac{4}{\pi}\int_\D |\calB_f(z,i)|^2 dA(z)+  \frac{4}{\pi}\int_{\D^*} |\calB_g(z,i)|^2 dA(z)\\
        =& \frac{4}{\pi}\int_\D \bigg| \frac{f'(z)}{f(z)-f(i)} - \frac{1}{z-i} - \frac{1}{2} \frac{f''(z)}{f'(z)} \bigg|^2 dA(z)\\
        & \hspace{27mm} +  \frac{4}{\pi}\int_{\D^*} \bigg| \frac{g'(z)}{g(z)-g(i)} - \frac{1}{z-i} - \frac{1}{2} \frac{g''(z)}{g'(z)} \bigg|^2 dA(z).
    \end{align*}
    Wang's formula \eqref{Eq:LEH} is cleaner because $\calB_f(z,\infty)$ and $\calB_g(z, \infty)$ are very simple given $f(\infty) = \infty = g(\infty)$. 

    \item An objection to the value of Theorems \ref{Thm:BoundaryLE} and \ref{Thm:InteriorLE} is that their proofs are elementary: both arguments begin with Wang's formula (\eqref{Eq:LEH} and \eqref{Eq:LEDisk}, respectively) and simply apply the \Mob covariance of the $\calB$, $\calB^*$, and $\calC$ operators (Lemmas \ref{Lemma:NPSMobiusInvariance1} and \ref{Lemma:NPS*MobiusInvariance}, respectively).  While we concur that the proofs are indeed trivial, we suggest the value of these results is in defining the operators in such a way as to obtain the amenable \Mob interaction.  The fact that $\calB$ appears to have thus far eluded explicit definition undergirds the value of the formulation \eqref{Eq:IntroNPSDef}.

    \item Given the desirable behavior of $\calB$ vis-a-vis the Schwarzian $\calS$, as summarized in Table \ref{Table:ABS} (see also Table \ref{Table:NPSProps}), the normalized pre-Schwarzian may be of independent interest. 
\end{enumerate}

\subsection{Motivation and informal  discussion}\label{Sec:IntroDiscussion}


Bishop has occasionally described the philosophy behind his plethora of geometric characterizations of Weil--Petersson curves \cite{Bishop2019weil} as, paraphrasing, ``Find a reasonable definition of curvature, and show an $L^2$ version of it characterizes Weil--Petersson curves."  Our paper echoes this philosophy.  The genesis of the normalized pre-Schwarzian and our welding energy lies in the familiar fact that \Mob transformations $T$ preserve the cross-ratio, i.e.
\begin{align}\label{Eq:MobiusPreservesCR}
    \frac{T(z_1)-T(z_4)}{z_1-z_4} \cdot \frac{T(z_2)-T(z_3)}{z_2-z_3} \cdot \frac{z_1-z_3}{T(z_1)-T(z_3)} \cdot \frac{z_2-z_4}{T(z_2)-T(z_4)} = 1
\end{align}
whenever $z_j \in \C$ are distinct.  ``Fusing'' two pairs of points by sending $z_3 \rightarrow z_1=:x$ and $z_2\rightarrow z_4 =:y$ yields
\begin{align}\label{Eq:MobDifferenceQuotient}
    \frac{(T(x)-T(y))^2}{T'(x)T'(y)(x-y)^2 } = 1,
\end{align}
and thus $\log \frac{(T(x)-T(y))^2}{T'(x)T'(y) (x-y)^2 }$ pointwise vanishes (and, in fact, by integrating one sees this identity characterizes \Mob transformations).  The ``degree of non-vanishing'' of this expression thus gives a measurement of how non-\Mob a conformal map is, and we could consider this as an indirect definition of curvature of the resulting boundary curve. Our generalized Loewner energy formulas \eqref{Eq:IntroBoundaryLE} and \eqref{Eq:IntroInteriorLE1} loosely say that, for $f: D \rightarrow \Omega$, the boundary curve $f(\partial D) = \gamma$ is Weil--Petersson if and only if 
\begin{align*}
    \partial_z \log  \frac{(f(z)-f(w))^2}{f'(z)f'(w)(z-w)^2 }    =  \nabla_z \log \bigg| \frac{(f(z)-f(w))^2}{f'(z)f'(w)(z-w)^2 } \bigg|   \in L^2(D)
\end{align*}
for any $w$.  The Douglas formula then says this is equivalent to 
\begin{align*}
    \log \bigg| \frac{(f(\cdot)-f(w))^2}{f'(\cdot )f'(w)(\cdot -w)^2 } \bigg| \in H^{\half}(\partial D),
\end{align*}
which can be viewed as parallel (though not identical) to our Theorem \ref{Thm:WFiniteEquiv} characterizing Weil--Petersson curves via the welding energy $W_h(y)$.  Thus, at some level, these seventy pages examine the identity \eqref{Eq:MobDifferenceQuotient} from an appropriate $L^2$-vantage point.

\subsection{Organization}\label{Sec:Organization}
We begin with a thorough background discussion on Weil--Petersson \Teich space, Loewner energy, and other preliminaries in Section~\ref{Sec:Preliminaries}.  We also adapt some known results to our setting.  
In Section~\ref{Sec:welding energy} we construct our welding energy $W_h(y)$ through several steps and study its properties. Section~\ref{Sec: BoundayNormalizedforLE} introduces the normalized pre-Schwarzian and proves the two generalized Loewner-energy formulas, Theorems \ref{Thm:BoundaryLE} and \ref{Thm:InteriorLE}.  In Section~\ref{Sec: topologicalconvergence} we recall the arc-length parametrization of $T_0(1)$ and show many different characterizations of convergence in the normalized Weil--Petersson class. We begin to combine all these tools in Section~\ref{Sec:MainProof}, where we prove Theorems~\ref{Comparable} and \ref{Thm:LoewnerIP}, first for the analytic case, followed by the general case via approximation. These results allow us to prove  Theorems~\ref{Thm:GeneralCompBound} and~\ref{Thm:LEGrowth} in Section~\ref{sec: entropypart}.

\bigskip

\subsection{Acknowledgements}\label{Sec:Acknowledgements}
We thank Fredrik Viklund and Yilin Wang for helpful discussions.  Several passing conversations with Mario Bonk, during the latter's occasional sojourns in Beijing, were also inspiring. The authors are also grateful to Prof. Jinsong Liu for his support, as well as to the Beijing International Center for Mathematical Research and Prof. Zhiqiang Li for their support.

\section{Preliminaries}\label{Sec:Preliminaries}

In this section, we present some preliminaries to facilitate our investigation of Loewner energy and welding energy. We recall the universal \Teich space and Weil--Petersson \Teich space in terms of their four models in \S\ref{Sec:T(1)FourModels} and \S\ref{Sec:T0(1)FourModels}, respectively. In \S\ref{Pre: LE}, we review the Loewner energy and its \Mob invariance, which will be a key property we use in this work. Finally, \S\ref{Sec: Honehalfspcae}, \S\ref{Sec: Douglasformula}, and \S\ref{Pre: Pullbackoperator} establish the foundation for us to define welding energy and investigate its properties.

Introduced by Bers~\cite{Bers61, Bers65}, $T(1)$
bridges between spaces of univalent functions and general \Teich spaces, forming an infinite-dimensional complex Banach manifold. 
As the largest \Teich space (modded by the Fuchsian group $\textrm{id}$), $T(1)$ naturally contains the \Teich spaces of all hyperbolic Riemann surfaces as embedded submanifolds. The universal properties inherent in $T(1)$ have motivated profound mathematical investigations. Moreover, $T(1)$ provides a promising geometric framework for the non-perturbative formulation of bosonic string theory, rendering it a compelling object of study in mathematical physics~\cite{BowickRajeev87,UTS95}. We recommend~\cite{Lehtobook,TTbook,LecturesonUTS14Sergeev} to the interested readers for a more detailed introduction to universal \Teich space.

\subsection{Four models of Universal \Teich space}\label{Sec:T(1)FourModels}

Let $U, V \subset \C$ be open. A mapping $f: U \to V$ is \emph{$K$-quasiconformal} if it is homeomorphism whose gradient, interpreted in the sense of distributions, belongs to $L_{\textrm{loc}}^{2}(\mathbb C)$ and satisfies $\|\mu\|_{\infty}=k < 1$. Here $\mu$ is $f$'s a.e.-defined \emph{complex dilatation}
\[\mu := \partial_{\bar{z}}f/ \partial_{z}f,\] and $K:=\frac{1+k}{1-k} \geq 1$.  As usual, $\partial_z := \frac{1}{2}(\partial_x - i \partial_y)$, and $\partial_{\bar{z}} := \frac{1}{2}(\partial_x + i \partial_y)$.  A map is \emph{quasiconformal} (abbreviated \emph{QC}) if it is $K$-quasiconformal for some $K$.

We begin by reviewing four equivalent definitions of the universal \Teich space $T(1)$ based on the model space of the upper half-plane $\H$.  

\subsubsection{Beltrami differentials}\label{Sec:PrelimBeltramiDiff}

\noindent Starting from a so-called \emph{Beltrami differential} 
\begin{align}\label{Eq:BeltramiSet}
    \mu \in L^{\infty}_{1}(\H^*):= \left\{ \mu \in L^{\infty}(\H^*): \|\mu\|_{\infty}<1 \right\},
\end{align}
extend $\mu$ to $\C$ via the reflection
\begin{align}\label{eq: BeltramiCoe.H}
   \mu(z)=\overline{\mu(\bar{z})}. 
\end{align}
By the measurable Riemann mapping theorem, let $w_{\mu}: \C \to \C$ be the unique solution to the Beltrami equation \begin{align}\label{Eq:Beltrami}
    \partial_{\bar{z}} w = \mu \partial_{z} w
\end{align}
which fixes $0,1,\infty$. Then $w_{\mu}$ satisfies $w_{\mu}(z)= \overline{w_{\mu}(\bar{z})}$ and fixes domains $\H, \H^*$, and thus the real line $\R$. For $\mu,\nu \in L^{\infty}_{1}(\H^*)$, we define the equivalent relation as $\mu \sim \nu$ if $w_{\mu}|_{\R}=w_{\nu}|_{\R}$.
The universal \Teich space $T(1)$ is then $T(1):= L^{\infty}_{1}(\H^*) / \sim.$

\subsubsection{Univalent functions}\label{Sec:PrelimUnivalentFunc}

\noindent An alternative extension of $\mu \in L^{\infty}_{1}(\H^*)$ is to set $\mu \equiv 0$ on $\H$. Then the unique solution $w^{\mu}$ to \eqref{Eq:Beltrami} which fixes $0,1,\infty$ is holomorphic on $\H$. For any $\mu, \nu \in L^{\infty}_{1}(\H^*)$, the above equivalent relation transfers into $\mu \sim \nu$ if and only if $w^{\mu}=w^{\nu}$ on $\H$ (see \cite[Ch.2 Lem. 4]{LecturesonUTS14Sergeev} for the details), then the map $[\mu] \mapsto w^{\mu}|_{\H}$ yields
\begin{align}\label{Eq:T1UnivalentFunc}
    T(1) \simeq \left\{ \text{conformal }f:\H \to \C \text{ fixing} \, 0,1, \text{and}\, \infty, \, \text{and extendable to a QC map of } \C \right\}.
\end{align}
For us, a \emph{conformal map} is a holomorphic (or meromorphic if $\infty$ is in the image of $f$) and univalent function.  

\subsubsection{Quasicircles}\label{Sec:PrelimQuasicircles}
\noindent 
A Jordan curve $\gamma$ is a \emph{$K$-quasicircle} if it is an image of a circle $C \subset \hatC$ under a $K$-quasiconformal map of $\C$, and a \emph{quasicircle} if it is a $K$-quasicircle for some $K$.  The images $\gamma = f(\R)$ for all $f$ belonging to the right-hand side of \eqref{Eq:T1UnivalentFunc} are thus oriented, normalized quasicircles. Conversely, given any such $\gamma$, we can find an associated normalized conformal map $f$.  Thus sending $[\mu] \mapsto \gamma_{\mu}:=w^{\mu}(\R)$ yields the equivalent description
\begin{align}\label{Fact: QuasicirclesModelofT01}
    T(1) \simeq \left\{ \text{all oriented quasicircles $\gamma$ passing through} \, 0,1, \text{and}\, \infty \right\}.
\end{align}
The orientation of $\gamma$ can be a detail that is easy to overlook.  Note, for instance, that $\gamma$ and its complex conjugate $\overline{\gamma}$ are not the same points in $T(1)$, even though both curves pass through $0,1,$ and $\infty$.

\subsubsection{Quasisymmetric homeomorphisms of the real line}

\noindent Let $\ophom{C}$ be the orientation-preserving homeomorphisms of $C$, where $C$ is a circle on the Riemann sphere $\hatC$.  Most commonly for us, $C \in \{\R,\SSS\}$. An element $h:C \to C$ of $\ophom{C}$ is \emph{$M$-quasisymmetric} if there exists $M\geq 1$ such that for all adjacent arcs $I,J$ of equal length, $|h(I)|/|h(J)| \leq M$, where $| \cdot |$ denotes arc-length.\footnote{There is also an equivalent definition in terms of quasisymmetry which can be generalized to general metric spaces; see \cite{Heinonen01Lecturebook}.}  We say $h$ is \emph{quasisymmetric} (abbreviated QS) if it is $M$-quasisymmetric for some $M$. Let $\operatorname{QS}(C)$ be set of orientation-preserving quasisymmetric homeomorphisms of $C$.  As quasisymmetric maps are closed under composition, $\operatorname{QS}(C)$ forms a group \cite[Ch.III~\S1.2]{Lehtobook}.  

Ahlfors and Beurling showed that quasisymmetric $h: \R \rightarrow \R$ are precisely the boundary values of quasiconformal maps on $\H$ (or $\D$) \cite{BAextension} (or \cite[Thm. 5.1]{Lehtobook}), and thus the map $[\mu] \mapsto w_\mu|_\R$ yields identification
\begin{align}\label{Eq:T1NormalizedWelding} 
    T(1) \simeq \left\{\, h \in \QS(\R) \; : \; h(0) = 0 = h(1)-1 \,\right\} \simeq \PSL_2(\R) \bs \QS(\R).
\end{align}
Note that since such $h$ are increasing homeomorphisms, we may also regard them as fixing all three points $0,1,$ and $\infty$.  As we will see in the next subsection, elements of \eqref{Eq:T1NormalizedWelding} are precisely the normalized conformal weldings of the quasicircles of \S\ref{Sec:PrelimQuasicircles}.

\subsubsection{Moving between the models}\label{Sec:PrelimMovingBetweenModels}
The equivalence of these four models is standard, and can be found in, for instance,~\cite[Ch.III \S1]{Lehtobook},~\cite[\S2.1]{RandomCW}, and~\cite[\S2.1]{LecturesonUTS14Sergeev}. We proceed to discuss, however, how one can move between the quasicircle model and the QS homeomorphism model fairly explicitly through conformal welding.  

First, given a quasicircle $\gamma$, denote the two components of $\C \bs \gamma$ by $H$ and $H^*$. Since $\gamma$ is a quasicircle, the uniformizing map $w^{\mu}$ from $\H$ to $H$ can be quasiconformally extended to $\C$ with Beltrami differential $\mu$ on $\C$.  With $w_{\mu}$ constructed from $\mu$ as in \S\ref{Sec:PrelimBeltramiDiff} above, set $g_{\mu}:=w^{\mu}\circ w_{\mu}^{-1}|_{\H^*}$, which is conformal on $\H^*$.
The \emph{conformal welding} associated to $\gamma_{\mu}$ is then
$g_{\mu}^{-1} \circ f_{\mu}|_{\R}= w_{\mu}|_{\R}$, which is an element of $\QS(\R)$ thanks to the Beurling-Ahlfors' theorem~\cite{BAextension}. Note that this argument also enables one to move from univalent map model of $T(1)$ to the $\QS(\R)$ model.

Conversely, given $h \in \ophom{\R}$, extend $h$ quasiconformally to $w:\H^* \to \H^*$ fixing $0,1$, and $\infty$ such that $w|_{\R}=h$. Setting $\mu := \partial_{\bar{z}}w/\partial_z w$ as the corresponding Beltrami differential, let $w_{\mu}$ and $w^{\mu}$ be as in \S\ref{Sec:PrelimBeltramiDiff} and \S\ref{Sec:PrelimUnivalentFunc} above. Then $f:=w^{\mu}|_{\H}$ and $g:=w^{\mu}\circ w_{\mu}^{-1}|_{\H^*}$ are the two desired normalized conformal maps and $\gamma=f(\R)$ is the quasicircle corresponding to $h$. Note that the uniqueness of normalized solutions to the Beltrami equation yields that $f$ and $g$ are unique, regardless of what extension $w$ of $h$ we begin with. We refer the interested readers to~\cite[\S2.1]{LecturesonUTS14Sergeev}, \cite[Ch.III \S1]{Lehtobook}, and~\cite[\S2.1]{RandomCW}) for more details.

\subsubsection{Further comments on welding}
In the absence of the normalization conditions in $T(1)$, weldings $h$ and Jordan curves $\gamma$ can be considered as elements of equivalence classes.  Since $h$ is invariant under post-composing $\gamma$ by an automorphism of $\widehat{\C}$, we can view $\gamma$ as an element of  $\PSL_2(\C) \bs \left\{\text{Jordan curves}\right\}$ from the vantage point of welding. Similarly, pre-composing the conformal maps $f$ and $g$ by automorphisms of $\H$ and $\H^*$, respectively, does not change $\gamma$, and thus we may view (an un-normalized) welding as an element of 
\[ \PSL_{2}(\R) \bs \ophom{\R} / \PSL_{2}(\R).\]

It has been observed that the group $\PSL_2(\R) \times \PSL_2(\R)$ describing the freedom of choice can be viewed as the collection of orientation-preserving and time-preserving isometries of $\operatorname{AdS}^3$ space. Furthermore, the graph of the welding homeomorphism can be viewed as a space-like curve in $\partial_{\infty}\operatorname{AdS}^3$. See~\cite{Survey20AdSandTeich, Yilin2024two} for further details.

\subsubsection{Some intuition}

As evident from the nature of the Beltrami differentials \eqref{Eq:BeltramiSet} used to construct $T(1)$, the universal \Teich space is a type of $L^\infty$ class.  This intuition is also evident from the \emph{Bers' embedding} 
\begin{align}\label{Eq:BersEmbedding}
   \beta([\mu]):= \calS_{w^{\mu}|_{\H}},  
\end{align}
which embeds $T(1)$ into the Banach space
\begin{align}
    A_{\infty}^1(\mathbb{H}) &:= \big\{\, \phi\in \calH(\H) \; \big| \;  \ \|\phi\|_{A_{\infty}^1(\mathbb{H})}:=\sup_{z\in \H} |\phi(z)| \rho_\H^{-1/2}(z) \,< \infty \,\big\},\label{Def:A^1infty}
\end{align}
another $L^\infty$-type description.  Here $\calH(\Omega)$ is the collection of holomorphic functions on $\Omega$, and $\rho_{\mathbb{H}} (z) = 1/y^2$ is hyperbolic area density in $\mathbb{H}$.\footnote{We define parallel function spaces such as $A^1_\infty(\D)$ by using the corresponding hyperbolic metric $\rho_\D$ for that domain.}  (We largely follow the notation of \cite{TTbook} for function space definitions.)  Pulling back the complex structure on $A^1_\infty(\H)$ via $\beta$ equips $T(1)$ with an infinite-dimension Banach manifold structure.

\subsection{Weil--Petersson \Teich space}\label{Sec:T0(1)FourModels}

We can view the \emph{Weil--Petersson \Teich space $T_0(1)$} as an $L^2$-subclass of $T(1)$. Parallel to the four models of $T(1)$ in \S\ref{Sec:T(1)FourModels} above, we have the following four descriptions of $T_0(1)$.  Recall the equivalence relation for Beltrami differentials $\mu$ defined in \S\ref{Sec:PrelimBeltramiDiff} above.

\subsubsection{Beltrami differentials}

Cui~\cite{Cui} and Takhtajan-Teo~\cite{TTbook} showed $[\mu] \in T_0(1)$ if and only if it has a representative $\mu \in L_1^{\infty}(\H^*) \cap L^2(\H^*,dA_\rho)$, where $dA_\rho$ is hyperbolic area measure.  That is, the representative $\mu$ satisfies
\begin{align}\label{Eq:WPBeltrami}
    \int_{\H^*} |\mu(z)|^2 \rho_{\H^*}(z) \,dA(z) < \infty.
\end{align}
While Cui's and Takhtajan-Teo's results were formulated for the underlying domain of $\D$ or $\D^*$, they immediately transfer to $\H$ or $\H^*$ by a change of variables; similarly below.  (See \S\ref{Sec:WPHistory} for a more precise accounting of Cui and Takhtajan-Teo's contributions.)

\subsubsection{Univalent functions}\label{Sec:PrelimUnivalentFuncWP}

From the point of view of univalent functions, $T_0(1)$ can be viewed as those normalized conformal maps $f:\H \rightarrow \C$ which have a quasi-conformal extension to $\C$ and whose Schwarzian derivative \eqref{Def:Schwarzian} belongs to the separable Hilbert space
\begin{align}
    A_2(\mathbb{H}) := \big\{\, \phi \in \calH(\H) \; \big| \;  \ \|\phi\|^2_{A_2(\mathbb{H})}:=\int_{\mathbb{H}}|\phi(z)|^2 \rho^{-1}_{\mathbb{H}}(z) \,d A(z)< \infty \,\big\}. \label{Def:A_2}
\end{align}
See \cite[Thm. 2]{Cui} or \cite[Ch.I Cor. 3.2]{TTbook}. As in \S\ref{Sec:PrelimUnivalentFunc}, here ``normalized" means $f$ fixes $0,1,$ and $\infty$. Phrased in terms of the Bers' embedding \eqref{Eq:BersEmbedding}, this says $T_0(1)$ are those elements $[\mu] \in T(1)$ satisfying $\beta([\mu]) \in A_2(\H)$.  The Hilbert manifold structure on $T_0(1)$ introduced in~\cite[Ch.I~\S2]{TTbook} is modeled on $A_2(\mathbb{H})$ via the Bers' embedding. 

We can also characterize the univalent functions corresponding to $T_0(1)$ via the pre-Schwarzian derivative \eqref{Def:Preschwarzian}.  Namely, they are those normalized conformal maps $f: \H \rightarrow \C$ which have QC extension to $\C$ and satisfy $\calA_f \in L^2(\H)$ \cite[Prop. 1]{Cui}, \cite[Ch.II Thm.~1.12]{TTbook}, \cite[Thm. 4.4]{STWrealline}.  Following \cite{TTbook}'s notation, this says $\calA_f \in A^1_2(\H)$, where
\begin{align}\label{Def:A^1_2}
    A^1_2(\H) := \calH(\H)\cap L^2(\H) = \big\{\, \phi \in \calH(\H) \; \big| \;  \|\phi\|^2_{A_2^1(\mathbb{H})}:=\int_{\mathbb{H}}|\phi|^2  \, d A < \infty \,\big\}.
\end{align}
We remark that $A_2^1(\mathbb{H})$ continuously embeds in $A_{\infty}^1(\mathbb{H})$~\cite{Zhu2007operator}.

\subsubsection{Quasicircles}

\noindent As $T(1)$ quasicircles  are the images of $\R$ under the corresponding univalent functions $f$, so $T_0(1)$ quasicircles are images of $\R$ under the normalized univalent maps whose Schwarzian lies in \eqref{Def:A_2}.  In particular, these are oriented quasicircles which pass through $0,1,$ and $\infty$. 

In general, we say a Jordan curve $\gamma$ is \emph{Weil--Petersson} if it is $\PSL_2(\C)$-equivalent to such a quasicircle.

\subsubsection{Quasisymmetric homeomorphisms}

As we noted in \S\ref{Sec:PrelimMovingBetweenModels} above, the quasisymmetric homeomorphism model of $T(1)$ consists of the normalized conformal weldings of quasicircles.  Similarly, the QS homeomorphism model of $T_0(1)$ consists of the normalized weldings of $T_0(1)$ quasicircles.  If we set $\WP(C)$ to be the collection of non-normalized conformal weldings of Weil--Petersson quasicircles $\gamma$, then 
\begin{align*}
    T_0(1) \simeq \PSL_2(\R) \bs \WP(\R).
\end{align*}
We recall from the introduction that Shen and his collaborators characterized the $h \in \WP(C)$ for $C \in \{\SSS, \R\}$ as essentially those satisfying $\log h' \in H^{\half}(C)$.  See Propositions \ref{IntroProp:ShenWeldingD} and \ref{IntroProp:ShenWeldingH}.

\subsubsection{Historical sketch}\label{Sec:WPHistory}

The study of Weil--Petersson \Teich space appears to have been initiated by Cui~\cite{Cui}, although under the name of ``integrable asymptotic affine homeomorphisms." Cui was the first to study how Beltrami differentials $\mu$ satisfying \eqref{Eq:WPBeltrami} relate to the corresponding univalent maps, and gave the characterizations in \S\ref{Sec:PrelimUnivalentFuncWP} above in the setting of $\D$. He also showed that the space of these homeomorphisms is the completion of $\mathcal{M}= \text{M\"ob}(\SSS) \bs \operatorname{Diff}(\SSS)$ under the Weil--Petersson metric.  We recall that \cite{BowickRajeev87} had previously showed $\mathcal{M}$ has a unique K\"ahler metric up to scaling, although it was known that $\mathcal{M}$ was not complete under this metric.  Understanding the completion of $\mathcal{M}$ appears to have been a key motivation for Cui.

Takhtajan and Teo~\cite{TTbook} introduced a new Hilbert structure on the universal Teichm\"uller space $T(1)$ which divides $T(1)$ into uncountably-many components, and showed that the connected component of the identity $T_0(1)$ coincides with Cui's integrable asymptotic affine homeomorphism space. They also gave other novel characterizations of the space, including in terms of the universal Liouville action and Hilbert--Schmidt operators, for instance.  Their work includes a different proof that $T_0(1)$ is the completion of $\text{M\"ob}(\SSS) \backslash \text{Diff}(\SSS)$ under the Weil--Petersson metric. 

It was Shen~\cite{ShenWP} who characterized $T_0(1)$ in terms of the quasisymmetric homeomorphism model, and appears to be the first to call $T_0(1)$ the ``Weil--Petersson \Teich space." He and his collaborators~\cite{WPII,WPIII} investigated the space from the perspective of harmonic analysis and showed continuous dependence between these weldings and the corresponding Riemann maps. 

Motivated by probability theory, Wang and her collaborators studied the Loewner energy and showed that finite-energy curves are precisely Weil--Petersson quasicircles, thus opening a surprising door between SLE theory and \Teich theory~\cite{Yilin19,Yilinequivalent,Interplay,LDP24EveliinaYilin}. Bishop~\cite{Bishop2019weil} further expanded the horizons by giving many equivalent descriptions of Weil-Petersson quasicircles in terms of Sobolev spaces, geometric measure theory, and hyperbolic geometry.

The cumulative effect of these works has been to generate sustained interest in the Weil--Petersson \Teich space, which now finds itself at the interface of complex analysis, probability theory, hyperbolic geometry, and mathematical physics.


\subsection{Loewner energy}\label{Pre: LE}

Consider a simple curve $\gamma\colon [0, T) \rightarrow \bar{\mathbb{H}}$, where $T \in (0, \infty]$, satisfying $\gamma(0)=0, \, \gamma(0,T)\subset \mathbb{H}$.  For each $t \in[0, T)$, we can  associate $\gamma[0,t]$ with a unique conformal map $g_t: \mathbb{H} \backslash \gamma_{[0, t]} \rightarrow \mathbb{H}$ with the normalization condition near $\infty$: $g_t(z)= z +\frac{2 t}{z} +o\left(\frac{1}{z}\right)$.
By Carathedory extension theorem, $g_t$ can be continuously extended to the prime ends, thus we can define the driving function $W(t): =g_t(\gamma_t)$. This construction uniquely determines a family of conformal maps $\{g_t(z)\}_{0 \leq t<T}$
governed by the so-called Loewner differential equation

\begin{align}\label{Intro: chordalLE}
   \partial_t g_t(z) =\frac{2}{g_t(z)-W_t}, \quad g_0(z)=z.
\end{align}

Emerging from Loewner's equation, the \emph{chordal Loewner energy} of a simple curve $\g \subset \H$ with target points $0$ and $\infty$, introduced independently in~\cite{ShekharFriz17} and~\cite{Yilin19}, is the Dirichlet energy of its driving function, i.e.
\begin{align*}
 I^{C}(\gamma): =\frac{1}{2}\int_0^\infty W'(t)^2 dt.
\end{align*}
The driving function $W(t)$ encodes a simply connected domain, or its uniformizing conformal map, via a real-valued driving function of its boundary, and $I^{C}(\gamma)$ measures how far away $\gamma$ deviates from the hyperbolic geodesic.

Rohde-Wang~\cite{Loopenergy} later extended this from chords connecting distinct boundary points of a domain $\Omega \subsetneq \C$ to  Jordan curves $\gamma \subset \hatC$ as follows.  Let $\gamma: \left[0,1\right] \mapsto \C$ be an oriented Jordan curve with root $\gamma(0)= \gamma(1)$. For arbitrary $\epsilon>0$, set  $\gamma \left[\epsilon,1\right]$ be a chord connecting $\gamma(\epsilon)$ and $\gamma(1)$ in the simply connected domain $\widehat{\C} \bs \gamma \left[0, \epsilon\right]$, the \emph{loop Loewner energy} is defined to be
\[I^{L}(\gamma, \gamma(0)):= \lim\limits_{\epsilon \to 0} I_{\widehat{\C}\bs \gamma[0,\epsilon],\gamma(\epsilon),\gamma(0)}(\gamma[\epsilon,1]).\]
Furthermore, they showed this energy is root-invariant, which is highly nontrivial.

One consequence of their work is that the Loewner energy is completely conformally invariant, i.e.
\begin{align}\label{Eq: MobinvariantofLE}
    I^{L}(S \circ \gamma)=I^{L}(\gamma),
\end{align}
for any $S \in \PSL_2(\C)$.  This is a non-trivial property but immediately follows from Rohde-Wang's definition of the loop Loewner energy in \cite[\S3]{Loopenergy}.\footnote{Another way to see $I^{L}(\cdot)$ is \Mob invariant is that we can express Loewner energy $I^{L}(\cdot)$ by means of zeta-regularized determinants of Laplacians~\cite[Thm. 7.3]{Yilinequivalent} and then use Polyakov-Alvarez conformal anomaly formula~\cite{Polyakov81,Alvarez83,OsgoodPhillipsSarnak88} to show it is the universal Liouville action \eqref{Eq:LEDisk} in smooth case. Since these determinants are conformally invariant, the \Mob invariance for $I^{L}(\cdot)$ follows. This approach, however, is rather indirect.}

Wang~\cite{Yilinequivalent} proceeded to give the surprising equivalent descriptions of Loewner energy \eqref{Eq:LEDisk}, \eqref{Eq:LEH}, which made the first contact between $\SLE$ theory and \Teich theory. We comment that showing the equivalence between \eqref{Eq:LEDisk} and \eqref{Eq:LEH}, while having the appearance of a perhaps-routine problem in complex function theory, appears to be non-trivial.  Wang's approach in \cite[\S7]{Yilinequivalent} is to begin with smooth curves and use zeta-regularized determinants of Laplacians.  She then approximates a general curve via smooth use and continuity results in the Weil--Petersson \Teich space established by Takhtajan-Teo~\cite[Cor. A.4, Cor. A.6]{TTbook}, along with the lower semicontinuity of the Loewner energy.

\subsection{A fractional Sobolev space}\label{Sec: Honehalfspcae}

For $C$ a circle in $\hatC$, the space $H^{\half}(C)$ consists of all $u \in L^1_{loc}(C)$ such that
\begin{equation}\label{Def:H1/2SemiNorm}
    \|u\|_{H^{1/2}(C)}^2 := \frac{1}{4\pi^2} \int_C \int_{C} \frac{\abs{u(\zeta) - u(\xi)}^2}{\abs{\zeta - \xi}^2} \,|\dd \zeta|\, |\dd \xi| < \infty,
\end{equation}
where integration is with respect to arc length.  We will mostly use either $C=\R$ or $C=\SSS$.  The semi-norm \eqref{Def:H1/2SemiNorm} arises from the inner product
\begin{align}\label{Eq:H1/2IP}
    \brac{f,g}_{H^{1/2}(C)} := \frac{1}{4 \pi^2}\int_{C}\int_{C} \frac{(f(\zeta)-f(\xi))(\overline{g(\zeta)-g(\xi))}}{|\zeta-\xi|^2} \, |\dd \zeta|\,|\dd \xi|,
\end{align}
and upon modding out by constants $H^\half(C)$ becomes a Hilbert space.  We often simplify notation by writing $\|\cdot \|_{\half(C)}$ or even $\|\cdot \|_{\half}$ when the domain $C$ is contextually clear, and similarly for $\brac{f,g}_{H^{1/2}(C)}$.

\Mob transformations act as isometries with respect to the $H^\half$ semi-norm, in the sense that if $\PSL_2(\C) \ni T: C_1 \rightarrow C_2$ and $u:C_2 \rightarrow \C$, then $\| u \|_{H^{1/2}(C_2)} = \| u \circ T \|_{H^{1/2}(C_1)}$.  We recall this fact in Lemma \ref{Lem: pullbackoperatorMob} below.

\subsection{The Douglas formula}\label{Sec: Douglasformula}

For a domain $\Omega \subset \C$, set 
\begin{align*}
    \calD(\Omega) := \big\{\, F \in C^1(\Omega) \; | \; \mathcal{D}_\Omega (F) := \frac{1}{2\pi}\int_\Omega |\nabla F(z)|^2 \dd x\,\dd y<\infty \,\big\}.
\end{align*}
(We note that some authors define the Dirichlet energy as $\frac{1}{\pi}\| \nabla F\|^2_{L^2(\H)}$, while others as $\frac{1}{2\pi}\| \nabla F\|^2_{L^2(\H)}$.  We follow the latter convention in order to have the trace map be an isometry from the Dirichlet space to $H^{\half}$, as in \eqref{Eq: DouglasformulaFirstDisk} below.) The Dirichlet energy $\calD_\Omega(F)$ is invariant under pushforward $f_*(F) = F \circ f^{-1}$ by conformal maps $f$, in the sense that if $f: \Omega_1 \rightarrow \Omega_2$ is conformal and $F \in C^1(\Omega_1)$, then $\calD_{\Omega_2}(f_*(F)) = \calD_{\Omega_1}(F)$. Write $\mathcal{E}_{\text{harm}}(\Omega) \subset \mathcal{D}(\Omega)$ for the sub-collection of harmonic elements of $\calD(\Omega)$. 

In his study of Plateau problem, Douglas~\cite{Douglas} showed the Dirichlet energy of $U \in \mathcal{E}_{\text{harm}}(\D)$ satisfies
\begin{align} \label{Eq: DouglasformulaFirstDisk}
    \|U\|_{\calD(\D)}^2 := \calD_{\D}(U) &= \frac{1}{16\pi^2} \int_{0}^{2\pi}\int_{0}^{2\pi} \frac{|u(e^{i\alpha})-u(e^{i\beta}) |^2}{ \sin^2\!\big(\frac{\alpha- \beta}{2}\big)} \,\dd \alpha \, \dd \beta \notag\\
    &= \frac{1}{4\pi^2} \int_{\SSS}\int_{\SSS} \left|\frac{u(z_1)-u(z_2)}{z_1-z_2}\right|^2 |\dd z_1||\dd z_2|=\|u\|_{H^{1/2}(\SSS)}^2,
\end{align}
where $u = \mathrm{tr}(U)$ is the boundary trace of $U$, defined by taking non-tangential limits. In particular we see that $\mathrm{tr}:\mathcal{E}_{\text{harm}}(\D) \rightarrow H^{\half}(\SSS)$ is an isometry.

In the other direction, beginning with boundary data $u \in L^1(\SSS)$, one obtains a harmonic function $U$ through the Poisson integral $P_\D(u)$, and when $u \in H^{\half}(\SSS)$, computation~\cite[Theorem 2.5]{AhlforsCI} shows 
\begin{align}\label{Eq: Douglasformuladisk}
    \calD_{\D}(P_{\D}(u))= \|u\|_{H^{1/2}(\SSS)}^2.
\end{align}
See~\cite[Sec.4.3.1]{Julien09} for an alternative proof using the Neumann jump operator $N$.  Thus $P_\D:H^\half(\SSS) \rightarrow \mathcal{E}_{\text{harm}}(\D)$ is likewise an isometry.  

By \Mob invariance of the Dirichlet energy and the $H^\half$ semi-norm (see Lemma \ref{Lem: pullbackoperatorMob} for the latter), these results carry over to $\H$ and $\partial \H = \R$.  That is, $\calD_{\H}(U)=\|u\|_{H^{1/2}(\R)}^2$ for $U \in \calD(\H)$ and $u = \mathrm{tr}(U)$, and $\calD_{\H}(P_{\H}(u))= \|u\|_{H^{1/2}(\R)}^2$ when $u \in H^\half(\R)$.  Here $P_\H$ is, of course, the Poisson extension of $u$ to $\H$.  

Note that Wang's Loewner energy formula \eqref{Eq:LEH} for $\H$ thus translates as
\begin{align}
    I^L(\gamma) &= 2\mathcal{D}_\mathbb{H}(\log|f'(z)|) + 2\mathcal{D}_{\mathbb{H}^*}(\log|g'(z)|) \notag\\
    &= 2\big\| \log|f'(x)| \big\|_{H^{1/2}(\R)}^2 + 2\big\| \log|g'(x)| \big\|_{H^{1/2}(\R)}^2.\label{Eq:LoewnerDirichletHalf}
\end{align}

\subsubsection{Recovering a function from its boundary values}

For any $U \in \mathcal{E}_{\text{harm}}(\H)$, $U \circ C^{-1} \in \mathcal{E}_{\text{harm}}(\D)$ by \Mob invariance of the harmonic Dirichlet space, where $C: \H \to \D $ is the Cayley transformation. Thus $U \circ C^{-1}$ has nontangential limit $u \circ C^{-1}$ almost everywhere on $\SSS$. Furthermore, since $\mathcal{E}_{\text{harm}}(\D)$ is a subspace of harmonic Hardy space $h^2(\D)$, $U$ can be recovered by its boundary values $u$. That is, 
\[
    U \circ C^{-1}= P_{\D}(u \circ C^{-1})
\]
(see~\cite[Page 256]{Zhu2007operator} or \cite[\S2]{Dirichletchordarc25} for details). Moving back to $\H$, it follows that $U$ can be recovered from its boundary values $u$ by the Poisson integral, i.e.
\begin{align}\label{Eq:PoissonLeftInverse}
    U = P_{\D}(u \circ C^{-1}) \circ C =P_{\H}(u) = P_\H(\mathrm{tr}(U)).
\end{align}

In particular, for Weil--Petersson $\gamma$ passing through $\infty$ and associated conformal maps $f: \H \to H$ and $g: \H^* \to H^*$ which fix $\infty$, we see from \eqref{Eq:LoewnerDirichletHalf} that $\log|f'|\in \mathcal{E}_{\text{harm}}(\H)$ and $\log|g'| \in \mathcal{E}_{\text{harm}}(\H^*)$.  Thus we may recover these functions by the Poisson integrals $P_\H(\log|f'(x)|)$ and $P_{\H^*}(\log|g'(x)|)$ of their boundary traces $\log|f'(x)|, \log|g'(x)|$ on $\R$, a fact we will repeatedly use in \S\ref{Sec:MainProof}.\footnote{In the case of bounded Weil--Petersson curves $\gamma \subset \C$, we obtain the same result for the conformal maps $f:\D \rightarrow \Omega$, $g: \D^* \rightarrow \Omega^*$ from the fact that $\Omega, \Omega^*$ are chord-arc domains, and hence Smirnov domains.  The latter, we recall, are precisely the set of simply-connected domains $\Omega$ with $\partial \Omega \subset \C$ where $\log|f'|$ is identical to the Poisson extension of its boundary values.  See \cite[Section 7.3]{Pommerenke}.}

We can also apply \eqref{Eq:PoissonLeftInverse} to the Dirichlet and $H^\half$ inner products.  Writing
\begin{align*}
    (F,G)_{\nabla(\Omega)} := \frac{1}{2\pi}\int_{\Omega} \nabla F \cdot \nabla G \, dx\,dy
\end{align*}
for the former, we thus have by the polarization identity, \eqref{Eq: Douglasformuladisk}, and \eqref{Eq:PoissonLeftInverse}, that for any  $F,G \in \mathcal{E}_{\text{harm}}(\H)$,
\begin{align}\label{Eq:Parallel}
    \brac{\mathrm{tr}(F), \mathrm{tr}(G)}_{H^{1/2}(\R)}= \big(P_\H(\mathrm{tr}(F)), P_\H(\mathrm{tr}(G))\big)_{\nabla(\mathbb{H})}= (F, G)_{\nabla(\mathbb{H})},
\end{align}
and similarly for $F,G \in \mathcal{E}_{\text{harm}}(\H^*)$. This identity will also prove useful in \S~\ref{Sec:MainProof}.

\subsection{The pullback operator}\label{Pre: Pullbackoperator}

Given functions $h:X_1 \rightarrow X_2$ and $\varphi:X_2 \rightarrow X_3$, we define the pullback operator as
 $\calP_h(\gvp) :=\gvp \circ h.$  In our context, $\varphi$ will typically be a function on some circle $C_2 \subset \hatC$, and $h$ will be a homeomorphism $h:C_1 \rightarrow C_2$ of circles or an element of $\ophom{C_2}$.  As already mentioned in the introduction, in this latter setting Nag and Sullivan \cite{Nag-SullivanTeichmuller} proved that $\QS(C_2)$ consists precisely of those functions for which $\mathcal{P}_h$ acts boundedly on $H^{\half}(C_2)$.  While their original result was for $C_2 = \SSS$, we formulate this in terms of $C_2 = \R$.\footnote{The Nag-Sullivan proof essentially only uses the Dirichlet principle (the Dirichlet energy is minimized by the Poisson extension) and the local inequality \cite[(40)]{Nag-SullivanTeichmuller}, which hold in any disk in $\hatC$. We also note the definition of $\mathcal{P}_h$ in \cite{Nag-SullivanTeichmuller} includes subtracting off the average value of the composition $\varphi \circ h$, but that does not affect $\| \cdot \|_{\half}$.}  We repeat this result from the introduction for the reader's convenience.
\begin{prop}[\cite{Nag-SullivanTeichmuller}] 
   $\calP_h$ is a bounded operator on $H^{\half}(\R)$ if and only if $h \in \operatorname{QS}(\R)$. Moreover, $\|\mathcal{P}_h\| \leq \sqrt{K+K^{-1}}$ if $h$ has a K-quasiconformal extension to $\mathbb{H}$.
\end{prop}
\noindent Our $h$ of interest will be weldings of $K$-quasicircles, for which the proposition yields the following.
\begin{cor}\label{Cor:OperatorBound}
    Let $h \in \ophom{\R}$ be the conformal welding of a $K$-quasicircle $\gamma$.  Then 
    \begin{align*}
        1 \leq \max\{ \|\mathcal{P}_h\|, \| \mathcal{P}_{h^{-1}}\|\} \leq \sqrt{K^2+K^{-2}}.
    \end{align*}
\end{cor}
\begin{proof}
    The lower bound is obvious, since $1=\|\mathcal{P}_h \circ \mathcal{P}_{h^{-1}}\| \leq \|\mathcal{P}_h \| \cdot \|\mathcal{P}_{h^{-1}}\|$.

    As $\gamma$ is a $K$-quasicircle, the conformal maps $f$ and $g$ from $\mathbb{H}$ and $\mathbb{H}^*$ to either side of $\gamma$ have $K^2$-quasiconformal extensions to $\mathbb{C}$ \cite[Ch.I Lem. 6.2]{Lehtobook}, and so the welding $h = g^{-1}\circ f$ has a $K^2$-quasiconformal extension to $\mathbb{C}$.  Since $h^{-1}$ corresponds to the complex-conjugated curve $\overline{\gamma}$, we see the conformal maps from $\mathbb{H}$ and $\mathbb{H}^*$ are $\overline{g(\bar{z})}$ and $\overline{f(\bar{z})}$, respectively, which have extensions of the same regularity as $g$ and $f$.  Hence the upper bound immediately follows from Proposition \ref{const.} above.
\end{proof}

\noindent We also remark that, so long as we are relying upon the Nag-Sullivan bound Proposition \ref{const.}, the appearance of $K^2$ in the previous bound, rather than $K$, is unavoidable.

\begin{lemma}\label{Lemma:K^2Optimal}
    An element $h\in \ophom{\R}$ welds a $K$-quasicircle $\gamma$ if and only if it has a $K^2$-quasiconformal extension to $\C$.
\end{lemma}
\begin{proof}
    The ``only if'' direction is in the proof of Corollary \ref{Cor:OperatorBound} above; simply use the $K^2$-QC extensions of both $f$ and $g^{-1}$.  Conversely, suppose $h$ extends $K^2$-quasiconformally to $\C$, with the extension $H$ having complex dilatation $\mu:\C \rightarrow \D$.  Let $f$ be a solution of the Beltrami equation with dilatation 0 in $\H$ and $\mu$ in $\H^*$.  Then $f$ and $g:= f\circ H^{-1}$ are conformal maps of $\H$ and $\H^*$, respectively, to either sides of the quasicircle $\gamma := f(\widehat{\R})$.  Since $f$ is conformal in $\H$ and extends $K^2$-quasiconformally in $\H^*$, $\gamma$ is a $K$-quasicircle \cite[Thm. 4(ii)]{Smirnov}.  
\end{proof}

We also comment that much more than Proposition \ref{const.} is true when $h$ is M\"obius, as the following well-known lemma asserts.
\begin{lemma}\label{Lem: pullbackoperatorMob}
    Let $C_1, C_2 \subset \hatC$ be circles.  If $T:C_1 \rightarrow C_2$ is an element of $\PSL_2(\C)$, then $\calP_{T}: H^\half(C_2) \rightarrow H^\half(C_1)$ is an isometry. 
\end{lemma}
\begin{proof}
    Using \eqref{Eq:MobDifferenceQuotient} we see 
    \begin{align*}
        4\pi^2 \| \varphi \circ T \|_{H^{1/2}(C_1)}^2 = &\int_{C_1} \int_{C_1} \frac{\left| \varphi \circ T (z_1) - \varphi \circ T (z_2) \right|^2}{\left|z_1 - z_2\right|^2} |\dd z_1| |\dd z_2|\\
       =& \int_{C_2} \int_{C_2}  \frac{\left|\gvp(w_1)-\gvp(w_2)\right|^2}{ \left| T^{-1}(w_1) -T^{-1}(w_2) \right|^2} |(T^{-1})'(w_1)||(T^{-1})'(w_2)| |\dd w_1| |\dd w_2| \\
       =& \int_{C_2} \int_{C_2} \frac{\left|\gvp(w_1)-\gvp(w_2)\right|^2}{\left|w_1-w_2\right|^2} |\dd w_1| |\dd w_2| = 4\pi^2 \| \varphi \|_{H^{1/2}(C_2)}^2. \qedhere
    \end{align*}
\end{proof}

\subsubsection{The pullback operator acting on other classes of regularity}
Between the quasisymmetric case of Proposition \ref{const.} and the \Mob case of Lemma \ref{Lem: pullbackoperatorMob} is the Weil--Petersson case, and here the nature of $\calP_h$ has also been investigated.  While we will not use these results we summarize them for the reader's convenience.  Hu-Shen~\cite{Shen12OnQS} proved that a projection $P_h^{-}$ of $\calP_h$ is Hilbert--Schmidt if and only if $h\in \WP(\SSS)$. More recently, by translating $\calP_h$ into matrix $\Pi( h)$, Fan-Sung~\cite[Lem. 2.7]{FanJinwoo2025quasiinvariancesleweldingmeasures} showed $\Pi( h)\Pi( h)^*-I$ is Hilbert--Schmidt if and only if $h\in \WP(\SSS)$, and related this operator to the quasi-invariance of the log-correlated Gaussian ﬁeld and $\text{GMC}$ measure on $\SSS$. 

With respect to the norm of $\calP_h$, \cite[Prop. 6.9]{WeiMatsuzaki23} showed $\|\calP_h\|$ depends only on the Weil--Petersson distance from $h$ to $id$. We comment that what appears to be absent in the literature (and what would be useful for us) is a sharpening of Proposition \ref{const.} for Weil-Petersson $h$.  We suspect should be possible but have not attempted it for this study.

The nature of $\calP_h$ has also been examined for $h$ in larger function spaces.  For instance, $h$ is a strongly quasisymmetric homeomorphism from $\R$ to $\R$ if and only if $\calP_{h}$ is an automorphism on $\operatorname{BMO}(\R)$ \cite{PullbackBMO83, TeichandBMOA91}.

\subsection{Absolutely-continuous functions on the extended real line}\label{Pre: AC}

We recall that a function $h$ is \emph{locally absolutely continuous on $I \subset \R$}, denoted $h \in \ACloc(I)$, if it is absolutely continuous (AC) on each compact subset of $I$.  That is, $h \in \AC(K)$ for all compacts $K$ of $I$.  We need a version of this for functions defined on the extended reals $\hatR$.
\begin{definition}\label{Def:ACOnHatR}
    We say an element $h \in \ophom{\wR}$ is locally absolutely continuous, denoted $h \in \ACp(\hatR)$, if there exist \Mob transformations $S,T \in \PSL_2(\R)$ such that $S \circ h \circ T \in \ophom{\R}\cap \ACloc(\R)$.
\end{definition}
\noindent Note that, in particular, the \Mob conjugation $S \circ h \circ T$ in the definition fixes $\infty$.  We claim that the definition is independent of the choice of $S$ and $T$.  This is not too hard to see, but perhaps merits comment, and towards that end we first observe the removability of points for monotone AC functions.  We recall absolute continuity on the circle is defined with the same $\epsilon-\delta$ definition as on the real line, but where arc length $d$ replaces the Euclidean metric for both the domain and range.  
\begin{lemma}\label{Lemma:ACPtRemovable}
    Let $L \in \{\R,\SSS\}$ and $K \subset L$ be compact with $x_0 \in K$.  If $h:K \rightarrow L$ is an element of $\ACloc(K \bs \{x_0\})$ that is continuous at $x_0$ and monotone in a neighborhood of $x_0$, then $h \in \AC(K)$.
\end{lemma}

\noindent The proof is an exercise in the $\epsilon-\delta$ definition of absolute continuity, and the same logic extends to any finite collection of points $\{x_0, \ldots, x_n\} \subset K$.  Note the lemma is false without the assumption of monotonicity nearby $x_0$, as the example of $x \mapsto x \sin(1/x)$ shows.  

With the lemma in hand, we argue for the legitimacy of Definition \ref{Def:ACOnHatR}.  Indeed, suppose $H_1:=S_1 \circ h \circ T_1 \in \ophom{\R}\cap \ACloc(\R)$, and $S_2, T_2$ are other elements of $\PSL_2(\R)$ such that $H_2:=S_2 \circ h \circ T_2 \in \ophom{\R}$; we  show $H_2 \in \ACloc(\R)$ as well.  We have $H_2 = S \circ H_1 \circ T$ for $S:= S_2 \circ S_1^{-1}$ and $T:=T_1^{-1}\circ T_2$.  Consider two cases.
\begin{enumerate}[$(i)$]
    \item If $T$ fixes $\infty$ (i.e. is affine), then so does $S$, and thus $H_2 \in \ACloc(\R)$ since $H_1$ already is.
    \item Suppose there is some $x_0 \in \R$ such that $T(x_0) = \infty$.  Let $K\subset \R$ be compact.  Noting that $S^{-1}(\infty) \not\in H_1\circ T(K)$, the only issue is whether or not $x_0 \in K$. If $x_0 \notin K$, then $H_2$ is a composition of increasing AC functions on $K$ and thus belongs to $\AC(K)$.  If $x_0 \in K$, then we still have that $H_2$ is continuous, increasing at $x_0$, and an element of $\ACloc(K\bs\{x_0\})$, and thereby belongs to $\AC(K)$ by Lemma \ref{Lemma:ACPtRemovable}.
\end{enumerate}  
We conclude membership in $\ophom{\R}\cap \ACloc(\R)$ is independent of the choice of \Mob transformations in Definition \ref{Def:ACOnHatR}, as claimed.

The following lemma will be useful for translating results between $\hatR$ and $\SSS$.  Its main significance is that the local absolute continuity condition in the definition of $\ACp(\hatR)$ is equivalent to full absolute continuity on $\SSS$ upon a change of variables. 
\begin{lemma}\label{Lemma:ACHatREquivalences}
    The following are equivalent:
    \begin{enumerate}[$(i)$]
        \item $h \in \ACp(\hatR)$,
        \item $C \circ h \circ C^{-1} \in \AC(\SSS)$, where $C: \H \rightarrow \D$ is the Cayley transform $C(z) = \frac{z-i}{z+i}$.
    \end{enumerate}
\end{lemma}  

\noindent In item $(ii)$, we could replace $\SSS$ with any circle $C = \partial D \subset \C$ and the Cayley transform by any \Mob transformation $S:\H \rightarrow D$.
\begin{proof}
        $(ii) \Rightarrow (i)$.  Suppose $H := C \circ h \circ C^{-1} \in \AC(\SSS)$.  By post-composing with a rotation $R$ of $\SSS$, we have that $R \circ H(1) = 1$, and therefore the map
        \begin{align*}
            C^{-1}\circ R \circ H \circ C = C^{-1} \circ R \circ C \circ h =: S \circ h   
        \end{align*}
        fixes $\infty$.  It thereby suffices to show $S\circ h \in \ACloc(\R)$.  And indeed, for $K \subset \R$ compact, $L:=R \circ H \circ C(K)$ is bounded away from 1, and so $C^{-1}$ is absolutely continuous on $L$ as a map from $(\SSS,d)$ to $(\R,|\cdot|)$, and hence $C^{-1}\circ R \circ H \circ C$ is a composition of monotone, absolutely-continuous functions on $K$.  We conclude $S \circ h \in \AC(K)$.

        \noindent $(i) \Rightarrow (ii)$. Here we have $S \circ h \circ T \in \ophom{\R}\cap \ACloc(\R)$ for some $S,T \in \PSL_2(\R)$, and thus $H := C\circ S \circ h \circ T \circ C^{-1} \in \ACloc(\SSS\bs \{1\})$.  As $H$ is monotone and continuous on $\SSS$, we thus see via Lemma \ref{Lemma:ACPtRemovable} that $H \in \AC(\SSS)$.  It follows that
        \begin{align*}
            C \circ h \circ C^{-1} = C \circ S^{-1} \circ C^{-1} \circ H \circ C \circ T^{-1} \circ C^{-1} = S_2 \circ H \circ T_2 \in \AC(\SSS),
        \end{align*}
        since $S_2,T_2 \in \Aut(\D)$ are M\"obius. 
    \end{proof}

\noindent Lemma \ref{Lemma:ACHatREquivalences} yields that $\ACp(\hatR)$ is closed under composition.  
\begin{cor}\label{Cor:ACHatRGroup}
    $\ACp(\hatR)$ is a monoid.  
\end{cor}
\begin{proof}
    It is clear that composition is associative and that there is an identity element, so we need to show that $h_1 \circ h_2 \in \ACp(\hatR)$, whenever both $h_1$ and $h_2$ are.  Applying the equivalences of Lemma \ref{Lemma:ACHatREquivalences}, we know $C\circ h_j \circ C^{-1}$, $j=1,2$, are monotone elements of $\AC(\SSS)$, and hence their composition $C \circ h_1 \circ h_2 \circ C^{-1}$ likewise is. Lemma \ref{Lemma:ACHatREquivalences} then yields the result.
\end{proof}

\subsection{Miscellaneous}
The following variant of the Dominated Convergence Theorem will be a convenient tool at several points in our argument.
\begin{prop}[Generalized Dominated Convergence~\cite{Follandbook}]\label{Thm: DCL}
If $f_n, g_n, f, g \in L^1$ satisfy $f_n \rightarrow f$ a.e., $g_n \rightarrow g$ a.e., $|f_n|\leq g_n$ a.e., and $\int g_n \rightarrow \int g$, then $\int f_n \rightarrow \int f$. 

\end{prop}

\section{Welding energies}\label{Sec:welding energy}
In this section we introduce a rooted M\"{o}bius-covariant welding energy $W_h(y)$, where $h: C \rightarrow C$ is a conformal welding defined on a circle $C \subset \hat{\C}$, and $y \in C$.  Up to affine changes of coordinates, all possible circles are covered by the bounded case $C = \partial \D = \SSS$ and the unbounded case $C = \hatR$, and thus our primary focus falls upon $C \in \{\SSS,\hatR\}$.  

Our construction of $W_h(y)$ proceeds in three stages.  First, we define a point-wise welding functional $L_h(\cdot, \cdot): C\times C \rightarrow (-\infty,\infty]$ in \S\ref{Sec:LFunctional} and examine its basic properties.  We take an $H^\half$ norm of $L_h$ in \S\ref{Sec:KFunctional} and study the resulting functional $K_h(y):= \| L_h(x,y) \|^2_{H^{1/2}(C),dx}$.  In \S\ref{Sec:W_h} we consider the more symmetric version $K_h(y) + K_{h^{-1}}(h(y))$, defining this to be $W_h(y)$.  We then prove the properties of $W_h(y)$ listed in the introduction.  

In \S\ref{Sec:WhyNotRootInvariant} we show by example that $y \mapsto W_h(y)$ is not constant, which is to say, the welding energy is not root-invariant.  In \S\ref{Sec:UpperLowerWeldingEnergy} we develop several variations of $W_h(y)$ that are, in effect, root invariant, as well as completely \Mob invariant.

\subsection{A point-wise welding functional}\label{Sec:LFunctional}

Let $C \in \{\SSS,\hatR\}$ and $h:C \rightarrow C$ a conformal welding, or more generally an element of $\ophom{C}$.  Motivated by the cross-ratio property \eqref{Eq:MobDifferenceQuotient} for \Mob transformations, we define
\begin{align}\label{Eq:L}
    L_h(x,y) := \log \bigg|\frac{(h(x) - h(y))^2}{h'(x)(x -y)^2} \bigg| \in (-\infty,\infty]
\end{align}
when $x\neq y \in \C$, $h'(x)$ exists, and both $h(x)$ and $h(y)$ are finite.  
We interpret $\log(\infty)$ as $\infty$, although this does not matter much, since for Weil--Petersson $h$ the derivative will vanish on a set of measure zero and $x$ will be a variable of integration.  We omit $h'(y)$ in the denominator to make $L_h(x,y)$ defined for all $y$, not just at those where $h'(y)$ exists, and thus $L_h(x,y)$ differs from $\log \Big|\frac{(h(x) - h(y))^2}{h'(x)h'(y)(x -y)^2} \Big|$ at generic points by a constant independent of $x$.  This difference will disappear when we hold $y$ fixed and take the $H^\half$-norm in $x$.

When $x=y \in \C$, we take a limit in \eqref{Eq:L} and set 
    \begin{align}\label{Eq:Lx=y}
        L_h(x,x) := \log|h'(x)|.
    \end{align}
In the unbounded case $C=\hatR$, we also wish to define $L_h(x,y)$ at any $y \in \widehat{\R}$, including when one or both of $y$ and $h(y)$ is infinite.  In words, the corresponding factor in \eqref{Eq:L} becomes unity when either $y$ or $h(y)$ are infinite.  That is:
\begin{itemize}
    \item   If $y \neq \infty = h(y)$, 
\begin{align}\label{Eq:Lhinfty}
    L_h(x,y) := \log \bigg|\frac{1}{h'(x)(x -y)^2} \bigg|.
\end{align}
    \item If $y = \infty \neq h(y)$, 
\begin{align}\label{Eq:Lyinfty}
    L_h(x,\infty) := \log \bigg|\frac{(h(x) - h(\infty))^2}{h'(x)} \bigg|.
\end{align}
    \item If $y = \infty = h(y)$, 
\begin{align}\label{Eq:Lhyinfty}
    L_h(x,\infty) := \log \bigg|\frac{1}{h'(x)} \bigg|.
\end{align}
\end{itemize}
In the these definitions we have continued assuming that $x \in \R$ is a generic point where $h(x) \in \R$ and $h'(x)$ exists. These seemingly-peculiar choices arise from \eqref{Eq:MobiusPreservesCR}.  For \eqref{Eq:Lhinfty}, for instance, send $z_4 \rightarrow y \neq \infty$.  Then if $T(y) = \infty$ and we fuse $z_2 \rightarrow y$ and $z_3 \rightarrow x = z_1$, we find
\begin{align*}
    \frac{C}{T'(x)(x-y)^2} = 1.
\end{align*}
Taking the logarithm and the $H^\half$-norm in $x$ makes $C$ irrelevant, and so we select 1 for convenience.  The other two expressions \eqref{Eq:Lyinfty} and \eqref{Eq:Lhyinfty} arise similarly.

As thus defined, our $L$ functional satisfies a type of chain rule, as stated in the following lemma.  While elementary, this will turn out to be vital to understanding how the welding energy interacts with composition, which in turn will be important for proving our main inequalities.

\begin{lemma}\label{Lemma:LGeneralComp}
    Let $C_j$ be circles in $\hatC$ and $f:C_1 \rightarrow C_2$ and $g:C_2 \rightarrow C_3$ be injective.  Then for all $y \in C_1$ and all $x\in C_1$ such that $x, f(x), g(f(x)) \in \C$ and such that both $f'(x)$ and $g'(f(x))$ exist and are non-zero,
    \begin{align}\label{Eq:LGeneralComp}
        L_{g \circ f}(x,y) = L_{g}\big(f(x),f(y)\big) + L_{f}(x,y).
    \end{align}
\end{lemma}

\begin{proof}
    When all the $C_j$ are circles in $\C$ (the bounded case), we immediately compute
    \begin{align*}
            L_{g \circ f}(x,y) &= \log \bigg| \frac{\big( g(f(x)) - g(f(y)) \big)^2}{g'(f(x))f'(x)(x-y)^2} \bigg|\\
            &= \log \bigg| \frac{\big( g(f(x)) - g(f(y)) \big)^2}{g'(f(x))( f(x)-f(y))^2} \bigg| + \log \bigg| \frac{( f(x) - f(y) )^2}{f'(x)(x-y)^2} \bigg|\\
            &= L_{g}\big(f(x),f(y)\big) + L_{f}(x,y),
        \end{align*}
    as claimed.  If some or all of the $C_j$ are unbounded, there are up to eight cases to verify, corresponding to the different definitions \eqref{Eq:Lhinfty}, \eqref{Eq:Lyinfty}, and \eqref{Eq:Lhyinfty} of $L$ for when infinity appears among $y, f(y)$, and $g(f(y))$.  The beauty of the definitions is that \eqref{Eq:LGeneralComp} still holds in every instance.  We do two cases to give a sense for the elementary argument, and leave the remainder for the interested reader. 
    \begin{itemize}
        \item $f(y), g(f(y)) \in \C$ but $y=\infty$.  Here \eqref{Eq:Lyinfty} says
        \begin{align*}
            L_{g \circ f}(x,y) &= \log \bigg| \frac{\big( g(f(x)) - g(f(y)) \big)^2}{g'(f(x))f'(x)} \bigg|\\
            &= \log \bigg| \frac{\big( g(f(x)) - g(f(y)) \big)^2}{g'(f(x))( f(x)-f(y))^2} \bigg| + \log \bigg| \frac{( f(x) - f(y) )^2}{f'(x)} \bigg|\\
            &= L_{g}\big(f(x),f(y)\big) + L_{f}(x,y)
            \end{align*}
        by \eqref{Eq:L} and \eqref{Eq:Lyinfty}.
    \item $g(f(y)) \in \C$ while $f(y) = y = \infty$.  We use \eqref{Eq:Lyinfty} and \eqref{Eq:Lhyinfty} to see
    \begin{align*}
            L_{g \circ f}(x,y) &= \log \bigg| \frac{\big( g(f(x)) - g(f(y)) \big)^2}{g'(f(x))f'(x)} \bigg|\\
            &= \log \bigg| \frac{\big( g(f(x)) - g(f(y)) \big)^2}{g'(f(x))} \bigg| + \log \bigg| \frac{1}{f'(x)} \bigg| = L_{g}\big(f(x),f(y)\big) + L_{f}(x,y).
            \end{align*}
    \end{itemize}
    The outstanding cases are similar rearrangements.
\end{proof}

The compositions we are interested in are changes of coordinates $S \circ h \circ T$ of weldings $h:C \rightarrow C$, where $S,T \in \PSL_2(\C)$ are M\"obius.  As above, $x$ continues to be a generic point at which all expressions involving it are both finite and well defined.

\begin{lemma}\label{Lemma:LMobiusCovariance}
    Let $C_j$ be circles in $\hatC$, $h:C_2 \rightarrow C_2$ injective, and $S,T \in \PSL_2(\C)$ such that $T:C_1 \rightarrow C_2$.  Let $x \in C_1$ such that $x, T(x), h(T(x))$, and $S(h(T(x))$ all belong to $\C$, and such that $T'(x), h'(T(x)),$ and $S'(h(T(x)))$ all exist and are non-zero.  Then for any $y \in C_1$,
    \begin{align}\label{Eq:LMobiusCovariance}
        L_{S\circ h \circ T}(x,y) = A(S,y) + L_h(T(x),T(y)) + A(T,y),
    \end{align}
    where $A(R,y)$ is a constant that only depends on $R$ and $y$.
\end{lemma}
\noindent The utility of the lemma for us is that, upon taking the $H^\half$-norm in $x$, the constants vanish, leaving only the $L_h(T(x),T(y))$ term.

\begin{proof}
    Lemma \ref{Lemma:LGeneralComp} immediately yields
    \begin{align*}
        L_{S\circ h \circ T}(x,y) = L_S\big( h(T(x)),h(T(y)) \big) + L_h\big( T(x),T(y) \big) + L_T(x,y).
    \end{align*}
Consider the last term.  When $y, T(y) \in \C$, $L_T(x,y) = \log|T'(y)|$ by \eqref{Eq:L} and \eqref{Eq:MobDifferenceQuotient}. There are three other cases to consider, based on whether one or both of $y$ and $T(y)$ is infinite.  Write $T(z) = \frac{az+b}{cz+d}$ where $ad-bc=1$.
\begin{itemize}
    \item $y=\infty \neq T(y)$.  Using \eqref{Eq:Lyinfty} and computing, we find
    \begin{align*}
        L_T(x,y) = \log\bigg| \frac{(T(x)-T(\infty))^2}{T'(x)} \bigg| = -2\log|c|.
    \end{align*}
    \item $y \neq \infty = T(y)$.  Here we may write $T(z) = \frac{az+b}{c(z-y)}$, where $c=-1/(ay+b)$.  Using \eqref{Eq:Lhinfty}, we thus see
    \begin{align*}
        L_T(x,y) = \log\bigg| \frac{1}{T'(x)(x-y)^2} \bigg| = 2\log|c|.
    \end{align*}
    \item $y = \infty = T(y)$.  Here $T(z) = z+b$, and by \eqref{Eq:Lhyinfty} we have $L_T(x,y) = -\log|T'(x)| = 0$.
\end{itemize}
Hence in all cases we obtain a constant only depending on $T$ and $y$. Since $S$ is likewise M\"obius, this analysis also applies to $L_S( h(T(x)),h(T(y)) )$.
\end{proof}

\subsection{A rooted welding functional}\label{Sec:KFunctional}

We define a rooted welding functional $K_h(y)$ as the $H^{\half}$-norm of $L_h(\cdot,y)$.  We recall $L_h$ is defined above in  \S\ref{Sec:LFunctional}.
\begin{definition}\label{Def:WeldingK}
    For $C$ a circle in $\hatC$, $h \in \ophom{C}$ and $y\in C$, we set
    \begin{align*}
        K_h(y) := \big\| L_h(\cdot, y) \big\|_{H^{1/2}(C)}^2
    \end{align*}
    if $h \in \AC(C)$ when $C \subset \C$ or $h \in \ACp(C)$ when $C$ is unbounded, and $K_h(y) := \infty$ otherwise.  
\end{definition}
\noindent We recall Definition \ref{Def:ACOnHatR} for $\ACp(\hatR)$, with $\ACp(C)$ defined analogously for other unbounded $C$.  Again, the two cases of interest for us are $C = \SSS$ and $C = \hatR$.

As an example, if $C = \hatR$ and $h \in \ACp(\hatR)$ such that $h(\infty)=\infty$, rooting at $y=\infty$ yields $$K_h(\infty) = \big\| \log|h'| \big\|_{H^{1/2}(\R)}^2$$ by \eqref{Eq:Lhyinfty}.  We will often abbreviate the $H^{\half}(C)$ semi-norm by $\|\cdot \|_{\half(C)}$ or simply $\|\cdot \|_{\half}$ when there is no confusion about the circle in question.

The \Mob covariance of Lemma \ref{Lemma:LMobiusCovariance} translates to $K_h(\cdot)$ as follows, which will prove to be very useful for us.

\begin{thm}\label{Thm:KMobiusInvariance}
    Let $C_j$ be circles in $\hatC$, $h \in \ophom{C_2}$, and $S,T \in \PSL_2(\C)$ such that $T:C_1 \rightarrow C_2$. Then for any $y \in C_1$,
    \begin{align*}
        K_{S\circ h \circ T}(y) = K_{h}(T(y)).
    \end{align*}
\end{thm}

\begin{proof}
    Using Lemma \ref{Lemma:LMobiusCovariance} we see
    \begin{align*}
        K_{S\circ h \circ T}(y) = \big\| A(S,y) + L_h(T(\cdot), T(y)) + A(T,y) \big\|_{H^{1/2}(C_1)}^2 = \big\| L_h(T(\cdot), T(y)) \big\|_{H^{1/2}(C_1)}^2
    \end{align*}
    since $y$ is fixed. Writing $L(\cdot) := L_h(\cdot, T(y))$ and recalling \eqref{Eq:MobDifferenceQuotient}, we have
    \begin{align*}
        \big\| L_h(T(\cdot), T(y)) \big\|_{H^{1/2}(C_1)}^2 &= \frac{1}{4\pi^2}\int_{C_1}\int_{C_1} \frac{|L(T(u))-L(T(v))|^2}{|u-v|^2} |\dd u||\dd v|\\
        &= \frac{1}{4\pi^2}\int_{C_1}\int_{C_1} \frac{|L(T(u))-L(T(v))|^2}{|T(u)-T(v)|^2} |T'(u)||\dd u||T'(v)||\dd v|\\
        &= \frac{1}{4\pi^2}\int_{C_2}\int_{C_2} \frac{|L(u)-L(v)|^2}{|u-v|^2} |\dd u||\dd v|\\
        &= \big\| L_h(\cdot, T(y)) \big\|_{H^{1/2}(C_2)}^2 = K_h(T(y)). \qedhere
    \end{align*}
\end{proof}

In the remainder of this section, for $y \in \hatR$ we set $T_y$ to be any fixed element of $\PSL_2(\R)$ which satisfies 
\begin{align}\label{Eq:T_y}
    T_y(\infty) = y.
\end{align}
The two degrees of freedom in the choice of $T_y$ are immaterial for us. The \Mob covariance of Theorem \ref{Thm:KMobiusInvariance} along with Proposition \ref{IntroProp:ShenWeldingH} yield the following.

\begin{thm}\label{Thm:KFiniteEquiv}
     Let $C$ be a circle in $\hatC$ and $h \in \ophom{C}$ a conformal welding for a Jordan curve $\gamma \subset \hatC$.  The following are equivalent:
    \begin{enumerate}[$(a)$]
        \item $I^L(\gamma)<\infty$.
        \item There exists $y \in C$ such that $K_h(y) < \infty$.
        \item For every $y \in C$, $K_h(y) < \infty$.
    \end{enumerate}
\end{thm}

\begin{proof}
We show $(b) \Rightarrow (a) \Rightarrow (c) \Rightarrow (b)$, first assuming $C= \hatR$.

\noindent $(b) \Rightarrow (a)$.  Suppose $K_h(y) < \infty$ for some $y \in \wR$.  If $y=\infty$ then $\|\log|h'|\|_{\half} < \infty$, showing $I^L(\gamma) < \infty$ by Proposition \ref{IntroProp:ShenWeldingH}.  If $y \in \R$, set $H:= T_{h(y)}^{-1}\circ h \circ T_y$, and use Theorem \ref{Thm:KMobiusInvariance} to note
    \begin{align*}
        K_H(\infty) = K_{T_{h(y)}^{-1} \circ h \circ T_y}(\infty) = K_h(y) < \infty
    \end{align*}
    by assumption.  As $H \in \ophom{\R}$ is a conformal welding for $\gamma$, $I^L(\gamma) < \infty$ by Proposition \ref{IntroProp:ShenWeldingH}.

    \noindent $(a) \Rightarrow (c)$.  Choose $y \in \hatR$ and again consider $H:= T_{h(y)}^{-1}\circ h \circ T_y$, a welding for $\gamma$ which fixes $\infty$. Our assumption, Proposition \ref{IntroProp:ShenWeldingH}, and Theorem \ref{Thm:KMobiusInvariance} yield
    \begin{equation*}
        \infty > K_H(\infty) = K_{T_{h(y)}^{-1} \circ h \circ T_y}(\infty) = K_h(y). 
    \end{equation*}
    The last implication $(c) \Rightarrow (b)$ is trivial.
    
    If $C \neq \hatR$, take $T \in \PSL_2(\C)$ such that $T:\hatR \rightarrow C$.  Then for $y \in C$, by Theorem \ref{Thm:KMobiusInvariance}
    \begin{align}\label{Eq:hMobConjugate}
        K_h(y) = K_{T^{-1}\circ h \circ T}(T^{-1}(y)),
    \end{align}
    and so the equivalences in this case immediately follow from the above.
\end{proof}
As we noted in the introduction in \eqref{Eq:WOnQuotientSpace} for the welding energy $W_h(y)$ (formally defined below), Theorems \ref{Thm:KMobiusInvariance} and \ref{Thm:KFiniteEquiv} yield that, for fixed $y \in C$, 
\begin{align*}
    K_{(\cdot)}(y): \MOB(C)\bs \WP(C) / \MOB(C,y) \rightarrow [0,\infty).
\end{align*}

\begin{thm}\label{Thm:KVanishesMobius}
    Let $C$ be a circle in $\hatC$ and $h \in \ophom{C}$.  There exists $y \in C$ such that $K_h(y)=0$ if and only if $h$ is a \Mob transformation.  In this case, $K_h(y) = 0$ for all $y \in C$.
\end{thm}

\begin{proof}
    As we have \eqref{Eq:hMobConjugate} for any $T \in \PSL_2(\C)$, we may assume $C=\hatR$. 
    
    Suppose $K_h(y) =0$ for some $y$.  By Theorem \ref{Thm:KMobiusInvariance}, the welding $H:= T_{h(y)}^{-1} \circ h \circ T_y$ satisfies $K_H(\infty) = K_h(y)=0$.  We conclude that $\log |H'(x)|$, and therefore $H'(x)$ or $-H'(x)$ itself ($H$ is monotone), is almost everywhere a constant.  Since $H$ is locally absolutely continuous, we integrate and find $H$ is an affine function.  Hence $h = T_{h(y)} \circ H \circ T_y^{-1}$ is M\"{o}bius.  

    Conversely, if $h$ is M\"{o}bius, Theorem \ref{Thm:KMobiusInvariance} yields $K_h(y) = K_{h \circ \text{id}}(y) = K_{\text{id}}(y) = 0$.
    \end{proof}

\noindent Although we defer the proof until Section~\ref{Sec: cor1}, we note the continuity of the welding functional in the root.
\begin{thm}\label{Thm: KisContinuousatRoot}
    For $C \subset \hatC$ a circle and $h \in \ophom{C}$, $y \mapsto K_h(y)$ is continuous on $C$.
\end{thm}

\noindent Lastly, we use Corollary \ref{Cor:OperatorBound} and Lemma \ref{Lemma:LGeneralComp} to control the welding functional of a composition.

\begin{lemma}\label{Lemma:KGeneralComp}
    Let $C$ be a circle in $\hatC$ and $h_1, h_2 \in \ophom{C}$, with $h_2$ the conformal welding of a $K_2$-quasicircle.  Then for all $y \in C$,
    \begin{align}\label{Ineq:KGeneralComp}
        K_{h_1 \circ h_2}(y) \leq 2(K_2^2 + K_2^{-2})K_{h_1}(h_2(y)) + 2K_{h_2}(y).
    \end{align}
\end{lemma}

\begin{proof}
    If either of $h_1$ and $h_2$ do not weld Weil--Petersson quasicircles $\gamma_1$ and $\gamma_2$, the right-hand side of \eqref{Ineq:KGeneralComp} is $+\infty$ by Theorem \ref{Thm:KFiniteEquiv} and the inequality is trivial.  Thus we may assume $\max\{I^L(\gamma_1), I^L(\gamma_2)\} < \infty$.  
    
    As we wish to apply Corollary \ref{Cor:OperatorBound} to composition by $h_2$, we first suppose that $C = \hatR$  and that $h_2(\infty) = \infty$.  By changing coordinates for $h_1$ via $S \circ h_1 \circ T =:H_1$, we have that $H_1(\infty) = \infty$ as well, and thus both $H_1$ and $h_2$ are locally absolutely continuous on $\R$ by Proposition \ref{IntroProp:ShenWeldingH}.  Thus $h_1 \in \ACp(\hatR)$ (recall Definition \ref{Def:ACOnHatR}), and so $h_1 \circ h_2 \in \ACp(\hatR)$ as well by Corollary \ref{Cor:ACHatRGroup}. Applying Lemma \ref{Lemma:LGeneralComp} and Corollary \ref{Cor:OperatorBound} then yields
    \begin{align*}
       K_{h_1 \circ h_2}(y) &= \big\| L_{h_1}\big(h_2(\cdot),h_2(y)\big) + L_{h_2}(\cdot, y) \big\|^2_{\half(\R)}\\
       &\leq 2 \big\| L_{h_1}\big(h_2(\cdot),h_2(y)\big) \big\|^2_{\half(\R)} + 2\| L_{h_2}(\cdot, y) \|^2_{\half(\R)}\\
       &\leq 2 \|\calP_{h_2}\|^2 \| L_{h_1}(\cdot,h_2(y)) \|^2_{\half(\R)} + 2K_{h_2}(y)\\
       &\leq 2(K_2^2 + K_2^{-2}) K_{h_1}(h_2(y)) + 2K_{h_2}(y). 
    \end{align*}
    If $h_2(\infty) = x \in \R$, we consider $T_x^{-1}\circ h_2$ with $T_x$ as in \eqref{Eq:T_y}. Then the above along with Theorem \ref{Thm:KMobiusInvariance} yield
    \begin{align*}
        K_{h_1 \circ h_2}(y) = K_{(h_1 \circ T_x) \circ (T_x^{-1} \circ h_2)}(y) &\leq 2(K_2^2 + K_2^{-2}) K_{h_1\circ T_x}\big(T_x^{-1} \circ h_2(y) \big) + 2K_{T_x^{-1} \circ h_2}(y)\\
        &= 2(K_2^2 + K_2^{-2}) K_{h_1}(h_2(y)) + 2K_{h_2}(y).
    \end{align*}

    In the case of a bounded circle $C$, we have each $h_j \in \AC(C)$, and thus the conjugates $T^{-1} \circ h_j \circ T =:H_j \in \ACp(\hatR)$, where $T\in\PSL_2(\C)$ takes $\hatR$ to $C$.  Thus by Theorem \ref{Thm:KMobiusInvariance} and the above, 
    \begin{align*}
        K_{h_1 \circ h_2}(y) = K_{H_1 \circ H_2}(T^{-1}(y)) &\leq 2(K_2^2 + K_2^{-2}) K_{H_1}\big(H_2(T^{-1}(y))\big) + 2K_{H_2}(T^{-1}(y))\\
        &= 2(K_2^2 + K_2^{-2}) K_{h_1}(h_2(y)) + 2K_{h_2}(y). \qedhere
    \end{align*}
\end{proof}

\subsection{The rooted welding energy}\label{Sec:W_h}

The \emph{rooted welding energy} $W_h(\cdot)$, a more symmetric version of $K_h(\cdot)$, will be one of our main tools.  We recall $L_h$ is defined in \S\ref{Sec:LFunctional}, and $K_h$ in Definition \ref{Def:WeldingK}.  
\begin{definition}
    For $C$ a circle in $\hatC$, $h \in \ophom{C}$ and $y\in C$, the \emph{(boundary-normalized) welding energy of $h$ rooted at $y$} is
    \begin{align}\label{Def:W}
        W_h(y) := K_h(y) + K_{h^{-1}}(h(y)) = \big\| L_h(\cdot, y) \big\|_{H^{1/2}(\R)}^2 + \big\| L_{h^{-1}}(\cdot, h(y)) \big\|_{H^{1/2}(\R)}^2
    \end{align}
     if $h \in \AC(C)$ when $C \subset \C$ or $h \in \ACp(C)$ when $C$ is unbounded, and $K_h(y) := \infty$ otherwise.  
\end{definition}
\noindent We recall Definition \ref{Def:ACOnHatR} for $\ACp(\hatR)$, with $\ACp(C)$ defined analogously for other unbounded $C$.  Again, the two cases of interest for us are $C = \SSS$ and $C = \hatR$.  

As an example, if $C = \hatR$ and $h \in \ACp(\hatR)$ such that $h(\infty)=\infty$,
\begin{align*}
    W_h(\infty) = \big\| \log |h'| \big\|_{\half(\R)}^2 + \big\| \log     |(h^{-1})'| \big\|_{\half(\R)}^2.
\end{align*}
Note that $W_h(y) = W_{h^{-1}}(h(y))$.

The properties of $W_h(y)$ stated in the introduction are now simple consequences of the corresponding properties of $K_h(y)$ proved above in \S\ref{Sec:KFunctional}.

\begin{proof}[Proof of Theorem \ref{Thm:WMobiusInvariance}]
    We have
    \begin{multline*}
        W_{S\circ h \circ T}(y) = K_{S\circ h \circ T}(y) + K_{{T^{-1}\circ h^{-1} \circ S^{-1}}}({S\circ h \circ T}(y))\\
        = K_h(T(y)) + K_{h^{-1}}(h(T(y))) = W_h(T(y))
    \end{multline*}
    by Theorem \ref{Thm:KMobiusInvariance}.
\end{proof}


\begin{proof}[Proof of Theorem \ref{Thm:WFiniteEquiv}]
    Since $\gamma$ has finite Loewner energy if and only if its complex-conjugate curve $\gamma^*$ does, Theorem \ref{Thm:KFiniteEquiv} shows $K_h(y) < \infty$ if and only if $K_{h^{-1}}(h(y)) < \infty$.  The result then immediately follows from Theorem \ref{Thm:KFiniteEquiv}.
\end{proof}

\begin{proof}[Proof of Theorem \ref{Thm:WVanishesMobius}]
    This immediately follows from the definition of $W_h$ and Theorem \ref{Thm:KVanishesMobius}.
\end{proof}

While we are still deferring the proof of Theorem \ref{Thm: KisContinuousatRoot} on the continuity of $y \mapsto K_h(y)$ until \S\ref{Sec:WContinuous}, we note that Theorem \ref{Thm: WisContinuousatRoot} is a clear consequence.
\begin{proof}[Proof of Theorem \ref{Thm: WisContinuousatRoot}]
    This immediately follows from Theorem~\ref{Thm: KisContinuousatRoot}, the continuity of $h^{-1}$, and the definition \eqref{Def:W} of $W_h(y)$.
\end{proof}

\noindent Lastly, we adapt Lemma \ref{Lemma:KGeneralComp} to the setting of the welding energy.  This will be a key ingredient to proving our main inequalities, Theorems \ref{Thm:GeneralCompBound} and \ref{Thm:LEGrowth}.
\begin{lemma}\label{Lemma:WGeneralComp}
    Let $C$ be a circle in $\hatC$ and $h_1, h_2 \in \ophom{C}$ such that $h_1,h_2 \in \AC(C)$ if $C \subset \C$ is a bounded circle, or $h_1,h_2 \in \ACp(C)$ if $\infty \in C$.  Further suppose that each $h_j$ welds a $K_j$-quasicircle.  Then for all $y \in C$,
    \begin{align}\label{Ineq:WGeneralCompDiffK}
        W_{h_1 \circ h_2}(y) \leq 2(K_2^2 + K_2^{-2}) W_{h_1}(h_2(y)) +  2(K_1^2 + K_1^{-2})W_{h_2}(y).
    \end{align}
    In particular, if both $h_j$ weld $K$-quasicircles,
    \begin{align}\label{Ineq:WGeneralCompSameK}
        W_{h_1 \circ h_2}(y) \leq 2(K^2 + K^{-2})\big( W_{h_1}(h_2(y)) + W_{h_2}(y)  \big).
    \end{align}
\end{lemma}
\begin{proof}
    Using the definition \eqref{Def:W} and Lemma \ref{Lemma:KGeneralComp}, we have
    \begin{align*}
        W_{h_1 \circ h_2}(y) &= K_{h_1 \circ h_2}(y) + K_{h_2^{-1}\circ h_1^{-1}}(h_1\circ h_2(y))\\
        &\leq 2(K_2^2+K_2^{-2})K_{h_1}(h_2(y))  + 2K_{h_2}(y)\\
        & \hspace{28mm} + 2(K_1^2+K_1^{-2})K_{h_2^{-1}}(h_2(y)) + 2K_{h_1^{-1}}(h_1\circ h_2(y)),
    \end{align*}
    where we are using the fact that $h_1^{-1}$ is a welding of the complex-conjugated curve $\gamma_1^*$, which is still a $K_1$-quasicircle.  Collecting the $h_1$ and $h_2$ terms, we have
    \begin{align*}
        W_{h_1\circ h_2}(y) &\leq 2(K_2^2+K_2^{-2})\big( K_{h_1}(h_2(y)) + K_{h_1^{-1}}(h_1\circ h_2(y)) \big)\\
        & \hspace{35mm} + 2(K_1^2+K_1^{-2})\big( K_{h_2}(y) + K_{h_2^{-1}}(h_2(y))\big)\\
        &= 2(K_2^2+K_2^{-2})W_{h_1}(h_2(y)) + 2(K_1^2+K_1^{-2})W_{h_2}(y). \qedhere
    \end{align*}
\end{proof}

\subsection{Two numerical examples}\label{Sec:WhyNotRootInvariant}


Given that the welding energy is defined \eqref{Def:W} in terms of a root $y$, and that the explicit expression of the energy varies depending on the choice of $y$ (recall \eqref{Eq:L}, \eqref{Eq:Lx=y}, \eqref{Eq:Lhinfty}, \eqref{Eq:Lyinfty}, and \eqref{Eq:Lhyinfty}), it is natural to ask how $W_h(y)$ varies in $y$.  Thus far, all we know is that $y \mapsto W_h(y)$ is continuous, Theorem \ref{Thm: WisContinuousatRoot}.  We might wonder whether this function is, in fact, constant.  

Recall that this is the case with the Loewner energy.  As originally defined for Jordan curves $\gamma \subset \hatC$, the Loewner energy is \emph{rooted} at a point $\gamma(0)$ on the curve $\gamma$ \cite[Prop. 3.5]{Loopenergy}.  After defining it as such, however, Rohde and Wang prove the deep result that the energy is, in fact, root independent \cite[Thm. 1.2]{Loopenergy}.  

The situation with $W_h(\cdot)$ appears more complicated, and our first example gives strong evidence that $y \mapsto W_h(y)$ is not generally a constant function.

\begin{example}\label{Eg:WNotConstant}
    Consider the piecewise-\Mob welding $h: \hatR \rightarrow \hatR$ defined by 
    \begin{align*}
        h(x) := \begin{cases}
            7x & x \leq 0,\\
            \frac{7x}{6x+1} & 0 < x \leq 1,\\
            \frac{2x-9}{3x-10} & 1 < x \leq 3,\\
            7x-18 & 3 < x.
        \end{cases}
    \end{align*}
    See Figure \ref{Fig:WNotConstantEg1}.  Using \eqref{Eq:MobDifferenceQuotient}, one sees that there is a one-parameter family of $C^{2-\epsilon}$ piecewise-\Mob maps which fix $0,1,3,$ and $\infty$, with $h$ being the unique member satisfying $h'(\infty)=7$.  According to \cite[\S4.1, Cor. 4.1]{PiecewiseGeodesic}, $h$ is the welding of a Loewner energy minimizer through four points on $\hatC$, as also illustrated in Figure \ref{Fig:WNotConstantEg1}.   

    \begin{figure}
        \includegraphics[scale=0.65]{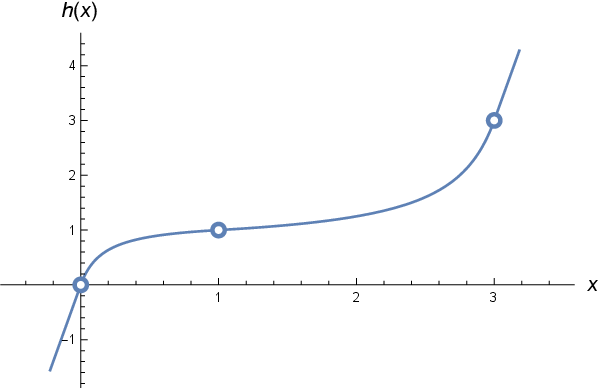} \qquad 
        \includegraphics[scale=0.6]{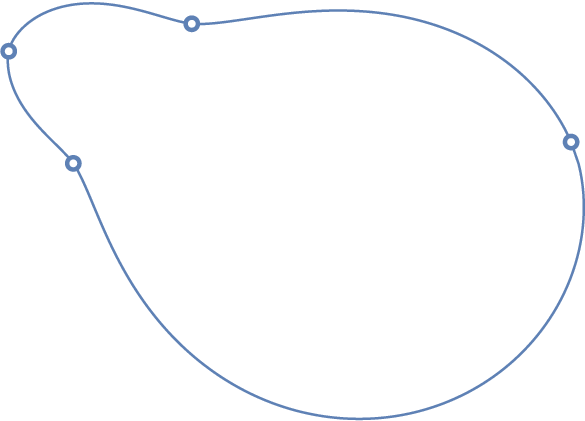}
        \caption{\small On the left, an explicit piecewise-\Mob welding $h \in C^{2-\epsilon}$ for which $y \mapsto W_h(y)$ is not numerically constant;    see Example \ref{Eg:WNotConstant} and Figure \ref{Fig:WNotConstantEg2}. On the right we simulate the welded Jordan curve $\gamma$, created with the help of Don Marshall's zipper algorithm software \cite{ZipperWebsite}.}
        \label{Fig:WNotConstantEg1}
    \end{figure}

    Recalling \eqref{Eq:L} for the definition of $L_h$, we use Mathematica to numerically integrate
    \begin{multline*}
        y \mapsto W_h(y) = \frac{1}{4\pi^2}\int_\R \int_\R \frac{\big(L_h(u,y)-L_h(v,y) \big)^2}{(u-v)^2}\,\dd u\,\dd v\\ + \frac{1}{4\pi^2}\int_\R \int_\R \frac{\big(L_{h^{-1}}(u,h(y))-L_{h^{-1}}(v,h(y)) \big)^2}{(u-v)^2}\,\dd u\,\dd v
    \end{multline*}
    over a mesh in $-1.5 \leq y \leq 5.5$ of width $\Delta y= 0.05$, and display the results in Figure \ref{Fig:WNotConstantEg2}.  Mathematica gives a warning for slow convergence of the numerical integral for five of the 141 values of $y$; these points are highlighted in the figure.  Overall, the discrepancy between the extremes of $W_h(y)$ seems far too large to attribute exclusively to numerical error.  We conclude that there is therefore compelling evidence for $W_h(\cdot)$'s root dependence.  

    Recalling Definition \ref{Def:UpperLowerW}, we have by Theorem \ref{Comparable} that the Loewner energy is dominated by $2/3$ of the lower welding energy, $I^L(\gamma) \leq \frac{2}{3}\underline{W_h} \approx 2.962$.  Allowing for some numerical error, we therefore have 
    \begin{align*}
        I^L(\gamma) \leq 2.97.
    \end{align*}
    We do not know of many other explicit computations or estimates of Loewner energy in the literature.\footnote{An interesting instance of several explicit estimates is in \S23 of an earlier version of Bishop's paper \cite{Bishop2019weil}, which is available on \url{https://www.math.stonybrook.edu/~bishop/papers/}.  See the ``earlier version (Nov 2020)'' of \emph{Weil-Petersson curves, beta-numbers, and minimal surfaces} there.}

    We discuss some further details of this example below in Example \ref{Eg:WNotConstant2}.  See also Example \ref{Eg:WConjugatedtoS1} for the energy of $h$ after a change of coordinates.

\begin{figure}[ht]
        \centering
        \includegraphics[scale=0.8]{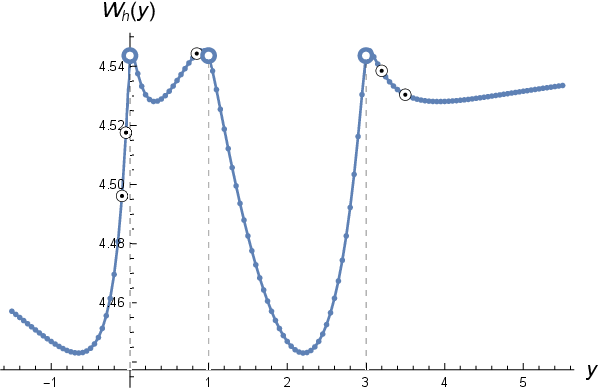}
        \caption{\small Numerical estimation of the function $y \mapsto W_h(y)$ for $h$ given in Example \ref{Eg:WNotConstant}.  Values of $W_h(y)$ at the glue points $y \in \{0,1,3\}$ are highlighted.  We also use ``$\odot$" to highlight the five points $y \in \{ -0.1,-0.05,0.85,3.2,3.5\}$ where Mathematica issues a warning for numerical convergence.  In each case, the  message is the same ``NIntegrate::slwcon'' warning about slow convergence of the numerical integral.  Although lacking an explanation for these errors, we conclude this is strong evidence of the non-constant nature of $y \mapsto W_h(y)$.}
        \label{Fig:WNotConstantEg2}
    \end{figure}

\end{example}
The one case where we know $y \mapsto W_h(y)$ is constant is when $W_h(y) \equiv 0$ for the welding $h$ of a circle, Theorem \ref{Thm:WVanishesMobius}.  Are there $h$ with $W_h(y)>0$ for which the function is likewise constant?  Translating to the energy gap $\Delta W_h$ of Definition \ref{Def:UpperLowerW}, this is equivalent to asking if $\Delta W_h$ vanishes for any non-\Mob $h$.  We anticipate a negative answer.

\begin{con}\label{Conj:EnergyGap}
    If $\Delta W_h = 0$, then $h$ is a welding for a circle.  
\end{con}

\begin{figure}
        \centering
        \includegraphics[scale=0.65]{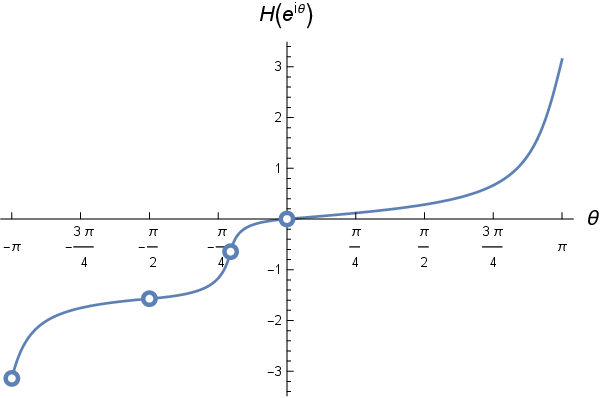} \qquad \includegraphics[scale=0.65]{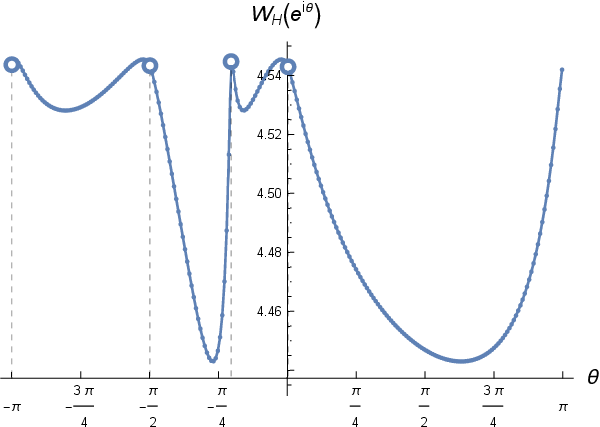}
        \caption{\small 
        On the left, the conjugation $H$ of the welding $h$ of Example \ref{Eg:WNotConstant} to $\SSS$.   On the right, the welding energy function $y \mapsto W_H(y)$ on $\SSS$.  The corresponding welded curve $\gamma$, of course, is still as in Figure \ref{Fig:WNotConstantEg1}.   Comparing with Figure \ref{Fig:WNotConstantEg2}, the interval $[-\pi,-\pi/2]$ here corresponds to $[0,1]$ there, with subsequent intervals likewise matched according to orientation. Mathematica had much more difficulty with the integral on $\SSS\times \SSS$ compared to on $\R \times \R$; here some 86\% of the points generated numerical integration error warnings (and thus we do no highlight all of them, as in Figure \ref{Fig:WNotConstantEg2}).
        }
        \label{Fig:hConjugated}
    \end{figure}
For our second example we change coordinates on welding of Example \ref{Eg:WNotConstant} to illustrate the \Mob covariance of Theorem \ref{Thm:WMobiusInvariance}. 
\begin{example}\label{Eg:WConjugatedtoS1}
     Consider the conjugation $H := C \circ h \circ C^{-1} \in \ophom{\SSS}$ of $h$ from Example \ref{Eg:WNotConstant} under the Cayley transform $C(z) = \frac{z-i}{z+i}$, as pictured on the left of Figure \ref{Fig:hConjugated}.  We again use Mathematica to numerically integrate
    \begin{multline*}
        y \mapsto W_H(y) = \frac{1}{4\pi^2}\int_\SSS \int_\SSS \frac{\big|L_h(u,y)-L_h(v,y) \big|^2}{|u-v|^2}\,|\dd u|\,|\dd v|\\ + \frac{1}{4\pi^2}\int_\SSS \int_\SSS \frac{\big|L_{h^{-1}}(u,h(y))-L_{h^{-1}}(v,h(y)) \big|^2}{|u-v|^2}\,|\dd u|\,|\dd v|
    \end{multline*}
    over a fine mesh on the circle, and display the results on the right of Figure \ref{Fig:hConjugated}.  Comparing with Figure \ref{Fig:WNotConstantEg2}, we see the invariance of the values of $W_H(y)$ and the covariance in the root expressed in Theorem \ref{Thm:WMobiusInvariance}.
\end{example}

\subsection{M\"obius-invariant welding energies}\label{Sec:UpperLowerWeldingEnergy}

Theorem \ref{Thm: WisContinuousatRoot} leads to three M\"obius-invariant welding energies $\overline{W_h}, \underline{W_h}$, and $\Delta W_h$, defined above in Definition \ref{Def:UpperLowerW}.  See also Corollary \ref{Cor:UpperLowerWProps} for basic properties of this trio.  Here we note  that Lemma \ref{Lemma:WGeneralComp} translates to $\overline{W_h}$ as follows.

\begin{cor}\label{Cor:UpperWGeneralComp} 
    Let $C$ be a circle in $\hatC$ and $h_1, h_2 \in \ophom{C}$ such that $h_1,h_2 \in \AC(C)$ if $C \subset \C$ is a bounded circle, or $h_1,h_2 \in \ACp(C)$ if $\infty \in C$.  Further suppose that each $h_j$ welds a $K$-quasicircle.  Then
        \begin{align*}
            \overline{W_{h_1\circ h_2}} \leq 2(K^2+K^{-2})\big( \overline{W_{h_1}} + \overline{W_{h_2}}  \big).
        \end{align*}
\end{cor}

\begin{example}\label{Eg:WNotConstant2}
    Revisiting the piecewise-\Mob welding of Example \ref{Eg:WNotConstant}, we might guess from Figure \ref{Fig:WNotConstantEg2} that $\upcalW{h} = W_h(y)$ for any glue-point $y \in \{0,1,3,\infty\}$.  However, zooming in around these $y$-values shows that the welding energy actually peaks nearby the glue-point, but not at it.  See Figure \ref{Fig:WNotConstantZoom}.
    
    \begin{figure}
        \includegraphics[scale=0.65]{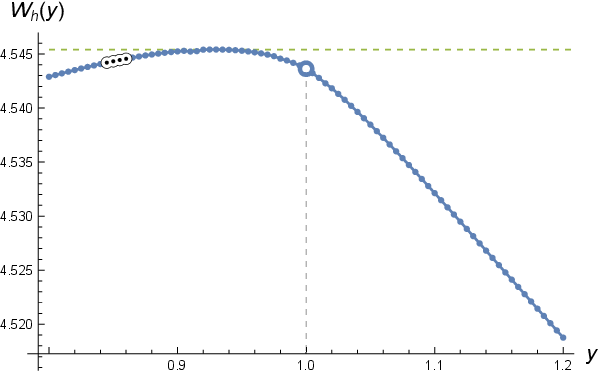} \qquad 
        \includegraphics[scale=0.65]{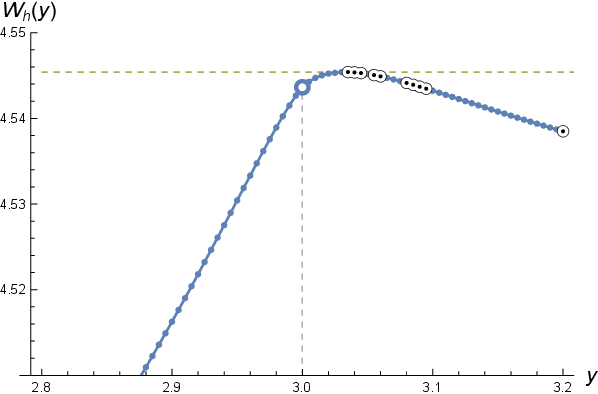}
        \caption{\small Zooming in around the glue points of the piecewise-\Mob welding of Example \ref{Eg:WNotConstant}, we see that $W_h(y)$ does not peak at the glue-points themselves, but only nearby.  The picture around $y=0$ is similar, with the max occurring to the right of $y=0$.  Numerically, each of these local maxima are the same to three decimal places.  (Here we follow Figure \ref{Fig:WNotConstantEg2} and use the nested circles $\odot$ to show points for which Mathematica issued the ``NIntegrate::slwcon" warning.  In the left plot there were four such (consecutive) points, and a total of ten such points in the right.)  }
        \label{Fig:WNotConstantZoom}
    \end{figure}
\end{example}

\section{Normalized Loewner Energy Formulas}\label{Sec: BoundayNormalizedforLE}
In this section we prove Theorems~\ref{Thm:BoundaryLE} and \ref{Thm:InteriorLE}, giving generalized Loewner energy formulas. 

\subsection{The normalized pre-Schwarzian derivative and the boundary-normalized formula}

We first prove Theorem \ref{Thm:BoundaryLE} in the case $D = \H$, in which case the statement becomes the following.

\begin{lemma}\label{Lemma:BoundaryLE1}
    Let $\gamma \subset \widehat{\C}$ be any Jordan curve, and $f:\H \rightarrow \Omega$ and $g: \H^* \rightarrow \Omega^*$ any conformal maps to the two complementary components of $\widehat{\C}\bs \gamma$.  Then for any $x \in \partial \H$ with $y := g^{-1}\circ f(x)$,
    \begin{multline*}
      I^L(\gamma) =\frac{4}{\pi}\int_{\H} \left| \frac{f'(z)}{f(z)-f(x)} - \frac{1}{z-x} - \frac{1}{2} \frac{f''(z)}{f'(z)} \right|^2 \dd A(z)\\  + \frac{4}{\pi}\int_{\H^*} \left| \frac{g'(z)}{g(z)-g(y)} - \frac{1}{z-y} - \frac{1}{2} \frac{g''(z)}{g'(z)} \right|^2 \dd A(z)
    \end{multline*}
    when each of $x, y, f(x),$ and $f(y)$ is finite.  If any of these are infinite, the corresponding term in the integrand disappears.
\end{lemma}
\noindent Here we are taking the boundary of $\H$ with respect to $\hatC$, and so $x$ or $y$ (or both) could be infinite.  For example, if $x = \infty$, $f(x) \in \C$, while $y \in \R$ and $g(y) = \infty$, the formula becomes
\begin{align*}
    I^L(\gamma) =\frac{4}{\pi}\int_{\H} \left| \frac{f'(z)}{f(z)-f(\infty)}  - \frac{1}{2} \frac{f''(z)}{f'(z)} \right|^2 \dd A(z)  + \frac{4}{\pi}\int_{\H^*} \left|  \frac{1}{z-y} + \frac{1}{2} \frac{g''(z)}{g'(z)} \right|^2 \dd A(z).
\end{align*}
The case of $x = f(x) = y = g(y) = \infty$ gives Wang's formula \eqref{Eq:LEH}.

The proof of Lemma \ref{Lemma:BoundaryLE1} is an elementary change of variables.

\begin{proof}
    Supposing first that each of $x, f(x), y$, and $g(y)$ is finite, we define the \Mob transformations $T(z):= -1/(z-f(x)) \in \PSL_2(\C)$ and $T_w(z) := wz/(z+1/w) \in \PSL_2(\R)$.  The function $F:= T \circ f \circ T_x$ is then a conformal map from $\mathbb{H}$ to one component of $\wC \bs T(\gamma)$ which fixes $\infty$.  Similarly, $G:= T \circ g \circ T_y$ maps $\H^*$ to the other component, also fixing $\infty$, and thus
    \begin{align}\label{Eq:LEStandard}
        I^L(T(\gamma)) = I^L(\gamma) = \frac{1}{\pi}\int_{\H} \left| \frac{F''(z)}{F'(z)} \right|^2 \dd A(z) + \frac{1}{\pi}\int_{\H^*} \left| \frac{G''(z)}{G'(z)} \right|^2 \dd A(z).
    \end{align}
    Consider the first integral on the right.  Using the pre-Schwarzian composition rule (see Table \ref{Table:ABS}) and a change of variables, we have
    \begin{align}
        \frac{1}{\pi}\int_{\H} \left| \frac{F''(z)}{F'(z)} \right|^2 \dd A(z) &= \frac{1}{\pi}\int_{\H} \left| \frac{-2f'(T_x(z))T_x'(z)}{f\circ T_x(z) - f(x)} + \calA_f(T_x(z))T_x'(z) + \calA_{T_x}(z) \right|^2 \dd A(z) \notag\\
        &= \frac{1}{\pi}\int_{\H} \left| \frac{-2f'(z)}{f(z) - f(x)} + \calA_f(z) + \frac{\calA_{T_x}(T_x^{-1}(z))}{T_x'(T_x^{-1}(z))} \right|^2 \dd A(z). \label{Eq:EnergyExpansion1}
    \end{align}
    Recalling
    \begin{align*}
        \frac{\calA_{T_x}(T_x^{-1}(z))}{T_x'(T_x^{-1}(z))} = -\calA_{T_x^{-1}}(z) = \frac{2}{z-x},
    \end{align*}
    we see that \eqref{Eq:EnergyExpansion1} gives the claimed integral for $f$.  Combined with the same reasoning applied to the $G$ integral in \eqref{Eq:LEStandard}, we obtain the claimed formula.
    
    The situation is similar when one or multiple of $x, f(x),$ and $y$ are infinite.  For instance, if $f(x)=g(y) = \infty$, we need to show 
    \begin{align}\label{Eq:LENew1}
        I^L(\gamma) =\frac{4}{\pi}\int_{\H} \left| \frac{1}{z-x} + \frac{1}{2} \frac{f''(z)}{f'(z)} \right|^2 \dd A(z) + \frac{4}{\pi}\int_{\H^*} \left| \frac{1}{z-y} + \frac{1}{2} \frac{g''(z)}{g'(z)} \right|^2 \dd A(z).
    \end{align}
    Normalized maps $F$ and $G$ for \eqref{Eq:LEStandard} in this case are $F:=f\circ T_x$ and $G = g\circ T_y$, and  expanding the pre-Schwarzian derivative and doing a change of variables again yields \eqref{Eq:LENew1}.  
    
    As the interested reader may verify, doing the appropriate pre- and post-compositions with $T_w$ and $T$ to handle the remaining outstanding cases always yields the claimed formula.
\end{proof}

\subsubsection{The normalized pre-Schwarzian derivative}\label{Sec:NPS}
The integrands of Lemma \ref{Lemma:BoundaryLE1} inspire the following definition.
\begin{definition}\label{Def:NPS} Let $f: \Omega \rightarrow \wC$ a conformal map on a domain $\Omega \subset \wC$ such that $f$ extends continuously to $\overline{\Omega}$.  We define the \emph{normalized pre-Schwarzian derivative of $f$} as the function $\calB_f(z,w)$ on $\Omega \times \overline{\Omega}$ given by
\begin{align}\label{Eq:NPSDef}
    \calB_f(z,w) := \frac{f'(z)}{f(z)-f(w)} - \frac{1}{z-w} - \frac{1}{2} \frac{f''(z)}{f'(z)} = \frac{1}{2}\partial_z \log\bigg( \frac{(f(z)-f(w))^2}{f'(z)(z-w)^2} \bigg).
\end{align}
Here we assume that each of $z,w,f(z)$, and $f(w)$ are finite, and that $z \neq w$.   We define $\calB_f$ by limits when when $\infty$ appears.  More explicitly, first assuming that $z, f(z) \in \C$, we have that if $w \neq \infty$ but $f(w) = \infty$, 
\begin{align}\label{Eq:NPSfInfty}
    \calB_f(z,w) := \lim_{w_1 \rightarrow w} \calB_f(z,w_1) = - \frac{1}{z-w} - \frac{1}{2} \frac{f''(z)}{f'(z)}.
\end{align}
If $w = \infty$ but $f(\infty) \neq \infty$,
\begin{align}\label{Eq:NPSwInfty}
    \calB_f(z,\infty) := \lim_{w \rightarrow \infty} \calB_f(z,w) = \frac{f'(z)}{f(z)-f(\infty)} - \frac{1}{2} \frac{f''(z)}{f'(z)},
\end{align}
and if $w =\infty = f(w)$,
\begin{align}\label{Eq:NPSfwInfty}
    \calB_f(z,\infty) := \lim_{w \rightarrow \infty} \calB_f(z,w) = -\frac{1}{2}\mathcal{A}_f(z).
\end{align}
Thus the term in \eqref{Eq:NPSDef} containing the infinite quantity simply vanishes.

Having thus defined $\calB_f(z,w)$ for all $w \in \overline{\Omega}$ when $z,f(z) \in \C$, we also take limits when either $z$ or $f(z)$ is infinite. If $\infty \in \Omega$, then $\calB_f(\infty,w) := \lim_{z \rightarrow \infty} \calB_f(z,w)$, while if $z$ is a pole of $f$, $\calB_f(z,w) := \lim_{z_1 \rightarrow z} \calB_f(z_1,w)$.  We will see in Lemma \ref{Lemma:NPSMobiusInvariance1} below that poles are invisible to $\calB_f$, in the sense that $\calB_{\frac{1}{z} \circ f}(z,w) = \calB_f(z,w)$.  Furthermore, Corollary \ref{Lemma:BFinite} will show $\calB_f(\infty,\cdot) \equiv 0$. For us $z$ is typically a variable of integration, and so we can often simply ignore the finite set of points where $z$ or $f(z)$ is infinite. 

We also define the case $z=w$ through limits; see the first row of Table \ref{Table:NPSProps} below and Lemma \ref{Lemma:NPSDiagonal}.
\end{definition}

Continuity of $f$ on $\partial \Omega$ in Definition \ref{Def:NPS} is with respect to the spherical metric if $\infty \in \partial \Omega$.  Recall also that $f$ extends continuously to $\partial \Omega$ if and only if $\partial \Omega$ is locally connected \cite[Thm. 2.1]{Pommerenke}.  For our purposes it would suffice to assume that $\partial\Omega$ is Jordan (and hence $f$ extends homeomorphically to $\overline{\Omega}$ \cite[Thm. 2.6]{Pommerenke}).  The only reason to make this assumption on $\partial \Omega$ is so that $\calB_f(z,w)$ is well defined for all $w \in \partial \Omega$.


\begin{remark}
    We observe 
    \begin{align*}
        \partial_w \calB_f(z,w) = \frac{f'(z)f'(w)}{(f(z)-f(w))^2}- \frac{1}{(z-w)^2},
    \end{align*}
    an integral kernel which has frequently appeared in the study of univalent functions.  See, for instance, \cite[\S3.1]{Osgood1998} or the $K_1(z,w)$ kernel in \cite[Ch.2 \S2.1]{TTbook}.  (Compare also \cite[(1.9)]{Anderson} with the definition \eqref{Eq:NPSDef} of $\calB_f$.)
\end{remark}

We wish to study some elementary properties of $\calB_f$, with the goal of showing it behaves superior to the classical pre-Schwarzian $\calA_f$ (recall Table \ref{Table:ABS}).  The proofs can be tedious, because there are many cases to check, owing to the different definitions \eqref{Eq:NPSDef}, \eqref{Eq:NPSfInfty}, \eqref{Eq:NPSwInfty}, and \eqref{Eq:NPSfwInfty} of $\calB_f$ above.  For this reason we have sequestered away the details into Appendix \ref{Appendix:NPS}.  We summarize the basic properties we prove in  Table \ref{Table:NPSProps}, and then use these to prove Theorem \ref{Thm:BoundaryLE}, the boundary-normalized Loewner energy formula.

\begin{table}[ht]
\renewcommand{\tabcolsep}{0.2cm}
\renewcommand{\arraystretch}{1.5}
\centering
\begin{tabular}{|p{0.2\linewidth}|p{0.5\linewidth}|p{0.2\linewidth}|}
    \hline
     \emph{Idea} & \emph{Mathematical summary} & \emph{Precise statement}\\
     \hline
     Value on diagonal & $\calB_f(z,w) = -\frac{1}{6}S_f(w)(z-w) + O(z-w)^2$, \qquad $z \rightarrow w \in \C$ & Lemma \ref{Lemma:NPSDiagonal}\\

     \hline
     Composition rule & $\calB_{f\circ g}(z,w) = \calB_{f}\big(g(z),g(w)\big)g'(z) + \calB_g(z,w)$ & Lemma \ref{Lemma:NPSGeneralComp}\\
     
     \hline
     Vanishes iff $f$ is \Mob & $z \mapsto \calB_f(z,w) \equiv 0$ iff $f \in \PSL_2(\C)$ & Lemma \ref{Lemma:BMobiusZero}\\
     
     \hline
     \Mob covariance & $\calB_{S\circ f\circ T}(z,w) =\calB_{f}\big(T(z),T(w)\big)T'(z)$ for all $S,T \in \PSL_2(\C)$ & Lemma \ref{Lemma:NPSMobiusInvariance1}\\

     \hline
     Integral invariance under \Mob precomposition & $\int_\Omega |\calB_f(\cdot,w)|^2 = \int_{T^{-1}(\Omega)}|\calB_{f \circ T}(\cdot, T^{-1}(w))|^2$ for all $T \in \PSL_2(\C)$ & Corollary \ref{Lemma:BIntegralInvariance}\\

     \hline
     Holomorphic in $z$ & For fixed $w \in \overline{\Omega}$, $\calB_f(\cdot, w) \in \calH(\Omega)$ & Corollary \ref{Lemma:BFinite}\\

     \hline
     Value at $\infty$ & $\calB_f(\infty,w) = 0$ to at least order two & Corollary \ref{Lemma:BFinite}\\     
     \hline
\end{tabular}

\caption{\small Properties of the normalized pre-Schwarzian that we state and prove in Appendix \ref{Appendix:NPS}. For us the most important is the \Mob covariance of Lemma \ref{Lemma:NPSMobiusInvariance1}, and the associated integral invariance of Corollary \ref{Lemma:BIntegralInvariance}.}
\label{Table:NPSProps}
\end{table}

\begin{proof}[Proof of Theorem \ref{Thm:BoundaryLE}]
The case $D = \mathbb{H}$ is the content of Lemma \ref{Lemma:BoundaryLE1}.  For a different disk $D \subset \wC$, take a \Mob transformation $T$ mapping $\H$ onto $D$.  Combining Lemma \ref{Lemma:BoundaryLE1} and Corollary \ref{Lemma:BIntegralInvariance} yields
\begin{equation*}
    \begin{split}I^L(\gamma) = &\frac{4}{\pi}\int_{\mathbb{H}} |\calB_{f \circ T}(z,T^{-1}(w_0))|^2 \dd A(z)+  \frac{4}{\pi}\int_{\H^*} |\calB_{g\circ T}(z,T^{-1}(w_1))|^2 \dd A(z)\\
    = &\frac{4}{\pi}\int_{D} |\calB_f(z,w_0)|^2 \dd A(z)+  \frac{4}{\pi}\int_{D^*} |\calB_g(z,w_1)|^2 \dd A(z). \qedhere
    \end{split}
\end{equation*}
\end{proof}

\subsection{The interior-normalized formula}\label{Sec:InteriorNormalizedLE}
In this subsection we prove Theorem \ref{Thm:InteriorLE}, which generalizes Wang's result \eqref{Eq:LEDisk}.  As mentioned in the introduction, the formula will require two additional differential operators. In contrast to $\calB_f$, they will depend upon a pair of conformal maps on complementary disks in $\widehat{\C}$, combined with a pair of interior points in these disks.

\subsubsection{Two additional differential operators}\label{Sec:B*C}

Fix a disk $D \subsetneq \widehat{\C}$ with $D^* := \overline{D}^c$, and let $f:D \rightarrow \Omega$ and $g:D^* \rightarrow \Omega^*$ be conformal maps to the two complementary components $\Omega, \Omega^*$ of a Jordan curve $\gamma \subset \widehat{\C}$.  Let $u \in \overline{D}$ and $v \in \overline{D^*}$.  We define
\begin{align}\label{Eq:NPS*}
    \calB_{f,g}^*(z,u,v) = \calB_{f,g,D}^*(z,u,v) := \frac{f'(z)}{f(z)-g(v)} - \frac{1}{z-u^*} - \frac{1}{2}\frac{f''(z)}{f'(z)}
\end{align}
when $z, f(z), g(v)$ and $u^*$ are all finite.  Here $u^* := j(u)= j_D(u)$ is the image of $u$ under the anti-holomorphic involution of $\widehat{\C}$ given by Schwarz reflection across $\partial D$.  (The ``$^*$'' decorating the operator name reflects this and helps distinguish $\calB^*$ from $\calB$; it is unrelated to an adjoint.)  

As in Definition \ref{Def:NPS}, appearances of $\infty$ in \eqref{Eq:NPS*} are handled by limits.  Assuming $z, f(z) \in \C$, we take a limit $u_1 \rightarrow u$ and/or $v_1 \rightarrow v$ if $u^* = \infty$ or $g(v) = \infty$, and the corresponding term in \eqref{Eq:NPS*} simply disappears, as for $\calB_f$.  We also define $\calB_{f,g}^*(z,u,v)$ when either $z$ or $f(z)$ is infinite through limits ($z$ again is typically a variable of integration, and so we can ignore these points).  Note that when $u\in \partial D$ and $v$ is the image $h(u)$ of $u$ under the conformal welding $h:=g^{-1} \circ f$, 
\begin{align}\label{Eq:NPSCaseOfNPS*}
    \calB_{f,g}^*(z,u,h(u)) = \calB_f(z,u),
\end{align}
and so $\calB^*$ generalizes and extends $\calB$.

With $f,g,D,D^*,u,v$ as above, we also define
\begin{align}
    \calC_{f,g}(z,u,v) = \calC_{f,g,D}(z,u,v) :=& \calB_f(z,u) - \calB^*_{f,g}(z,u,v)\label{Eq:CalCDef}\\
    =& \frac{f'(z)}{f(z)-f(u)} - \frac{f'(z)}{f(z)-g(v)} - \bigg(\frac{1}{z-u} - \frac{1}{z-u^*}  \bigg),\notag
\end{align}
when each of $z, f(z), u, u^*, f(u)$, and $g(v)$ is finite.  As above, $u^* := j_D(u)$ is the Schwarz reflection of $u$ across $\partial D$, and if any of $f(u), g(v), u,$ and $u^*$ is infinite, the corresponding term in \eqref{Eq:CalCDef} disappears.  We likewise extend to $z=\infty$ or a pole of $f$ by limits.

In order to prove our interior-normalized Loewner energy formula, Theorem \ref{Thm:InteriorLE}, we also need to understand how the $\calB^*$ and $\calC$ operators relate to pre- and post-composition by elements of $\PSL_2(\C)$.  As for the normalized $\calB_f$, we hide these details away in Appendix \ref{Appendix:NPS}, but summarize in Table \ref{Table:NPS*Props}.

\begin{table}
\renewcommand{\tabcolsep}{0.2cm}
\renewcommand{\arraystretch}{1.5}
\centering
\begin{tabular}{|p{0.115\linewidth}|p{0.605\linewidth}|p{0.19\linewidth}|}
    \hline
     \emph{Idea} & \emph{Mathematical summary} & \emph{Precise \qquad statement}\\
     
     \hline
     \Mob covariance & 
     $\calB^*_{S\circ f \circ T_u, S \circ g \circ T_v} (z,u,v) = \calB^*_{f,g}\big(T_u(z),T_u(u), T_v(v)\big)T_u'(z)$
     & Lemma \ref{Lemma:NPS*MobiusInvariance}\\

     \hline
     \Mob covariance & $\calC_{S\circ f \circ T_u, S \circ g \circ T_v} (z,u,v) = \calC_{f,g}\big(T_u(z),T_u(u), T_v(v)\big)T_u'(z)$ & Lemma \ref{Lemma:NPS*MobiusInvariance}\\
     
     
     \hline
\end{tabular}

\caption{\small Properties of the $\calB^*$ and $\calC$ found in Appendix \ref{Appendix:NPS}.}
\label{Table:NPS*Props}
\end{table}

\subsubsection{The interior-normalized formula}
Taking Lemma \ref{Lemma:NPS*MobiusInvariance} for granted, we can prove our interior-normalized energy formula.
\begin{proof}[Proof of Theorem \ref{Thm:InteriorLE}]
    As the roles of $D$ and $D^*$ are interchangeable, it suffices to prove
\begin{multline*}
    I^L(\gamma) = \frac{4}{\pi}\int_D | \calB_{f,g}^*(z,u,v) |^2\dd A(z) + \frac{4}{\pi}\int_{D^*} | \calB_{g}(z,v) |^2\dd A(z) \\ 
    - \frac{2}{\pi}\int_D | \calC_{f,g}(z,u,v) |^2\dd A(z) - \frac{2}{\pi}\int_{D^*} | \calC_{g,f}(z,v,u) |^2\dd A(z).
\end{multline*}
    In the case of $D= \D$, $u = 0$, $v= \infty$, $\gamma \subset \widehat{\C}\bs\{0,\infty\}$, and $f$ and $g$ the corresponding conformal maps normalized with $f(0) = 0 = 1/g(\infty)$, this is Wang's formula \eqref{Eq:LEDisk}, as we saw in Example \ref{Eg:InteriorLE}.

    Now let $D \subsetneq \widehat{\C}$ be a general disk with $u\in D$, $v \in D^*$, and let $\gamma \subset \widehat{\C}$ be a general Jordan curve. Take \Mob $T_u: \D \rightarrow D$ with $T_u(0)=u$, \Mob $T_v: \D^* \rightarrow D^*$ with $T_v(\infty)=v$, and \Mob $S$ such that $S(f(u)) = 0 = 1/S(g(v))$.  Considering $F:= S \circ f \circ T_u$ and $G := S \circ g \circ T_v$, we apply the previous paragraph, Lemmas \ref{Lemma:NPSMobiusInvariance1} and \ref{Lemma:NPS*MobiusInvariance}, and a change of variables to see
    \begin{align*}
        I^L(\gamma) &= \frac{4}{\pi}\int_{\D} | \calB^*_{F,G}(z,0,\infty) |^2\dd A(z) + \frac{4}{\pi}\int_{\D^*} | \calB_{G}(z,\infty) |^2\dd A(z)\\ & \hspace{8mm}- \frac{2}{\pi}\int_{\D} | \calC_{F,G}(z,0,\infty) |^2\dd A(z) - \frac{2}{\pi}\int_{\D^*} | \calC_{G,F}(z,\infty,0) |^2\dd A(z)\\
            &= \frac{4}{\pi}\int_{\D} | \calB^*_{f,g}(T_u(z),u,v) |^2|T_u'(z)|^2\dd A(z) + \frac{4}{\pi}\int_{\D^*} | \calB_{g}(T_v(z),v) |^2|T_v'(z)|^2\dd A(z)\\ & \hspace{8mm}- \frac{2}{\pi}\int_{\D} | \calC_{f,g}(T_u(z),u,v) |^2|T_u'(z)|^2\dd A(z) - \frac{2}{\pi}\int_{\D^*} | \calC_{g,f}(T_v(z),v,u) |^2|T_v'(z)|^2\dd A(z)\\
        &= \frac{4}{\pi}\int_{D} | \calB^*_{f,g}(z,u,v) |^2\dd A(z) + \frac{4}{\pi}\int_{D^*} | \calB_{g}(z,v) |^2\dd A(z)\\ & \hspace{8mm}- \frac{2}{\pi}\int_{D} | \calC_{f,g}(z,u,v) |^2\dd A(z) - \frac{2}{\pi}\int_{D^*} | \calC_{g,f}(z,v,u) |^2\dd A(z). \hfill \qedhere
    \end{align*}     
\end{proof}
We can generalize \eqref{Eq:GrunskyID} to give a version of the interior-normalized formula that replaces the $\calC$ integrals with a log term, similar to Wang's identity \eqref{Eq:LEDisk}.  We will do so in the more specific setting of a curve $\gamma$ not passing through $\infty$, and where $D$ is a bounded disk.  Furthermore, we assume $f$ maps to the bounded component of $\C \bs \gamma$, and $v = u^*$. 
\begin{cor}\label{Cor:InteriorLELog}
    Let $\gamma \subset \C$ be a Jordan curve, $D \subset \C$ a bounded disk, and $f:D \rightarrow \Omega$ and $g: D^* \rightarrow \Omega^*$ any conformal maps to the two complementary components of $\widehat{\C}\bs \gamma = \Omega \cup \Omega^*$, where $\Omega$ is the bounded component of $\C \bs \gamma$.  Then for any $u \in D$,
\begin{align}\label{Eq:InteriorLELog}
    I^L(\gamma) &= \frac{4}{\pi}\int_D | \calB_{f}(z,u) |^2\dd A(z) + \frac{4}{\pi}\int_{D^*} | \calB^*_{g,f}(z,u^*,u) |^2\dd A(z) +4 \log \bigg|  \frac{f'(u) \tilde{g}'(\tilde{u}^*)(u-u^*)^2}{(f(u)-g(u^*))^2}  \bigg|,
\end{align}
where the fraction inside the logarithm is interpreted as follows: if $u^*=\infty$, then $(u-u^*)^2$ is replaced by 1, and if $g(u^*)=\infty$, then $(f(u)-g(u^*))^2$ is replaced by 1.  Here $\tilde{g}'(\tilde{u}^*)$ is the derivative of $g$ at $u^*$ in the coordinate charts given by $j(z)=1/z$, when either $u^* = \infty$ or $g(u^*) = \infty$ (or both).
\end{cor}
\noindent More explicitly, if both $u^*$ and $g(u^*)$ are finite, the log term in \eqref{Eq:InteriorLELog} is 
\begin{align*}
    4 \log \bigg|  \frac{f'(u) g'(u^*)(u-u^*)^2}{(f(u)-g(u^*))^2}  \bigg|.
\end{align*}
If $u^* \neq \infty = g(u^*)$, we obtain
\begin{align*}
    4 \log \big|  f'(u) \tilde{g}'(u^*)(u-u^*)^2  \big|
\end{align*}
for $\tilde{g} = j \circ g$, while if $u^* = \infty \neq g(u^*)$, the expression is
\begin{align*}
    4 \log \bigg|  \frac{f'(u) \tilde{g}'(0)}{(f(u)-g(\infty))^2}  \bigg|,
\end{align*}
where $\tilde{g} = g \circ j$.  Lastly, if $u^* = g(u^*) = \infty$, we have $4 \log |  f'(0) \tilde{g}'(0)|$ for $\tilde{g} = j \circ g \circ j$.  Since 
\begin{align}\label{Eq:WangsIDCase}
    4 \log |  f'(0) \tilde{g}'(0)| = 4 \log \bigg| \frac{f'(0)}{\lim\limits_{z \rightarrow \infty} g'(z)} \bigg| = 4 \log \bigg| \frac{f'(0)}{g'(\infty)} \bigg|,
\end{align}
this cases recovers Wang's identity \eqref{Eq:LEDisk}, as we saw in Example \ref{Eg:InteriorLE}.  

As an aside, we note that the logarithm term in \eqref{Eq:InteriorLELog} is thus similar in both structure and behavior to our point-wise welding functional $L_h(x,y)$ introduced in \S\ref{Sec:LFunctional}. 

Given Lemma \ref{Lemma:NPS*MobiusInvariance}, the proof of Corollary \ref{Cor:InteriorLELog} is elementary and essentially a computation; we include some details for completeness.
\begin{proof}
    We first assume $D = \D$.  By Theorem \ref{Thm:InteriorLE}, it suffices to show 
    \begin{align}\label{Eq:InteriorLELogNTS1}
        - \frac{2}{\pi}\int_{\D} | \calC_{f,g}(z,u,u^*) |^2\dd A(z) - \frac{2}{\pi}\int_{\D^*} | \calC_{g,f}(z,u^*,u) |^2\dd A(z) = 4 \log \bigg|  \frac{f'(u) \tilde{g}'(\tilde{u}^*)(u-u^*)^2}{(f(u)-g(u^*))^2}  \bigg|,
    \end{align}
    with the right-hand side interpreted as in the statement of the corollary.  There are four cases, depending on whether each of $u^*$ and $g(u^*)$ is infinite or not.  When $u^*=g(u^*) = \infty$, the log term is as in \eqref{Eq:WangsIDCase}, which is the desired equality by the Grunsky identity \eqref{Eq:GrunskyID} (see \cite[Ch.II Rmk. 2.2]{TTbook}).  
    
    We next consider the case $u^* \neq \infty$, $g(u^*) \neq \infty$, and leave the remaining two cases for the computation of the interested reader; the ideas are similar.  Given $u^* = 1/\bar{u} \neq \infty \neq g(u^*)$, we must show    \begin{align}\label{Eq:InteriorLELogNTS2}
        - \frac{2}{\pi}\int_{\D} | \calC_{f,g}(z,u,u^*) |^2\dd A(z) - \frac{2}{\pi}\int_{\D^*} | \calC_{g,f}(z,u^*,u) |^2\dd A(z) = 4 \log \bigg|  \frac{f'(u) {g}'({u}^*)(u-u^*)^2}{(f(u)-g(u^*))^2}  \bigg|.
    \end{align}
    Utilizing the \Mob transformations $S(z) := \frac{1}{z-g(u^*)}$ and $T(z) := \frac{z+u}{\bar{u}z+1} \in \text{Aut}(\D)$, we consider $F := S \circ f \circ T$ and $G := S \circ g \circ T$.  By the above-discussed $u^* = \infty = g(u^*)$ case,
    \begin{align}\label{Eq:InteriorLELogNTS3}
        - \frac{2}{\pi}\int_{\D} | \calC_{F,G}(z,0,\infty) |^2\dd A(z) - \frac{2}{\pi}\int_{\D^*} | \calC_{G,F}(z,\infty,0) |^2\dd A(z) = 4\log \frac{|F'(0)|}{|G'(\infty)|},
    \end{align}
    and here the left-hand side is the left-hand side of \eqref{Eq:InteriorLELogNTS2} by Lemma \ref{Lemma:NPS*MobiusInvariance} and a change of variables.  Turning to the right-hand side of \eqref{Eq:InteriorLELogNTS3}, computation yields
    \begin{align}\label{Eq:InteriorLELogF'(0)}
        |F'(0)| = \bigg| \frac{f'(u)(1-|u|^2)}{(f(u)-g(u^*))^2}  \bigg|,
    \end{align}
    and we find $G$ has expansion near $\infty$ given by
    \begin{align*}
        G(z) = \frac{1}{ g'(u^*)u^*(u-u^*)z^{-1} + O(z^{-2})} = \frac{z}{g'(u^*)u^*(u-u^*)} + O(1).
    \end{align*}
    Thus $\lim\limits_{z \rightarrow \infty} G'(z) = \frac{1}{g'(u^*)u^*(u-u^*)}$, which, combined with \eqref{Eq:InteriorLELogF'(0)}, yields the right-hand side of \eqref{Eq:InteriorLELogNTS2}.

    For a general bounded disk $D$, we pre-compose $f$ and $g$ with an affine mapping $A(z)=az+b$ of $\D$ onto $D$, normalizing at $A^{-1}(u)$ and $A^{-1}(u^*)$.  One then checks that both sides of \eqref{Eq:InteriorLELogNTS1} are invariant under this operation.  They are, and we spare the reader the tedious details of the four cases.  

\end{proof}

\subsubsection{Comment on the relation between the boundary-normalized and interior-normalized cases}\label{Sec:RelationBetweenTwoFormulas}

We conclude \S\ref{Sec: BoundayNormalizedforLE} by observing that there is a sense in which we can view the boundary-normalized formula of Theorem \ref{Thm:BoundaryLE} as a limiting case of the interior-normalized formula of Theorem \ref{Thm:InteriorLE}.  Indeed, take sequences $\{u_n\} \subset D$, $\{v_n\} \subset D^*$ such that $u_n \rightarrow w_0 \in \partial D$ and $v_n \rightarrow w_1 := g^{-1}\circ f(w_0)$.  By continuity of the conformal map up to $\partial D$, $\calB_{f}(z,u_n) \rightarrow \calB_{f}(z,w_0)$ point-wise a.e. $z \in D$, while, similarly,
    \begin{align*}
        \calB^*_{g,f}(z,v_n,u_n) = &\frac{g'(z)}{g(z)-f(u_n)} - \frac{1}{z-v_n^*} - \frac{1}{2}\frac{g''(z)}{g'(z)}\\
        \rightarrow & \frac{g'(z)}{g(z)-f(w_0)} - \frac{1}{z-w_1} - \frac{1}{2}\frac{g''(z)}{g'(z)} = \calB_g(z,w_1)
    \end{align*}
    (the formulas here assume $f(u_n), v_n^*, f(w_0),$ and $w_1$ all belong to $\C$, but the limit still holds by similar reasoning even if they do not).  Also,
    \begin{align*}
        \calC_{f,g}(z,u_n,v_n) = \frac{f'(z)}{f(z)-f(u_n)} - \frac{f'(z)}{f(z) - g(v_n)} - \frac{1}{z-u_n} + \frac{1}{z-u_n^*} \rightarrow 0,
    \end{align*}
    while similarly $\calC_{g,f}(z,v_n,u_n) \rightarrow 0$.  Thus the limit
    \begin{align*}
        I^L(\gamma) =  \lim_{n \rightarrow \infty}\bigg[ &\frac{4}{\pi}\int_D | \calB_{f}(z,u_n) |^2\dd A(z) + \frac{4}{\pi}\int_{D^*} | \calB^*_{g,f}(z,v_n,u_n) |^2\dd A(z)\\ 
    &\hspace{4mm}- \frac{2}{\pi}\int_D | \calC_{f,g}(z,u_n,v_n) |^2\dd A(z) - \frac{2}{\pi}\int_{D^*} | \calC_{g,f}(z, v_n, u_n ) |^2\dd A(z)\bigg],
    \end{align*}
    which by Theorem \ref{Thm:InteriorLE} is the limit of a constant sequence, is the same as
    \begin{align*}
        &\frac{4}{\pi}\int_D \lim_{n \rightarrow \infty}| \calB_{f}(z,u_n) |^2\dd A(z) + \frac{4}{\pi}\int_{D^*} \lim_{n \rightarrow \infty}|\calB^*_{g,f}(z,v_n,u_n) |^2\dd A(z)\notag\\ 
    &\hspace{4mm}- \frac{2}{\pi}\int_D \lim_{n \rightarrow \infty}| \calC_{f,g}(z,u_n,v_n) |^2\dd A(z) - \frac{2}{\pi}\int_{D^*} \lim_{n \rightarrow \infty}| \calC_{g,f}(z,v_n,u_n) |^2\dd A(z)\\
     =&\frac{4}{\pi}\int_D | \calB_{f}(z,w_0) |^2\dd A(z) + \frac{4}{\pi}\int_{D^*}|\calB_{g}(z,w_1) |^2\dd A(z) = I^L(\gamma),
    \end{align*}
    by \ref{Thm:BoundaryLE}.  In words, the equality of Wang's two expressions \eqref{Eq:LEDisk} and \eqref{Eq:LEH} for the Loewner energy (which undergirds the equality of our interior-normalized and boundary-normalized formulas) yields that we can interchange the limit and integral in our interior-normalized energy formula.

\section{Topology of the normalized Weil--Petersson class}\label{Sec: topologicalconvergence}

The point of this section is to review equivalent modes of convergence in the normalized Weil-Petersson \Teich space on $\H$, as summarized in Proposition~\ref{Thm: equiv convergence H}.  We start by recalling the arc-length parametrization and associated normalizations.

\subsection{Arc-length parametrization and normalization}



We first recall Bishop's characterization of bounded Weil-Petersson curves $\gamma \subset \C$ in terms of an arc-length parametrization $a: \SSS \to \gamma$.  

\begin{prop}\rm{\cite[Thm. 1.1]{Bishop2019weil}}
    A Jordan curve $\gamma \subset \C$ is Weil--Petersson if and only if it has an arclength parametrization $a: \SSS \to \gamma$ in the Sobolev space $H^{\frac{3}{2}}(\SSS)$, or equivalently $a' \in H^{\half}(\SSS)$.
\end{prop}


\noindent Shen and Wu~\cite{WPIII} addressed the case when $\infty \in \gamma$ in terms of an arc-length parametrization $z: \R \rightarrow \gamma$.  We recall that a rectifiable curve $\gamma$ is \emph{chord arc} if for all $x,y \in \gamma$, $\ell(\gamma(x,y)) \lesssim |x-y|$, where $\gamma(x,y)$ is the arc on $\gamma$ of shortest length connecting $x$ to $y$.


\begin{prop}\rm{\cite[Thm. 2.2]{WPIII}}\label{Prop:ShenWuWPR}
A Jordan curve $\gamma$ with $\infty \in \gamma$ is Weil--Petersson if and only if it is chord-arc and there exists some function $b$ in the real Sobolev class $H^{\half}(\R)/\R$ such that $z'(s)=e^{ib(s)}$.
\end{prop}
\noindent Thus for unbounded $\gamma$, the derivative $z'(s)=e^{ib(s)}$ itself is not necessarily in $H^\half(\R)$, but rather its lift is.  We follow Shen--Wu~\cite{WPIII} in normalizing Weil--Petersson curves as follows.

\begin{definition}\label{Def:NormalizedCurve}
    A Weil--Petersson curve $\gamma$ is \emph{normalized} if it passes through $0$, $x>0$, and $\infty$, where the arc length of $\gamma$ between zero and $x$ is exactly 1.  By $z: \R \rightarrow \gamma$ we denote the associated \emph{normalized arc-length parametrization} satisfying $z(0) =0$ and $z(1)=x$.  The \emph{left domain $H$ (right domain $H^*$)} of $\gamma$ is then the component of $\C \bs \gamma$ which is positively-oriented (negatively-oriented) with respect to traversing $\gamma$ by passing through $0, x,$ and $\infty$ in order.  We say conformal maps $f: \H \rightarrow H$ and $g: \H^* \rightarrow H^*$ are \emph{normalized} if they satisfy $f(s)=z(s) = g(s)$ for $s=0,1,\infty$.  The associated \emph{normalized conformal welding} is $h := g^{-1} \circ f$, and fixes 0,1, and $\infty$.
\end{definition}
\noindent By Proposition \ref{Prop:ShenWuWPR}, the arc-length parametrization $z$ of normalized $\gamma$ is of the form $z(s)=e^{ib(s)}$ for some $b \in H^\half(\R)$. Proposition \ref{Thm: equiv convergence H} below will describe convergence of normalized Weil--Petersson curves in terms of each of $f,g,h,$ and $b$.

\begin{remark}\label{rmk: TwoNormalizationDiff.}
The normalization of Definition \ref{Def:NormalizedCurve} differs slightly from the usual normalization \eqref{Fact: QuasicirclesModelofT01} of $T(1)$ in terms of oriented quasicircles passing through $0,1,$ and $\infty$ (or quasidisks, as in \cite[Ch.III  \S1.5]{Lehtobook}).  However, as noted by Shen-Wu, this normalization still preserves the equivalence relation on $T(1)$ \cite[\S2.2]{WPIII}. 
\end{remark}

An arc-length parametrization $z$ leads to two new elements of $\ophom{\R}$ through setting $H_f:= z^{-1} \circ f$ and $H_g:= z^{-1}\circ g$.  We may thereby factorize of the conformal welding as $h = H_g^{-1} \circ H_f$, and we can furthermore express the Loewner energy solely in terms of $H_f$ and $H_g$.

\begin{lemma}\label{Lemma:LEWeldingFactorization}
    Let $\gamma$ be a normalized Weil-Petersson curve with associated maps as in Definition \ref{Def:NormalizedCurve} and the previous paragraph.  Then
    \begin{align*}
        I^L(\gamma) = 2\big\|\log |H_f'| \big\|_{H^{1/2}(\R)}^2+ 2\big\|\log |H_g'| \big\|_{H^{1/2}(\R)}^2.
    \end{align*}
\end{lemma}
\begin{proof}
    For a.e. $x \in \R$, $|H_f'(x)| = |f'(x)|$, and similarly for $H_g$, and so the formula immediately follows from \eqref{Eq:LERHHalf}.
\end{proof}

\subsection{Topological convergence of the normalized Weil--Petersson class}

The following result collects various equivalent perspectives on convergence of normalized Weil--Petersson curves from \cite{WPIII}, which will prove useful for our approximation arguments in \S\ref{Sec:ProofOfComparable} below.  We recall that \cite{WPIII} shows four different models of Weil--Petersson \Teich space, equipped with their corresponding real Hilbert manifold structures, are all topologically equivalent to the Takhtajan-Teo complex Hilbert manifold structure on $T_0(1)$~\cite{TTbook}.  In particular,  we obtain the following equivalences for convergence.  We recall from \eqref{Def:A^1_2} that $A^1_2(\H) = \calH(\H)\cap L^2(\H)$ is the holomorphic $L^2$ space, and $A^1_2(\H^*) = \calH(\H^*)\cap L^2(\H^*)$ is the analogous set of functions on $\H^*$.  As usual, $\calA_F$ is the pre-Schwarzian derivative \eqref{Def:Preschwarzian} of $F$.

\begin{prop}{\cite[Theorems 2.3 and 2.4]{WPIII}}\label{Thm: equiv convergence H}
      Let $\{\gamma_n\}_{n=1}^{\infty}$ be a sequence of normalized Weil--Petersson curves with associated normalized conformal maps $f_n$ and $g_n$, normalized weldings $h_n$, and normalized arc-length parametrizations $z_n = e^{ib_n}$, where normalizations are according to Definition \ref{Def:NormalizedCurve}, and $b_n$ is the corresponding element of $H^\half(\R)$.  Similarly, let $f, g, h$, and $z = e^{ib}$ be the normalized functions associated to Weil--Petersson $\gamma$.  Then the following are equivalent: 
   \begin{enumerate}[(i)]
        \item $b_n \rightarrow b$ in $H^\half(\R)$, 

         \item\label{Lemma:TopologyWelding} $\log h_n' \rightarrow \log h'$ in $H^\half(\R)$,
    
      \item\label{Lemma:TopologyPSH} $\mathcal{A}_{f_n} \rightarrow \calA_f$ in $A_2^1(\H)$,

      \item\label{Lemma:TopologyPSH*} $\mathcal{A}_{g_n} \rightarrow \calA_g$ in $A_2^1(\H^*)$.

   \end{enumerate}
   If these are satisfied, we say $\gamma_n \rightarrow \gamma$ in the normalized Weil--Petersson \Teich space.
\end{prop}

\noindent We can also obtain the following.  We recall $\calS_F$ is the Schwarzian derivative \eqref{Def:Schwarzian} of $F$, and $A_2(\H)$ is as in \eqref{Def:A_2}, with $A_2(\H^*)$ defined similarly.
\begin{lemma}\label{Lemma:SchwarziansLEConverge}
    Let $\gamma_n,\gamma$ be normalized Weil--Petersson curves such that $\gamma_n \rightarrow \gamma$ in the sense of Proposition \ref{Thm: equiv convergence H}.  Then $I^L(\gamma_n) \rightarrow I^L(\gamma)$, and the normalized conformal maps satisfy $\mathcal{S}_{f_n} \rightarrow \calS_f$ in $A_2(\H)$ and $\mathcal{S}_{g_n} \rightarrow \calS_g$ in $A_2(\H^*)$.
\end{lemma}
\noindent For the proof we will use an observation about the map 
\begin{align*}
    \mathcal{L}(f) := f' - \half f^2,
\end{align*} 
which is inspired by the identity $\mathcal{L}(\mathcal{A}_{f})=\mathcal{S}_{f}$, and plays an important role in connecting the pre-Schwarzian and Schwarzian models of the Weil--Petersson \Teich space.  While $\mathcal{L}$'s holomorphy has been shown in the disk case~\cite[Ch.II Lem. 1.5, and Lem. A.1]{TTbook}, we need a parallel result for $\H$.  It is known that $\calL: A^1_2(\H) \rightarrow A_2(\H)$ \cite[Thm. 4.4]{STWrealline}.
\begin{lemma}\label{Lemma: holographyofL}
    The map $\mathcal{L}: A_2^1(\H) \to A_2(\H)$ is a bounded and holomorphic mapping of Hilbert spaces.
\end{lemma}
\begin{proof} 
    We first show $\mathcal{L}$ is bounded. For $\psi$ in the unit ball of $A_2^1(\mathbb{H}),$ 
    \begin{align}
        \|\mathcal{L}(\psi)\|_{A_2(\H)}^2&=\int_{\H} \big|\psi'(z)-\half \psi(z)^2\big|^2 y^2 dA(z) \notag\\ 
        &\leq 2 \int_{\H}|\psi'(z)|^2y^2 dA(z) + \half \int_{\H} |\psi(z)|^4y^2dA(z).\label{eq: L2estimate}
    \end{align}
    For the first term, \cite[Prop. 8.3]{WPII} shows
    \begin{align}\label{eq: L1boundedness}
        \int_{\H}|\psi'(z)|^2y^2 dA(z)\leq C_1 \int_{\H} |\psi(z)|^2dA(z)
    \end{align}
    for some constant $C_1>0$.\footnote{While \cite[Prop. 8.3]{WPII} further assumes $\psi(\infty)=0$, the bound \eqref{eq: L1boundedness} in the first half of their argument does not require this.} On the other hand, since $\psi \in A_2^1(\mathbb{H}) \subset A_{\infty}^1(\mathbb{H})$ and the inclusion map is continuous \cite{Zhu2007operator}, the second term in \eqref{eq: L2estimate} satisfies
    \[
    \int_{\H} |\psi(z)|^4y^2dA(z)\leq \|\psi\|_{A_{\infty}^1(\mathbb{H})}^2 \int_{\H}|\psi|^2dA \leq C_2^2 \| \psi \|_{A_2^1(\H)}^2 \int_{\H}|\psi|^2dA \leq C_2^2 \int_{\H}|\psi|^2dA
    \]
    since $\psi$ is in the unit ball, and where $A^1_\infty(\H)$ is as in \eqref{Def:A^1infty}.  We conclude $\|\mathcal{L}(\psi)\|_{A_2(\H)}^2 \lesssim \|\psi\|_{A_2^1(\H)}^2$, as claimed.

    To show holomorphy, according to \cite[Ch.V Lem. 5.1]{Lehtobook} it remains to prove that for any $\psi, \varphi \in A_2^1(\mathbb{H})$, the map $\mathbb{C} \ni w \mapsto \mathcal{L}(\psi+w\varphi)$ is holomorphic as a map from a neighborhood of $0 \in \mathbb{C}$ into $A_2(\H)$.  That is, for all $w_0 \in B_\varepsilon(0) \subset \C$, we must show there is a continuous linear map $D_{\psi + w_0\varphi}\calL: \C \rightarrow A_2(\H)$ such that
    \begin{align}\label{Lim:NTSHolomorphic}
        \lim_{w \rightarrow w_0} \frac{1}{|w-w_0|}\big\| \calL(\psi + w \varphi) - \calL(\psi + w_0\varphi) - (D_{\psi + w_0\varphi}\calL)(w-w_0) \big\|_{A_2(\H)} = 0.
    \end{align}
    We have
    \begin{align*}
        D_{\psi + w_0\varphi}\calL = \frac{\dd}{\dd w}\Big|_{w=w_0}\mathcal{L}(\psi+w\varphi) = \varphi' - \psi \varphi - w_0 \varphi^2,
    \end{align*}
    which we claim this is an element of $A_2(\H)$.  By \eqref{eq: L1boundedness}, $\varphi' \in A_2(\H)$, and we also observe 
    \begin{align*}
        \| \psi\varphi \|_{A_2(\H)} =  \bigg( \int_{\H} |\psi(z)\varphi(z)|^2 y^2 dA(z)\bigg)^{1/2}  \leq \|\psi\|_{A^1_\infty(\H)} \|\varphi\|_{A^1_2(\H)} < \infty,
    \end{align*}
    while similarly 
    \begin{align}\label{Ineq:varphi^2}
        \| \varphi^2 \|_{A_2(\H)} \leq \|\varphi\|_{A^1_\infty(\H)} \|\varphi\|_{A^1_2(\H)} < \infty.
    \end{align}
    Thus $w \mapsto D_{\psi + w_0\varphi}\calL \cdot w$ maps $\C$ into $A_2(\H)$, as claimed.  Furthermore, the difference quotient in \eqref{Lim:NTSHolomorphic} is 
    \begin{align*}
        \Big\|\frac{\mathcal{L}(\psi+w\varphi)-\mathcal{L}(\psi+w_0\varphi)}{w-w_0}- D_{\psi + w_0\varphi}\calL \Big\|_{A_2(\H)} = \frac{|w-w_0|}{2}\|\varphi^2\|_{A_2(\H)} = O(|w-w_0|)
     \end{align*}
     by \eqref{Ineq:varphi^2}.
\end{proof}

\begin{proof}[Proof of Lemma \ref{Lemma:SchwarziansLEConverge}]
    The reverse triangle inequality applied to Proposition \ref{Thm: equiv convergence H} $(\ref{Lemma:TopologyPSH})$ and $(\ref{Lemma:TopologyPSH*})$ immediately yields $I^L(\gamma_n) \rightarrow I^L(\gamma)$.  For convergence of the Schwarzians, we proceed similarly to \eqref{eq: L2estimate}, noting
    \begin{align*}
       \|\mathcal{S}_{f_n}-\mathcal{S}_f\|_{A_2(\mathbb{H})}&= \|\calL(\mathcal{A}_{f_n})-\calL(\mathcal{A}_f)\|_{A_2(\mathbb{H})}\\
       &\lesssim \|\mathcal{A}_{f_n}'-\mathcal{A}_f'\|_{A_2(\mathbb{H})} + \|(\mathcal{A}_{f_n}+\mathcal{A}_f) ( \mathcal{A}_{f_n}-\mathcal{A}_f )\|_{A_2(\mathbb{H})}\\
       &\leq \big\|\calL_1\big(\mathcal{A}_{f_n}-\mathcal{A}_f\big)\big\|_{A_2(\mathbb{H})} + \big(\| \calA_{f_n}\|_{A^1_\infty(\H)} + \| \calA_f \|_{A^1_\infty(\H)} \big)\|\mathcal{A}_{f_n}-\mathcal{A}_f \|_{A_2^1(\mathbb{H})}\\
       &\lesssim \|\mathcal{A}_{f_n}-\mathcal{A}_f\|_{A_2^1(\mathbb{H})} + \big(\| \calA_{f_n}\|_{A^1_2(\H)} + \| \calA_f \|_{A^1_2(\H)} \big)\|\mathcal{A}_{f_n}-\mathcal{A}_f \|_{A_2^1(\mathbb{H})} \rightarrow 0,
   \end{align*}
   where we have used the boundedness of both $\calL_1:A^1_2(\H) \rightarrow A_2(\H)$ and the inclusion map $\iota: A^1_2(\H) \rightarrow A^1_\infty(\H)$. By symmetry, $\|\mathcal{S}_{g_n}-\mathcal{S}_g\|_{A_2(\mathbb{H}^*)} \rightarrow 0$ as well. 
\end{proof}

\begin{remark}\label{Remark:SchwarzianEnergy}
    In the setting of $\D$, \cite[Cor. 5.12]{UniversalLiouvilleaction}\footnote{Note that numbering in the arXiv version may be different than in the published version.} explicitly links the Loewner energy and norms of the associated Schwarzians of $\gamma \in C^{5, \alpha}$ via the identity
    \begin{align}\label{Eq:BBPWLE}
        \frac{\pi}{4} I^{L}(\gamma)= V(\gamma)-\half \int_{\D} \left|\mathcal{S}_f(z)\right|^2\frac{(1-|z|^2)^2}{4} dA(z) - \half \int_{\D^*} \left|\mathcal{S}_g(z)\right|^2 \frac{(1-|z|^2)^2}{4} dA(z),
    \end{align}
    where 
    \[ V(\gamma) := \lim\limits_{\epsilon \to 0^{+}} \int_{\H^3} \varphi_{\gamma}^* \mathbf{1}_{\xi \geq \epsilon} \frac{vol_{eucl}}{\xi^3}\] 
    (see~\cite[\S 4.2]{UniversalLiouvilleaction} for further definitions). While the energy convergence $I^L(\gamma_n) \rightarrow I^L(\gamma)$ is generally much weaker than the equivalences in Proposition \ref{Thm: equiv convergence H}, it does not seem to be known if the Schwarzian convergence of Lemma \ref{Lemma:SchwarziansLEConverge} is equivalent.  
    \end{remark}

\section{Comparison of Loewner energy and welding energy}\label{Sec:MainProof}
In this section we prove Theorem~\ref{Comparable} comparing the Loewner and welding energies.  In Section~\ref{sec: proofcomparableUnderanalytic} we show that this holds for all analytic curves when rooting at $\infty$, and then in \S\ref{Sec:ProofOfComparable} we invoke tools from Sections~\ref{Sec:welding energy}, \ref{Sec: BoundayNormalizedforLE}, and \ref{Sec: topologicalconvergence} to generalize to any root and all Weil--Petersson $\gamma$.  

We use this result in Section~\ref{Sec:WContinuous} to prove the continuity of the welding energy $y \mapsto W_h(y)$, as previously claimed in Section~\ref{Sec:welding energy}.  In  Section~\ref{Sec: cor1} we prove Corollary \ref{Thm: boundedcor}, showing $I^{L}(\gamma)$ is bounded from above by a constant depending only on $d_{\WP}(h,\text{id})$.

\subsection{Proof of Theorem \ref{Comparable} for analytic curves rooted at infinity}\label{sec: proofcomparableUnderanalytic}

Recall that $\gamma \subset \mathbb{C}$ is an \emph{analytic Jordan curve} if there exists a parametrization $\varphi: \partial \mathbb{D} \rightarrow \gamma$ such that $\varphi$ is conformal in a neighborhood of $\partial \mathbb{D}$ in $\mathbb{C}$.  This is equivalent to saying any conformal map $f$ from $\mathbb{D}$ to the interior of the domain $\Omega$ bounded by $\gamma$ extends to be conformal in a neighborhood of the closure $\overline{\mathbb{D}}$ of the unit disk \cite[Prop. 3.1]{Pommerenke}. Similarly, any conformal $g: \mathbb{D}^* \rightarrow \Omega^*$ to the exterior domain extends conformally to a neighborhood of $\overline{\mathbb{D}^*}$.

We call $\gamma \subset \hat{\mathbb{C}}$ an \emph{analytic Jordan curve $\gamma$ passing through $\infty$} when $\infty \in \gamma$ and there exists a M\"obius image $M(\gamma) \subset \mathbb{C}$ of it in the plane which is an analytic Jordan curve.  For such $\gamma$, the conformal maps $f: \mathbb{H} \rightarrow H$ and $g: \mathbb{H}^* \rightarrow H^*$ to the two domains $H \cup H^* = \C \bs \{\gamma\}$ defined by $\gamma$ extend to be conformal in some neighborhood $\{\, z \,:\, |\Im(z)|<\epsilon \,\}$ of $\mathbb{R}$ (although this is a strictly weaker condition).

The following lemma is at the heart of our argument for Theorem \ref{Comparable}.

\begin{lemma}\label{lem:smoothcrossterm}
Let $\g \subset \hat{\mathbb{C}}$ be an analytic Jordan curve passing through $\infty$, with $f: \mathbb{H} \rightarrow H$ and $g:\mathbb{H}^* \rightarrow H^*$ Riemann maps to either side of $\gamma$ which fix $\infty$, and $h := g^{-1} \circ f$ the corresponding conformal welding on $\mathbb{R}$. Then
\begin{align}\label{Eq:CrossTerm}
\begin{split}
    \brac{\log |f'|,\log |g'\circ h|}_{\half(\R)}+\brac{\log |g'|,\log |f'\circ h^{-1}|}_{\half(\R)} &= - \brac{\log|f'|, \log|h'|}_{\half(\R)} \\
    &=-\brac{\log|g'|, \log|(h^{-1})'|}_{\half(\R)}.
\end{split}
\end{align}
Furthermore, the inner products on the left side satisfy
\begin{align}
    \brac{\log |f'|,\log |g'\circ h|}_{\half(\R)} &= -\frac{1}{2\pi}\| \nabla \log|g'| \|^2_{L^2(\H^*)}, \; \quad \text{and}\label{Eq:IPDEg}\\
    \brac{\log |g'|,\log |f'\circ h^{-1}|}_{\half(\R)} &= -\frac{1}{2\pi}\| \nabla \log|f'| \|^2_{L^2(\H)}.\label{Eq:IPDEf}
\end{align}
\end{lemma}

\begin{remark}\label{Remark:MixedTermsNotEqual}
    As it is known that, in general,
    \begin{align*}
        \| \log|g'| \|^2_{L^2(\H^*)} \neq \| \log|f'| \|^2_{L^2(\H)},
    \end{align*}
    equations \eqref{Eq:IPDEg} and \eqref{Eq:IPDEf} show that the ``mixed terms'' appearing on the left-hand side of \eqref{Eq:CrossTerm} are not generally identical.  In contrast, the lemma says  the ``mixed terms'' on the right-hand side of \eqref{Eq:CrossTerm} are the same.
\end{remark}

\begin{proof}

    We begin with the first equality in \eqref{Eq:CrossTerm}.  To that end, let $P(F)$ and $P^*(F)$ denote the Poisson extensions $F:\mathbb{R} \rightarrow \mathbb{R}$ into $\mathbb{H}$ and $\mathbb{H}^*$, respectively. By \eqref{Eq:Parallel}, 
   \begin{align}
    \brac{\log |f'|,\log |g'\circ h|}_{\half} &= \brac{\log|f'|, P(\log|g'\circ h|)}_{\nabla(\mathbb{H})} \notag\\
    &= \frac{1}{2\pi}\int_\mathbb{H} \nabla \log|f'| \cdot \nabla P(\log|g'\circ h|) \, dx\,dy \notag\\
       &= \lim_{R \rightarrow \infty} \frac{1}{2\pi}\int_{B_R(0) \cap \mathbb{H}} \nabla \log|f'| \cdot \nabla P(\log|g'\circ h|) \, dx\,dy.\label{Eq:IPLimit}
   \end{align}

   For fixed $R$, Green's first identity and the harmonicity of $\log|f'|$ yield
   \begin{align}
       \frac{1}{2\pi}\int_{B_R(0) \cap \mathbb{H}} &\nabla \log|f'| \cdot \nabla P(\log|g'\circ h|) \, dx\,dy \notag\\
       &= \frac{1}{2\pi}\int_{-R}^R \log|g'\circ h|\,\partial_n \log|f'| dx + \frac{1}{2\pi}\int_{(\partial B_R(0)) \cap \mathbb{H}} P(\log|g'\circ h|)\, \partial_n \log|f'| ds, \label{Eq:GGSubdomain}
    \end{align}
    where $n$ is the outward-pointing normal vector to $B_R(0) \cap \mathbb{H}$.  We claim that the second integral is negligible when $R$ is large.  Indeed, analyticity of $\gamma$ yields that $g$ has a Laurent expansion of the form $a_1z + a_0 + a_{-1}/z + \cdots$ at $\infty$, and thus $g'(z) = a_1 + O(z^{-2})$ and $\log g'(z) = b_1 + O(z^{-2})$ as $z \rightarrow \infty$.  In particular, $\log|g'\circ h|$ is bounded on $\mathbb{R}$, showing
    \begin{align*}
        \Bigg| \frac{1}{2\pi}\int_{(\partial B_R(0)) \cap \mathbb{H}} P(\log|g'\circ h|)\, \partial_n \log|f'| ds \Bigg| \lesssim \int_{(\partial B_R(0)) \cap \mathbb{H}} \big|\partial_n \log|f'| 
        \big| ds.
    \end{align*}
    Now, $|\partial_n \log|f'| | \leq | \nabla \log |f'|| = |f''/f'| = O(|z|^{-3})$ as $z \rightarrow \infty$ by the similar expansion $\log f'(z) = c_1 + O(z^{-2})$ for $f$.  Thus passing to the $R \rightarrow \infty$ limit in \eqref{Eq:IPLimit} using \eqref{Eq:GGSubdomain}, we see
    \begin{align}\label{Eq:DirichletIPtoNormalDeriv}
        \langle\log|f'|, P(\log|g'\circ h|)\rangle_{\nabla(\H)}  = -\frac{1}{2\pi}\int_\R \log|g'\circ h|\,\partial_y \log|f'|\, ds.
    \end{align}
    By assumption, $f$ and $g$ are conformal in a neighborhood of $\R$, and thus $\log f'$ is analytic across $\R$ and $h'$ does not vanish.  The Cauchy-Riemann equations thus hold, yielding 
   \begin{align}
       -\frac{1}{2\pi}\int_\R \log|g'\circ h| \,\partial_y\log|f'| \,ds &= \frac{1}{2\pi}\int_\R \log|g'\circ h|\,\partial_x \arg f' \,ds \label{Eq:CrossTermPiece1a}\\
       &= -\frac{1}{2\pi}\int_\R \log|(g^{-1})'\circ f|\,\partial_x \arg f' \,ds \notag\\
       &= \frac{1}{2\pi}\int_\R (\log|f'| - \log|h'|)\partial_x\arg f' \,ds.\label{Eq:CrossTermPiece1}
   \end{align}
    
    Shifting to the second term on the left-hand side of \eqref{Eq:CrossTerm}, we apply the same logic as above to obtain
    \begin{align}\label{Eq:CrossTermShowEnergyA}
        \brac{\log|g'|, \log|f'\circ h^{-1}|}_{\half(\R)} &= \brac{\log|g'|, P^*(\log|f'\circ h^{-1}|)}_{\nabla(\mathbb{H}^*)} \notag\\
        &= \frac{1}{2\pi}\int_\R \log|f'\circ h^{-1}| \,\partial_y\log|g'|\,ds.
    \end{align}
    Harmonicity and a change of variables then yields
    \begin{align}
        \frac{1}{2\pi}\int_\R \log|f'\circ h^{-1}| \,\partial_y\log|g'|\,ds &= -\frac{1}{2\pi}\int_\R \log|f'\circ h^{-1}| \,\partial_x\arg g' \, ds \label{Eq:CrossTermAltStep}\\
        &= -\frac{1}{2\pi}\int_\R \log|f'| \big(\partial_x\arg g'\big)\circ h  \cdot h' \,ds \notag\\
        &= -\frac{1}{2\pi}\int_\R \log|f'| \,\partial_x\arg (g'\circ h)\,ds \notag\\
        &= -\frac{1}{2\pi}\int_\R \log|f'| \,\partial_x\arg (f') \, ds,\label{Eq:CrossTermShowEnergy}
    \end{align}
    and so summing with \eqref{Eq:CrossTermPiece1}, we obtain 
\begin{align}\label{Eq:CrossTermPenultimate1}
    \begin{split}
        \langle\log|f'|, \log|g'\circ h|&\rangle_{\half(\R)}+ \brac{\log|g'|, \log|f'\circ h^{-1}|}_{\half(\R)}\\
        &=- \frac{1}{2\pi}\int_\R (\partial_x \arg f') \log|h'| \,ds = -\frac{1}{2\pi}\int_\mathbb{R}\big(\partial_n\log|f'|\big)\log|h'| \,ds,
    \end{split}
\end{align}
where $\partial_n = -\partial_y$ is the outer normal derivative to $\H$.  Reversing the logic that gave us \eqref{Eq:DirichletIPtoNormalDeriv}, we find
\begin{align}\label{Eq:CrossTermLastStep}
    -\frac{1}{2\pi}\int_\mathbb{R}\big(\partial_n\log|f'|\big)\log|h'| \,ds &= -\langle\log|f'|, P(\log|h'|)\rangle_{\nabla(\H)} = - \brac{\log|f'|, \log|h'|}_{\half(\R)}.
\end{align}
(Here the same argument goes through since $\log|h'|$ is bounded on $\mathbb{R}$ by analyticity of $\gamma$.)  Combined with \eqref{Eq:CrossTermPenultimate1} we have the first equality in \eqref{Eq:CrossTerm}.

There are several ways to see the second equality in \eqref{Eq:CrossTerm}.  One method is to exchange the roles of the two integrals in the above argument.  Specifically, at \eqref{Eq:CrossTermPiece1a} do a change of variables $x \mapsto h^{-1}(x)$ and use $\arg f'\circ h^{-1} = \arg g'$.  Then for the second integral at \eqref{Eq:CrossTermAltStep}, use $\log|f'\circ h^{-1}| = \log|g'|-\log|(h^{-1})'|$.  This symmetric argument and yields the second equality in \eqref{Eq:CrossTerm}.\footnote{Another method is to replace $\gamma$ with its complex conjugate $\gamma^*$, which has conformal maps $F(z) = \overline{g(\bar{z})}$ and $G(z) = \overline{f(\bar{z})}$ from $\mathbb{H}$ and $\mathbb{H}^*$, respectively, and corresponding conformal welding $H(x) = h^{-1}(x)$.  Then apply the first equality in \eqref{Eq:CrossTerm} to $F,G,$ and $H$.}

By similar symmetry, it suffices to prove one of \eqref{Eq:IPDEg} and \eqref{Eq:IPDEf}, and we show the latter.  Indeed, by \eqref{Eq:CrossTermShowEnergyA} and \eqref{Eq:CrossTermShowEnergy},
\begin{align*}
    \brac{\log |g'|,\log |f'\circ h^{-1}|}_{\half} &= -\frac{1}{2\pi}\int_\R \log|f'| \,\partial_x\arg (f') \, ds\\
    &= \frac{1}{2\pi}\int_\R \log|f'| \,\partial_y \log|f'| \, ds\\
    &= -\frac{1}{2\pi}\int_\R \log|f'| \,\partial_n \log|f'| \, ds\\
    &= -\frac{1}{2\pi} \| \nabla \log|f'| \|_{L^2(\H)}^2,
\end{align*}
where the last step again follows from reversing the Green's-identity argument at the beginning of the proof.
\end{proof}
\begin{remark}\label{rmk: GeodesicCurvatureFormula}
    We can also prove this lemma by the geodesic curvature formula under the weaker assumption that $\gamma$ is smooth. More precisely, let $\partial_n$ and $\partial_{n^*}$ to be the outer normal derivatives on $\R= \partial \H= \partial \H^{*}$ with respect to $\H$ and $\H^{*}$, respectively, and $k(z), k^{*}(z)$  the geodesic curvature at $z \in \gamma= \partial H= \partial H^{*}$.  Instead of the Cauchy-Riemann equations, we can then use the geodesic curvature formulas
    \begin{align*}
       (\partial_{n} \log|f'|)(x)=k(f(x))|f'(x)| \quad \text{ and } \quad (\partial_{n^*} \log|g'|)(y)=k^*(g(y))|g'(y)|
    \end{align*}
    (see \cite[Appendix A]{Yilin19}) to handle the $\partial_{y} \log|f'|= -\partial_{n} \log|f'|$ and $\partial_{y} \log|g'|= \partial_{n^*} \log|g'|$ factors in \eqref{Eq:CrossTermPiece1a} and \eqref{Eq:CrossTermAltStep}. 
    The first equality in~\eqref{Eq:CrossTerm} then follows from a change of variables and the fact $k(z)+k^{*}(z) \equiv 0$ for all $z\in \gamma$, with a symmetric argument establishing the second equality.
\end{remark}

We can now prove Theorem \ref{Thm:LoewnerIP} in the case of $\gamma$ an analytic Jordan curve passing through $\infty$. \ref{lem:smoothcrossterm}.
\begin{cor}\label{Cor:LEEqualsCrossTerm}
    With $\gamma$, $f$, $g$, and $h$ as in the statement of Lemma \ref{lem:smoothcrossterm},
    \begin{align*}
        I^L(\gamma) &= 2\brac{\log|f'|, \log|h'|}_{\half(\R)}\\
        &= 2\brac{\log|g'|, \log|(h^{-1})'|}_{\half(\R)}\\
        &= -2\brac{\log |f'|,\log |g'\circ h|}_{\half(\R)} -2\brac{\log |g'|,\log |f'\circ h^{-1}|}_{\half(\R)}.
    \end{align*}
\end{cor}
\begin{proof}
    Recalling \eqref{Eq:LEHNabla}, compare \eqref{Eq:CrossTerm} with the sum of  \eqref{Eq:IPDEg} and \eqref{Eq:IPDEf}.
\end{proof}

\noindent We can also now prove Theorem \ref{Comparable} under the assumption of this high regularity.
\begin{lemma}\label{Lemma:ComparableAnalytic}
    Let $\g \subset \hat{\mathbb{C}}$ be an analytic Jordan curve passing through $\infty$ that is also a $K$-quasicircle, with associated conformal welding $h:\wR \rightarrow \wR$ that fixes $\infty$.  Then 
\begin{align*}
        \frac{1}{2}\Big(3+\frac{1}{K^2+K^{-2}}\Big)I^L(\gamma) \leq W_h(\infty) \leq \frac{1}{2}(3+K^2+K^{-2})\,I^L(\gamma).
\end{align*}
\end{lemma}
\begin{proof}
We have 
\begin{align*}
    W_h(\infty) &= \| \log |h'| \|_\half^2 + \| \log |(h^{-1})'| \|_\half^2\\
    &= \|\log |(g^{-1})'\circ f| + \log|f'|\|_{\half}^{2} + \|\log |(f^{-1})'\circ g| + \log |g'|\|_{\half}^{2},
\end{align*}
as we can freely differentiate the conformal maps on the boundary. Expanding the inner products and using \eqref{Eq:IntroDouglasFormula} or \eqref{Eq:LERHHalf} yields
\begin{align}
    W_h(\infty) &= \frac{1}{2} I^L(\gamma) + \|\log |(g^{-1})'\circ f|\|_{\half}^2 + \|\log |(f^{-1})'\circ g| \|_{\half}^{2} \notag\\
    &\hspace{13.5mm} +2 \brac{\log |(g^{-1})'\circ f|, \log|f'|}_{\half} + 2 \brac{\log |(f^{-1})'\circ g|, \log|g'|}_{\half} \notag\\
    &= \frac{1}{2} I^L(\gamma) + \|\log |g'\circ h|\|_{\half}^2 + \|\log |f'\circ h^{-1}| \|_{\half}^{2} \label{Eq:WExpandIP}\\
    &\hspace{13.5mm} -2 \brac{\log |g' \circ h|, \log|f'|}_{\half} - 2 \brac{\log |f'\circ h^{-1}|, \log|g'|}_{\half}.\notag
\end{align}
At this point, we apply Corollary \ref{Cor:OperatorBound}, \eqref{Eq:LERHHalf}, and Corollary \ref{Cor:LEEqualsCrossTerm} to conclude
\begin{align*}
    W_h(\infty) &\leq \frac{1}{2}I^L(\gamma) + (K^2+K^{-2}) \big( \|\log |g'|\|_{\half}^2 + \|\log |f'| \|_{\half}^{2}\big) + I^L(\gamma) \\
    &= \frac{1}{2}(3 + K^2+K^{-2})I^L(\gamma).
\end{align*}
The lower bound follows the same logic, but starts from \eqref{Eq:WExpandIP} by bounding $W_h(\infty)$ from below by writing
\begin{align*}
    \|\log |g'|\|_{\half}^2 + \|\log |f'| \|_{\half}^{2} &= \|\log |g'\circ h \circ h^{-1}|\|_{\half}^2 + \|\log |f'\circ h^{-1} \circ h| \|_{\half}^{2}\\
    &\leq (K^2+K^{-2})\big( \|\log |g'\circ h |\|_{\half}^2 + \|\log |f'\circ h^{-1}| \|_{\half}^{2} \big).
\end{align*}
The remaining steps are the same as for the upper bound.
\end{proof}

\subsection{Proofs for general Weil--Petersson curves and any root}\label{Sec:ProofOfComparable}

\subsubsection{Setup, argument sketch, and normalization}\label{Sec:ApproxNormalization}
We now consider the case of a general curve $\gamma = f(\hatR)$ of finite Loewner energy which passes through $\infty$, with associated conformal maps $f:\H \rightarrow H$, $g:\H^* \rightarrow H^*$  satisfying $f(\infty)=g(\infty) = \infty$. We approximate $\gamma$ with images under $f$ of larger and larger circles in $\H$.  After normalizing, each of the latter will be an analytic curves through $\infty$, and thus Lemma \ref{lem:smoothcrossterm} will apply.  We will argue in Lemma \ref{lem: curvesWPconvergence} that the approximating curves converge to $\gamma$ in the sense of Proposition \ref{Thm: equiv convergence H}, and it will follow that the conclusions of Lemma \ref{lem:smoothcrossterm} and Corollary \ref{Cor:LEEqualsCrossTerm} will pass through to the limiting curve $\gamma$.

\subsubsection{The approximating equipotentials}
We proceed to construct the approximating curves.  Let $C_n$ be the circle in $\mathbb{H}$ which has a diameter on the imaginary axis with endpoints $in$ and $i/n$, $n \in \{2,3,\ldots\}$.  Then 
\begin{align}\label{Eq:EquipotentialT_n}
    T_n(z):=i\frac{nz+i}{z+ni}
\end{align}
is the \Mob transformation from $\H$ to the bounded component $H_n$ of $\C \bs C_n$ which takes the triple $(0, i, \infty)$ to $(i/n, i, in)$.  See Figure \ref{Fig:Equipotentials}.  We wish to approximate $\gamma$ with the equipotentials $f(C_n) =: \gamma_n^*$, but need to further post-compose by \Mob transformations $M_n$ to normalize the $\gamma_n^*$ in the sense of \S\ref{Sec: topologicalconvergence}.  We construct these $M_n$ in two steps, first noting $T_n(1) = \frac{1+i/n}{1-i/n}$ and setting
\begin{align*}
    M_n^*(z) :=  \frac{f\big(\frac{1+i/n}{1-i/n} \big) - f(in)}{f\big(\frac{1+i/n}{1-i/n} \big)-f(i/n)}\cdot \frac{z-f(i/n)}{z-f(in)}.
\end{align*}
Thus $M_n^* \circ f \circ T_n$ fixes 0,1, and $\infty$.  Writing $\ell_n^*$ for the length of arc $M_n^* \circ f \circ T_n([0,1])$, we set $M_n := \frac{1}{\ell_n^*}M_n^*$, 
\begin{align}\label{Eq:NormalizedEquipotential}
    \gamma_n : = M_n \circ f \circ T_n(\R) = M_n \circ f (C_n),
\end{align} 
and name the associated map
\begin{align}\label{Eq:NormalizedEquipotentialMap}
    f_n : = M_n \circ f \circ T_n.
\end{align} 
Note that, by construction, $\gamma_n$ passes through 0 and $\infty$, and its positively-oriented arc-length parametrization $z_n$ satisfies $z_n(1) = \frac{1}{\ell_n^*} >0$.  Furthermore, $f_n(s) = z_n(s)$ for $s \in \{0,1,\infty\}$ by construction, and so $f_n$ is the unique normalized conformal map for $\gamma_n$ in the sense of \S\ref{Sec: topologicalconvergence}. What remains to show is that these $\gamma_n$ converge to $\gamma$ in the Weil--Petersson \Teich space.

\begin{figure}
    \centering
    \scalebox{0.85}{
	\begin{tikzpicture}

    \draw[fill=gray!15] (4.6,0.1) to[closed, curve through = {(4.5,2.5) (7.5,2.5) (10.2,2.5) (10.5,1.5) (10.2,0.5) (9.8,0.3) (8.1,0) (7.7,-0.08) (7.4,0) (6.2,0.5)} ] (5.4,0.3);
    \fill[white] (3.5,0) rectangle (5.4,3);
    \fill[white] (5,2) rectangle (10.5,3);
    \fill[white] (9.8,3) rectangle (10.75,0);
    \fill (7.4,0) circle (0.7mm);
    \node[yshift=3.5mm] at (7.4,0) {$0$};
    \fill (8.1,0) circle (0.7mm);
    \node[yshift=4.3mm] at (8.1,0) {$\frac{1}{\ell_n^*}$};
    \node at (6.3,0.8) {$\gamma_n$};
    \node at (9.3,1.2) {$\Omega_n$};
    \fill (7.6,1.3) circle (0.7mm);
    \draw[fill=white] (7.6,1.3) circle[radius=0.5mm];
    \draw[dashed] (7.4,-0.5) to[closed, curve through = {(8,-0.7) (8.6,-0.3) (7.6,-1.6) (5.8,-0.5)}] (6.6,-0.1);

    \fill [gray!15] (-2,0) rectangle (2.46,2);
    \node at (1.25,1) {$\H$};
    \fill (0,1) circle (0.7mm);
    \draw[fill=white] (0,1) circle[radius=0.5mm];
    \node[right] at (0,1) {$i$};
    \fill (0,0) circle (0.7mm);
    \fill (1,0) circle (0.7mm);
    \draw (-2,0) -- (1.6,0);
    \node at (2,0) {$\cdots$};
    \draw[thick,fill=white] (2.34,0.06) rectangle (2.46,-0.06);
    \node[yshift=-3.5mm] at (2.4,0) {$\infty$};
    \node[yshift=-3.5mm] at (0,0) {$0$};
    \node[yshift=-3.5mm] at (1,0) {$1$};

    \draw [-{Latex}] (0, 2.5) to ++(0, 1);
    \node[right] at (0,3) {$T_n$};
    \draw[-{Latex}] (7,3.5) to ++(0,-1);
    \node[right] at (7,3) {$M_n$};
    
    \draw[->,dashed] (-2,4) -- (2,4);
    \draw[fill=gray!15] (0,5.75) circle[radius=1.5];
    \fill (0,4.25) circle (0.7mm);
    \node[above] at (0,4.25) {$i/n$};
    \fill (0,4.25) circle (0.7mm);
    \fill (0.88,4.52) circle (0.7mm);
    \fill (0,5.25) circle (0.7mm);
    \draw[fill=white] (0,5.25) circle[radius=0.5mm] node[right] {$i$};
    \draw[thick,fill=white] (-0.06,7.19) rectangle (0.06,7.31);
    \node[yshift=3mm] at (0,7.25) {$in$};
    \node at (0,6.1) {$H_n$};
    \node at (1.65,5) {$C_n$};

    \draw[-{Latex}] (3.4,5.75) -- ++(1,0);
    \node[above] at (3.9,5.75) {$f$};
    \draw[-{Latex}] (3.4,1) -- ++(1,0);
    \node[below] at (3.9,1) {$f_n$};

    \draw[->,dashed] (5.4,5.15) .. controls (7,6.75) and (8.5,1.75) .. (9.8,4.3);
    \node at (6.1,5.1) {$\gamma$};
    \draw[fill=gray!15] (7,6) to[closed, curve through = {(6,6.5) (8,7) (9,7) (9.5,5) (8.5,4) }] (8,5.3);
    \foreach \pt in {(8.5,4), (8,5.3)} \fill \pt circle (0.7mm);
    \draw[thick, fill=white] (9.44,4.94) rectangle (9.56,5.06);
    \fill (8.6,4.8) circle (0.7mm);
    \draw[fill=white] (8.6,4.8) circle[radius=0.5mm];
    \node[above] at (7,6.1) {$\gamma_n^*$};
    \end{tikzpicture}
    }
    \caption{Approximating $\gamma$ with normalized equipotentials $\gamma_n = M_n \circ f(C_n)$.}
    \label{Fig:Equipotentials}
\end{figure}

\subsubsection{The approximation argument}
We first observe that if $\gamma$ is a $K$-quasicircle, then so are all of the $\gamma_n$. 
\begin{lemma}\label{Lemma:Equipotentials}
    Let $\gamma$ be a $K$-quasicircle  passing through $\infty$, and $f: \mathbb{H} \rightarrow \Omega$ a conformal map to one side of $\C\bs \gamma$ which fixes $\infty$.  Then the equipotentials $\gamma_n$ defined in \eqref{Eq:NormalizedEquipotential} are also all $K$-quasicircles.
\end{lemma}

\begin{proof}
    Since $\gamma$ is a $K$-quasicircle, the conformal map $f$ has a $K^2$-quasiconformal extension to the plane \cite[Ch.I \S6.2]{Lehtobook}.  Thus a conformal map for $\gamma_n^* = f(C_n)$ also has a $K^2$-quasiconformal extension to $\mathbb{C}$, and so by Smirnov's canonical Beltrami representation theorem \cite{Smirnov}, each $\gamma_n^*$ is a $K$-quasicircle.  Applying the \Mob transformation $M_n$ does not affect this.
\end{proof}

\noindent Our main approximation result is the following.  We recall that convergence $\gamma_n \rightarrow \gamma$ in the normalized Weil--Petersson \Teich space is defined by the equivalences in Proposition \ref{Thm: equiv convergence H}.
\begin{lemma}\label{lem: curvesWPconvergence}
    Let $\gamma$ be a Jordan curve of finite Loewner energy which passes through $\infty$, and $f: \mathbb{H} \rightarrow \Omega$ a conformal map to one side of $\C\bs \gamma$ which fixes $\infty$.  Further assume that both $\gamma$ and $f$ are normalized according to Definition \ref{Def:NormalizedCurve}. Then the normalized equipotentials $\gamma_n$ constructed in \eqref{Eq:NormalizedEquipotential} satisfy
    \begin{align*}
        \lim_{n \rightarrow \infty} \gamma_n = \gamma,
    \end{align*}
    where the limit is in the normalized Weil--Petersson \Teich space.
\end{lemma}

\begin{proof}
   By Proposition \ref{Thm: equiv convergence H}$(\ref{Lemma:TopologyPSH})$, it suffices to show 
    \begin{align}\label{Eq:EquipotentialApproxGoal}
        \lim_{n \rightarrow \infty}\| \mathcal{A}_{f_n} - \mathcal{A}_f\|_{L^2(\mathbb{H})} = 0,
    \end{align}
    where $f_n: \H \rightarrow \Omega_n$ is the normalized conformal maps for $\gamma_n$ constructed above in \eqref{Eq:NormalizedEquipotentialMap}, with $\Omega_n$ as in Figure \ref{Fig:Equipotentials}.  Noting that
    \begin{align*}
        \calA_{f_n} = \calA_{\frac{z - f(i/n)}{z-f(in)} \circ f \circ T_n},
    \end{align*}
    we use the pre-Schwarzian composition rule (see Table \ref{Table:ABS}) to compute 
    \begin{align}
        \| \mathcal{A}_{f_n} - \mathcal{A}_f\|_{L^2(\mathbb{H})} &= \Big\| \frac{-2(f'\circ T_n)T_n'}{f\circ T_n - f(in)} +  (\mathcal{A}_f\circ T_n) T_n'  + \mathcal{A}_{T_n} - \mathcal{A}_f\Big \|_{L^2(\mathbb{H})}\notag\\ 
        &\leq \Big\| (\mathcal{A}_f\circ T_n)T_n' - \mathcal{A}_f  \Big\|_{L^2(\mathbb{H})} + \Big\| \frac{-2(f'\circ T_n)T_n'}{f\circ T_n - f(in)} + \mathcal{A}_{T_n} \Big\|_{L^2(\mathbb{H})}.\label{Ineq:ApproximationsBound}
    \end{align}
    We claim both these terms are vanishing.  
    
    We first show the former tends to zero by Proposition \ref{Thm: DCL}. Indeed, as $T_n(z) \rightarrow z$ locally uniformly, we have $(\mathcal{A}_f\circ T_n(z))T_n'(z) - \calA_f(z) \rightarrow 0$ point-wise on $\H$. Furthermore,
    \[ \left|(\mathcal{A}_f\circ T_n)T_n'-\mathcal{A}_f\right|^2 \leq 2\left(|\mathcal{A}_f\circ T_n)T_n'|^2+|\mathcal{A}_f|^2\right) =: G_n(z) \rightarrow 4|\calA_f(z)|^2 =: G(z)
    \]
   point-wise. Recalling our above notation $H_n = T_n(\H)$ for the interior of the circle $C_n$, we thus have 
    \begin{align*}
        \int_\H G_n \,dA &= 2\int_{\H} \left|\frac{f''}{f'}(T_n)T_n'\right|^2 dA + 2\int_{\H} \left|\frac{f''}{f'}\right|^2dA \\ &= 2\int_{H_n} \left|\frac{f''}{f'}\right|^2 dA +2\int_{\H} \left|\frac{f''}{f'}\right|^2 dA \rightarrow \int G\,dA
    \end{align*} 
    as $n \rightarrow \infty$.    We conclude $\| (\mathcal{A}_f\circ T_t)T_t' - \mathcal{A}_f \|_{L^2(\mathbb{H})} \rightarrow 0$ by Proposition~\ref{Thm: DCL}.

The vanishing of the second term in \eqref{Ineq:ApproximationsBound} appears to be less trivial, and we start by observing
\begin{align}
         \Big\| \frac{-2(f'\circ T_n)T_n'}{f\circ T_n - f(in)} + \mathcal{A}_{T_n} \Big\|_{L^2(\H)}^2
         &=\int_{\H} \abs{\frac{2(f'\circ T_n)(z)}{f\circ T_n(z)-f(in)}-\frac{\mathcal{A}_{T_n}(z)}{T_n'(z)} }^2\abs{T_n'(z)}^2 dA(z) \notag\\
         &=\int_{H_n} \abs{\frac{2f'(w)}{f(w)-f(in)}-\frac{\mathcal{A}_{T_n}(T_n^{-1}(w))}{T_n'(T_n^{-1}(w))} }^2 dA(w) \notag\\
         &=\int_{H_n} \abs{\frac{2f'(w)}{f(w)-f(in)}+ \mathcal{A}_{T_n^{-1}}(w)}^2 dA(w) \notag\\
         &=4\int_{H_n} \abs{\frac{f'(w)}{f(w)-f(in)}-\frac{1}{w-in}}^2 dA(w),\label{Eq:EquipotentialApproxBoundThisA}
    \end{align}
    where we have used the pre-Schwarzian chain rule from Table \ref{Table:ABS} for the third equality, and the explicit expression \eqref{Eq:EquipotentialT_n} for $T_n$ for the fourth.
    
    To show that \eqref{Eq:EquipotentialApproxBoundThisA} limits to zero, it suffices to show that for each subsequence $\gamma_{n_m}$, there is a further subsequence on which it vanishes.  To that end, let $\gamma_{n_m}$ be any subsequence, which we relabel as $\gamma_n$.  Since $f$ fixes $\infty$, note that the integrand in \eqref{Eq:EquipotentialApproxBoundThisA} point-wise vanishes as $n \rightarrow \infty$.  We use Proposition \ref{Thm: DCL} again, combined with Theorem \ref{Thm:BoundaryLE} and a result on the convergence of the energy of nested equipotentials in the disk, \cite[Cor. 1.5]{Interplay}, to show that the entire integral tends to zero.  
    
    To construct the dominating functions $G_n$ (we are re-using the same notation $G_n$ as in the first part of the proof, but now the functions will be different), we observe
    \begin{align*}
        \Big|\frac{f'(w)}{f(w)-f(in)}&-\frac{1}{w-in}\Big|^2 = \Big|\frac{f'(w)}{f(w)-f(in)}-\frac{1}{w-in}\Big|^2 \chi_{H_n}(w)\\
        &\leq 2\abs{\frac{f'(w)}{f(w)-f(in)}-\frac{1}{w-in} -\frac{1}{2}\frac{f''(w)}{f'(w)}}^2\chi_{H_n}(w) + \frac{1}{2}\abs{\frac{f''(w)}{f'(w)}}^2\chi_{H_n}(w)\\
        &= 2|\calB_f(w,in)|^2\chi_{H_n}(w) + \frac{1}{2}|\calB_f(w,\infty)|^2\chi_{H_n}(w)\\
        &\leq 2|\calB_f(w,in)|^2 \chi_{H_n}(w) + 2|\calB_{g_n}(w,h_n(in))|^2 \chi_{H_n^*}(w) + \frac{1}{2}|\calB_f(w,\infty)|^2 \chi_{H_n}(w)\\
        &=: G_n(w),
    \end{align*}
where $g_n: H_n^*:= \hatC \bs \overline{H_n} \rightarrow \hatC \bs \overline{f(H_n)} =: \Omega_n^*$ is a conformal map from the outside of the circle $C_n$ to the outside of the circle's image $\gamma_n^* = f(C_n)$ under $f$, and $h_n := g_n^{-1} \circ f \in \ophom{C_n}$ the associated conformal welding.  With respect to normalization, we select $g_n$ to be the unique map satisfying
\begin{align}\label{Eq:MoveDomainToH^*}
    g_n(i/n) = f(i/n), \qquad g_n(T_n(1)) = f(T_n (1)), \quad \text{and} \quad 
    g_n(in) = f(in).
\end{align} 

Thus $h_n$ fixes $i/n, T_n(1),$ and $in$. To show that the integral in \eqref{Eq:EquipotentialApproxBoundThisA} tends to zero along a further subsequence $\gamma_{n_\ell}$, by Proposition \ref{Thm: DCL} it suffices to show that $(i)$ $G_{n_\ell} \in L^1(\C)$, $(ii)$ $G_{n_\ell}$ has a point-wise limit $G \in L^1(\C)$, and that $(iii)$ $\int_\C G_{n_\ell} \rightarrow \int_\C G$. 
\begin{enumerate}[$(i)$]
    \item By Theorem \ref{Thm:BoundaryLE},
    \begin{align*}
        \int_\C G_n(w) dA(w) &= 2I^L(\gamma_n^*) + \frac{1}{2}\int_{H_n} |\calB_f(w,\infty)|^2 dA(w)\\
        &\leq 2I^L(\gamma_n^*) + \frac{1}{2}\int_{\H} |\calB_f(w,\infty)|^2 dA(w)\\
        &\leq 2I^L(\gamma_n^*) + \frac{1}{2}I^L(\gamma) < \infty
    \end{align*}
    since the $\gamma_n^*$ are all analytic arcs.  Hence $\{G_n\} \subset L^1(\C)$.
    \item We claim that, along a further subsequence $\{\gamma_{n_\ell}\}$, we have the point-wise limit
    \begin{align}\label{Limit:EquipotentialsGMapsPtwise}
        G_{n_\ell}(w) \rightarrow 2|\calB_f(w,\infty)|^2 \chi_{\H}(w) + 2|\calB_{g}(w,\infty)|^2 \chi_{\H^*}(w) + \frac{1}{2}|\calB_f(w,\infty)|^2 \chi_{\H}(w) =: G(w)
    \end{align}
    a.e. $w \in \C$, where $g$ is a subsequential limit of the $g_{n_\ell}$.  We proceed to first show that such a limit $g$ exists and is a conformal map from $\H^*$ to $\Omega^* := \C \bs \overline{f(\H)}$, and then subsequently show the claimed point-wise convergence in \eqref{Limit:EquipotentialsGMapsPtwise}.

    \quad Indeed, by Lemma \ref{Lemma:Equipotentials}, the compositions $\tilde{g}_n := g_n \circ T_n : \H^* \rightarrow \Omega_n^*$ are conformal maps to one side of the $K$-quasicircles $\gamma_n^*$, and thus each of the $\tilde{g}_n$ extends as a $K^2$-quasiconformal map of $\C$, which we also denote by $\tilde{g}_n$.  Given our normalization \eqref{Eq:MoveDomainToH^*} of $g_n$ and the convergence of $T_n$ to the identity map (and that $f$ extends continuously to $\R$), the images of $0,1$, and $\infty$ under the family $\{\tilde{g}_n\}$ are mutually positively separated, thus showing $\{\tilde{g}_n\}$ is a normal family of quasiconformal mappings \cite[Ch.I Thm. 2.1]{Lehtobook}, with subsequential limit $\tilde{g}_{n_\ell} \rightarrow {g}$ for some ${g}: \C \rightarrow \C$ . Since $g(0)=f(0)$ and $g(1) = f(1)$, $g$ is not constant, and is thus also a $K^2$-quasiconformal map \cite[Ch.I Thm. 2.2]{Lehtobook}, which is furthermore conformal on $\H^*$ by Hurwitz's theorem.\footnote{Or, if you will, by the convergence of the complex dilatations $\mu_{\tilde{g}_n} \equiv 0$ on $\H^*$ \cite[Ch.I Thm. 4.5]{Lehtobook}.}  
    
   \quad Next, note that $g: \H^* \rightarrow \Omega^*$ conformally.  Indeed, any $w\in \Omega^*$ is also in $H_n^*$ for all $n$ by monotonicity of the domains, and there is therefore a fixed curve $\eta_w \subset \H^*$ such that $\tilde{g}_{n_\ell}(\eta_w)$ winds once around $w$ for large $\ell$ (we can construct $\eta_w$ by using the locally-uniform convergence of the inverse family $\{\tilde{g}_{n_\ell}^{-1}\}$, for instance).  Thus $g(\eta_w)$ also winds once around $w$, and so $g$ maps onto $\Omega^*$ by the argument principle.  Similarly, for $w \notin \Omega^*$ the image of any large curve $\eta \subset \H^*$ under $\tilde{g}_{n_\ell}$ does not wind around $w$, and so similarly $w \notin g(\H^*)$.    
    
    \quad Having thus established the existence of the locally-uniform limit $g:\C \rightarrow \C$ whose restriction to $\H^*$ is a conformal map to $\Omega^*$, we lastly show  the point-wise convergence \eqref{Limit:EquipotentialsGMapsPtwise} for a.e. $w \in \C$.  Recalling $f(\infty) = \infty$, the convergence 
    \begin{align}\label{Lim:NPSf}
        \calB_f(w,in) = \frac{f'(w)}{f(w) - f(in)} - \frac{1}{w-in} - \frac{1}{2} \frac{f''(w)}{f'(w)} \rightarrow - \frac{1}{2} \frac{f''(w)}{f'(w)} = \calB_f(w,\infty)
    \end{align}
    for $w \in \H$ is clear, and hence    
    \begin{align*}
        &2|\calB_f(w,in)|^2 \chi_{H_n}(w) + \frac{1}{2}|\calB_f(w,\infty)|^2 \chi_{H_n}(w)\\
        \rightarrow &2|\calB_f(w,\infty)|^2 \chi_{\H}(w)  + \frac{1}{2}|\calB_f(w,\infty)|^2 \chi_{\H}(w)
    \end{align*}
    a.e. $w \in \C$.  To obtain \eqref{Limit:EquipotentialsGMapsPtwise}, it remains to show    
    \begin{align}\label{Lim:NPSg_n}
       |\calB_{g_{n_\ell}}(w,h_{n_\ell}(in_\ell))|^2 \chi_{H_{n_\ell}^*}(w)   \rightarrow |\calB_{g}(w,\infty)|^2 \chi_{\H*}(w).
    \end{align}
    Since $h_n(in) = in$ by \eqref{Eq:MoveDomainToH^*}, and furthermore $g(\infty) = \infty$ since
    \begin{align*}
        \tilde{g}_n(\infty) = f(in) \rightarrow \infty,
    \end{align*}
    \eqref{Lim:NPSg_n} is clear from the definition of $\calB_{g_{n_\ell}}$, as above for $f$ in \eqref{Lim:NPSf}. We conclude $G_{n_\ell}(w) \rightarrow G(w)$ a.e. on $w \in \C$, as claimed.

    \item 
    We lastly show $\int_\C G_n \rightarrow \int_\C G$. Since 
    \begin{align*}
        \frac{1}{2}\int_{H_n} |\calB_f(w,\infty)|^2dA(w) \rightarrow \frac{1}{2}\int_\H|\calB_f(w,\infty)|^2dA(w),
    \end{align*}
    to obtain $\int_\C G_n \rightarrow \int_\C G$ it suffices to show 
    \begin{align*}
        2\int_\H |\calB_f(w,\infty)|^2 dA(w) + 2\int_{\H^*}|\calB_{g}(w,\infty)|^2 dA(w) = \frac{\pi}{2}I^L(\gamma)
    \end{align*}
    is the limit of 
    \begin{align*}
        2\int_{H_n} |\calB_f(w,in)|^2 dA(w) + 2\int_{H_n^*}|\calB_{g_n}(w,h_n(in))|^2 dA(w) = \frac{\pi}{2}I^L(\gamma_n^*),
    \end{align*}
    where the equalities are by Theorem \ref{Thm:BoundaryLE}.  To see this, write $S: \H \rightarrow \D$ for the Cayley transform $S(z) = \frac{z-i}{z+i}$, and note the images $S(C_n)$ of the circles $C_n$ are the concentric circles $\frac{1-1/n}{1+1/n}\SSS =: r_n \SSS$ in $\D$. Therefore, by the invariance of Loewner energy under automorphisms of $\hatC$, and the convergence of nested, circular equipotentials \cite[Cor. 1.5]{Interplay}, we have
    \begin{multline*}
        I^L(\gamma_n^*) = I^L(f(C_n)) = I^L(S\circ f \circ S^{-1} (r_n \T))\\
        \rightarrow I^L(S\circ f \circ S^{-1} (\T)) = I^L(f (\R)) = I^L(\gamma).
    \end{multline*}
    We thus have $\int_\C G_n \rightarrow \int_\C G$. 
\end{enumerate}
By Proposition \ref{Thm: DCL} we conclude \eqref{Eq:EquipotentialApproxBoundThisA} tends to zero, thus showing both terms in \eqref{Ineq:ApproximationsBound} asymptotically vanish.
\end{proof}

\subsubsection{Proofs of Lemma \ref{lem:generalcrossterm} and Theorem \ref{Comparable} via approximation}\label{Sec:ProofViaApprox}

Let $\gamma$ be a general curve of finite Loewner energy that passes through $\infty$.  We push Lemma \ref{lem:smoothcrossterm} and its corollary through to this general case by means of Lemma \ref{lem: curvesWPconvergence} and the framework of \S \ref{Sec: topologicalconvergence}.  The latter requires normalized curves in the sense of Definition \ref{Def:NormalizedCurve}, however, and thus we begin by post-composing by an affine map $A \in \PSL_2(\C)$ such that $\Gamma := A(\gamma)$ also passes through $0$, and such that $\Gamma$'s positively-oriented arc-length parametrization $z$ with $z(0) = 0$ satisfies $z(1)>0$.  We also want the associated conformal maps to be normalized, and hence select affine $B,C \in \PSL_2(\R)$ such that $F:= A \circ f \circ B$ and $G:= A \circ g \circ C$ satisfy $F(s) = z(s) = G(s)$ for $s =0,1,\infty$.  Thus $F$ and $G$ are the unique normalized conformal maps from $\H$ and $\H^*$ to the two sides of $\C \bs \Gamma$.

We claim we can operate with $\Gamma$, $F$, and $G$ in place of $\gamma, f$ and $g$ without loss of generality, in the following sense.  If $h = g^{-1} \circ f$ was the original conformal welding, the new welding is $H := G^{-1} \circ F =  C^{-1} \circ h \circ B$, and it is not hard to see that all of the inner products and norms appearing in Lemma \ref{lem:generalcrossterm} and Corollary \ref{Thm:LoewnerIP} are invariant under these affine changes of coordinates.  For instance,
\begin{align*}
    \brac{\log |F'|,\log |G'\circ H|}_{\half} &= \brac{\log |f'|,\log |g'\circ h|}_{\half}, \\
    \qquad \langle \log |G'|,\log |F'\circ H^{-1}| \rangle_{\half} &= \langle \log |g'|,\log |f'\circ h^{-1}| \rangle_{\half},
\end{align*}
and similarly for the others.  Thus we may suppose that both the curve and the associated conformal maps are normalized in the sense Definition \ref{Def:NormalizedCurve}, and we proceed to revert to the original names $\gamma, f$, and $g$ for these now-normalized objects.

\begin{proof}[Proof of Lemma~\ref{lem:generalcrossterm}]
As discussed above, we may assume that $\gamma,f$ and $g$ are all normalized according to Definition \ref{Def:NormalizedCurve}.  Let $\gamma_n$ be the normalized approximating equipotentials from Lemma \ref{lem: curvesWPconvergence}, and $f_n,g_n$ the associated normalized conformal maps.  The identities \eqref{Eq:CrossTermGeneral}, \eqref{Eq:IPDEgGeneral}, and \eqref{Eq:IPDEfGeneral} all hold for $f_n, g_n,$ and $h_n$ by Lemma \ref{lem:smoothcrossterm}.  Taking limits in these, we claim, yields the identities for $f,g,$ and $h$. For instance, considering the left-most inner product of \eqref{Eq:CrossTermGeneral}, we observe
\begin{align*}
    \big|  \langle \log |f_n'|,&\log |g_n'\circ h_n| \rangle_{\half} - \brac{\log |f'|,\log |g'\circ h|}_{\half} \big|\\
    &= \big|  \brac{\log |f_n'|-\log|f'|,\log |g_n'\circ h_n|}_{\half} + \brac{\log |f'|,\log |g_n'\circ h_n| - \log|g'\circ h|}_{\half} \big|\\
    &\leq \big\| \log |f_n'|-\log|f'| \big\|_{\half} \big\| \log |g_n'\circ h_n|\big\|_{\half} + \big\| \log|f'| \big\|_{\half} \big\| \log |g_n'\circ h_n| - \log|g'\circ h|\big\|_{\half}\\
    &\leq \big\| \log |f_n'|-\log|f'| \big\|_{\half} C \big\| \log |g_n'|\big\|_{\half} + \big\| \log|f'| \big\|_{\half} \big\| \log |g_n'\circ h_n| - \log|g'\circ h|\big\|_{\half}\\
    &= \frac{C}{2\pi}\big\| \nabla\log |f_n'|-\nabla \log|f'| \big\|_{L^2(\H)}  \big\| \nabla \log |g_n'|\big\|_{L^2(\H)}\\
    &\hspace{35mm}+ \frac{1}{\sqrt{2\pi}}\big\| \nabla \log|f'| \big\|_{L^2(\H)} \big\| \log |g_n'\circ h_n| - \log|g'\circ h|\big\|_{\half},
\end{align*}
where we have used Corollary \ref{Cor:OperatorBound} and the Douglas formula \eqref{Eq:IntroDouglasFormula}.  By Proposition \ref{Thm: equiv convergence H}$(\ref{Lemma:TopologyPSH})$ and $(\ref{Lemma:TopologyPSH*})$ the first term tends to zero. For the second term, we note
\begin{align*}
    \log |g_n'\circ h_n| = \log \frac{1}{|(g_n^{-1})'\circ f_n|} = \log \frac{
|f_n'|}{|h_n'|},
\end{align*}
while similarly $\log|g'\circ h| = \log\frac{|f'|}{|h'|}$ a.e., and therefore
\begin{align*}
    \big\| \log |g_n'\circ h_n| - \log|g'\circ h|\big\|_{\half} &\leq \big\| \log |f_n'| - \log|f'|\big\|_{\half} + \big\| \log|h'| - \log|h_n'|\big\|_\half\\
    &=\frac{1}{\sqrt{2\pi}}\big\| \nabla\log |f_n'|-\nabla \log|f'| \big\|_{L^2(\H)}+ \big\| \log|h'| - \log|h_n'|\big\|_\half \\&\rightarrow 0
\end{align*}
by Proposition \ref{Thm: equiv convergence H}$(\ref{Lemma:TopologyPSH})$ and $(\ref{Lemma:TopologyWelding})$.  We conclude $\langle \log |f_n'|,\log |g_n'\circ h_n| \rangle_{\half} \rightarrow \brac{\log |f'|,\log |g'\circ h|}_{\half}$. Similar arguments show
\begin{align*}
    \brac{\log |g_n'|,\log |f_n'\circ h_n^{-1}|}_{\half} &\rightarrow \brac{\log |g'|,\log |f'\circ h^{-1}|}_{\half},\\
    \brac{\log|f_n'|, \log|h_n'|}_{\half} & \rightarrow \brac{\log|f'|, \log|h'|}_{\half},\\
    \brac{\log|g_n'|, \log|(h_n^{-1})'|}_{\half} &\rightarrow \brac{\log|g'|, \log|(h^{-1})'|}_{\half}.
\end{align*}
Also, the right-hand sides of equations \eqref{Eq:IPDEgGeneral} and \eqref{Eq:IPDEfGeneral} converge by Proposition \ref{Thm: equiv convergence H}$(\ref{Lemma:TopologyPSH*})$ and $(\ref{Lemma:TopologyPSH})$, respectively.  Taking limits in Lemma \ref{lem:smoothcrossterm} thus yields Lemma \ref{lem:generalcrossterm}, as claimed.

\end{proof}

\begin{proof}[Proof of Theorem~\ref{Comparable}.]
By Theorem \ref{Thm:WFiniteEquiv} all three parts of the inequality are infinite if $\gamma$ is not Weil--Petersson, so we may assume $I^L(\gamma) < \infty$.  

We first prove \eqref{Ineq:Main} for rooting at $y=\infty$, and we again claim that here we may assume, without loss of generality, that $\gamma$, $f$, and $g$ are normalized according to Definition \ref{Def:NormalizedCurve}.  Indeed, the \Mob post-composition does not change the Loewner energy, and as discussed at the outset of \S\ref{Sec:ProofViaApprox}, the new welding after normalization is $H = C^{-1}\circ h \circ B$ for some affine $B,C \in \PSL_2(\R)$.  Hence $W_H(\infty) = W_h(\infty)$ by Theorem \ref{Thm:WMobiusInvariance}, and so we may indeed assume the curve and conformal maps are normalized, as claimed.

By Lemma \ref{Lemma:ComparableAnalytic}, \eqref{Ineq:Main} holds for each of the approximating equipotentials of Lemma \ref{lem: curvesWPconvergence}, and as $n \rightarrow \infty$, $I^L(\gamma_n) \rightarrow I^L(\gamma)$ by Lemma \ref{Lemma:SchwarziansLEConverge} and $\| \log|h_n'| - \log|h'|\|_\half \rightarrow 0$ by Proposition \ref{Thm: equiv convergence H}$(\ref{Lemma:TopologyWelding})$.  Furthermore, $\|\log|(h_n^{-1})'| - \log|(h^{-1})'|\|_\half \rightarrow 0$ as well since the conjugated curves $\overline{\gamma}_n$  converge to $\overline{\gamma}$ in the normalized Weil--Petersson \Teich space (reflect the conformal maps and again applying Proposition \ref{Thm: equiv convergence H}$(\ref{Lemma:TopologyPSH})$, for instance).  Hence $W_{h_n}(\infty) \rightarrow W_h(\infty)$, yielding \eqref{Ineq:Main} when $y=\infty$.

For a general root $y \in \R$, select $T_y, T_{h(y)} \in \PSL_2(\R)$ such that $T_y(\infty) = y$ and $T_{h(y)}(\infty) = h(y)$.  Then the welding $H := T_{h(y)}^{-1} \circ h \circ T_y$ fixes $\infty$ and corresponds to the same Jordan curve $\gamma$, and thus 
    \begin{align*}
        \frac{1}{2}\Big(3+\frac{1}{K^2+K^{-2}}\Big)I^L(\gamma) \leq W_H(\infty) \leq \frac{1}{2}\Big(3+K^2+K^{-2}\Big)I^L(\gamma)
    \end{align*}
    by the previous paragraphs.  However, $W_H(\infty) = W_h(y)$ by Theorem \ref{Thm:WMobiusInvariance}.    
\end{proof}

\subsection{Continuity of welding energy in the root}\label{Sec:WContinuous}

Let $C \subset \hatC$ be a circle and $h \in \ophom{C}$.  In this section we cycle back to prove Theorem \ref{Thm: WisContinuousatRoot}, the continuity of $y \mapsto W_h(y)$. While in some sense this is a question in classical analysis, it is also a question intertwined with the Weil--Petersson \Teich space, since by Theorem \ref{Thm:WFiniteEquiv} $W_h(y)$ is finite if and only if $h \in \WP(C)$.  Hence it is not entirely unreasonable to prove properties of $W_h(y)$ using energy-specific tools, and that is the approach we follow.

We recall the continuity of $y \mapsto W_h(y)$ immediately follows from the continuity of $y \mapsto K_h(y)$, Theorem \ref{Thm: KisContinuousatRoot}.  We prove the latter by using the \Mob covariance of $K_h(y)$ to reduce to the case $y=\infty$, and then apply the equivalences of Proposition \ref{Thm: equiv convergence H}.

\begin{proof}[Proof of Theorem~\ref{Thm: KisContinuousatRoot}]       
    If $I^L(\gamma) = \infty$, then $W_h(y) = \infty$ and every $W_h(y_n) = \infty$ by Theorem \ref{Thm:WFiniteEquiv}, so there is nothing to prove.  We thus assume $I^L(\gamma)< \infty$. 
       
    Take $\{y_n\} \subset 
 \R$ such that $y_n \rightarrow y \in \hatR$. Let $f$ and $g$ be conformal on $\H$ and $\H^*$, respectively, such that $h = g^{-1}\circ f$.  As we have no assumptions on the normalization of $h$ or $\gamma$, we start by M\"obius-rotating $\gamma$ on the sphere, and choosing a normalization for each $y_n$ to match Definition \ref{Def:NormalizedCurve}.  
 
 Indeed, first suppose $f(y) \neq \infty$.  Choose $z_0 \in \gamma \bs \big(\{f(y_n)\}_{n\geq 1} \cup \{f(y)\}\big)$, and set
\begin{align*}
    \tilde{M}_y(z) := \frac{z-z_0}{z-f(y)} \quad \text{ and } \quad \tilde{\gamma}_y := \tilde{M}_y(\gamma),
\end{align*}
while similarly setting 
\begin{align*}
    \tilde{M}_n(z) := \frac{z-z_0}{z-f(y_n)} \quad \text{ and } \quad \tilde{\gamma}_n := \tilde{M}_n(\gamma).
\end{align*}
Write $\tilde{f}_y := \tilde{M}_y \circ f$ and $\tilde{f}_n := \tilde{M}_n \circ f$, and let $\tilde{z}_y$ and $\tilde{z}_n$ be the positively-oriented arc-length parametrizations of $\tilde{\gamma}_y$ and $\tilde{\gamma}_n$ which satisfy $\tilde{z}_y(0) = \tilde{z}_n(0)=0$.\footnote{The orientation of $\tilde{z}_y$ is determined by $\tilde{f}_y$, in the sense that there is some $x>0$ such that $\tilde{f}_y(x) = \tilde{z}_y(1)$. Similarly for $\tilde{z}_n$.}  Writing 
\begin{align*}
    g_y(x) := \int_{x_0}^x |\tilde{f}_y'(s)|ds
\end{align*}
for $x \geq x_0 := f^{-1}(z_0)$, and similarly $g_n(x) := \int_{x_0}^x |\tilde{f}_n'(s)|ds$, we see by a simple dominated convergence argument that the monotone functions $g_n$ and $g$ satisfy the point-wise convergence $g_n(x) \rightarrow g(x)$, whence $g_n^{-1}(1) \rightarrow g^{-1}(1)$, and therefore
\begin{align*}
    \tilde{z}_n(1) = \tilde{f}_n\circ g_n^{-1}(1) \rightarrow \tilde{f}_y \circ g^{-1}(1) = \tilde{z}_y(1).
\end{align*}
It follows that 
\begin{align*}
    M_n(z) := e^{-i\arg(\tilde{z}_n(1))} \frac{z-z_0}{z-f(y_n)} \quad \rightarrow \quad e^{-i\arg(\tilde{z}_y(1))} \frac{z-z_0}{z-f(y)} =: M_y(z)
\end{align*}
uniformly away from $f(y)$. We normalize based off of these latter two maps, setting $\gamma_y:= M_y(\gamma)$, $\gamma_n:= M_n(\gamma)$, and write $z_y$ and $z_n$ for the corresponding positively-oriented arc-length parametrizations which satisfy $z_y(0)=z_n(0)=0$.  Following the above rotations, both $z_y(1)$ and $z_n(1)$ are positive reals.  

To finish normalizing in the sense of Definition \ref{Def:NormalizedCurve}, we pre-compose $M_y \circ f$ by $T_y \in \PSL_2(\R)$ that maps the triple $(0,1,\infty)$ to $\big( f^{-1}(z_0), f^{-1}\circ M_y^{-1}(z_y(1)), y \big)$, so that $f_y := M_y \circ f \circ T_y$ satisfies $f_y(s) = z_y(s)$, $s \in \{0,1,\infty\}$. Similarly choosing $T_{h(y)} \in \PSL_2(\R)$ so that $g_y := M_y \circ g \circ T_{h(y)}$ also agrees with $z_y(s)$ for $s \in \{0,1,\infty\}$, we have that $h_y:=g_y^{-1} \circ f_y$ is the unique normalized welding for $\gamma_y$, fixing the triple $(0,1,\infty)$.  Furthermore,
\begin{align}\label{Eq:ConjugatedWeldingK}
    K_{h_y}(\infty)=K_{ T^{-1}_{h(y)} \circ h \circ T_y}(\infty)=K_h(y)
\end{align}
by Theorem~\ref{Thm:KMobiusInvariance}.

Similarly setting $f_n = M_n \circ f \circ T_{n}$ and $g_n = M_n \circ g \circ T_{h(y_n)}$ for the unique elements $T_{n}, T_{h(y_n)} \in \PSL_2(\R)$ such that $f_n(s) = z_n(s) =g_n(s)$, $s \in \{0,1,\infty\}$, we have that the corresponding normalized welding $h_n$ satisfies $K_{h_n}(\infty) = K_h(y_n)$, and therefore
\begin{align*}
    |K_h(y_n)- K_h(y)|= | K_{h_n}(\infty) - K_{h_y}(\infty)| &=\big| \|\log h_n'\|_{\half}^2- \|\log h_y'\|_{\half}^2 \big| \\
    &\leq \big( \|\log h_n'\|_{\half}+ \|\log h_y'\|_{\half}\big) \| \log h_n' - \log h_y'\|_{\half}\\
    &\lesssim \| \log h_n' - \log h_y'\|_{\half},
\end{align*}
since by Theorem~\ref{Comparable} and the conformal invariance of both the quasiconformal constant and the Loewner energy, 
\begin{align*}
    \|\log h_n'\|_{\half}+ \|\log h_y'\|_{\half} &\leq \max\{\, 1,  \,\|\log h_n'\|_{\half}^2  \,\} + \max\{\, 1,  \,\|\log h_y'\|_{\half}^2 \,\}\\
    &\leq \max\{\, 1,  \,W_{h_n}(\infty)  \,\} + \max\{\, 1,  \,W_{h_y}(\infty) \,\}\\
    &\leq \max\{\, 2,  \,(3+K_\gamma^2+K_\gamma^{-2}) I^{L}(\gamma) \,\}.
\end{align*}
Thus it suffices to show $\| \log h_n' - \log h_y' \|_{\half} \to 0$, which, by Proposition \ref{Thm: equiv convergence H}, is equivalent to showing
\begin{align}\label{Lim:KContinuousNTS}
    \int_{\mathbb{H}} \abs{\frac{f_{n}''}{f_{n}'} - \frac{f_y''}{f_y'}}^2 dA \rightarrow 0.
\end{align}
We can see the latter by Proposition \ref{Thm: DCL}.  We first claim that, point-wise on $\H$, 
\begin{align}\label{Lim:NTSAPtWiseConv}
    \calA_{f_n}(z) \rightarrow \calA_{f_y}(z).
\end{align}  
Indeed, recalling the composition law for the pre-Schwarzian (see Table \ref{Table:ABS}), we have
\begin{align*}
    |\calA_{f_n} - \calA_{f_y}| \leq & |\calA_{T_n}-\calA_{T_{y}}|+ |\calA_f \circ T_n\cdot T_n'- \calA_f \circ T_y \cdot T_y'| \\
    &+ |(\calA_{M_n} \circ f \cdot f')\circ T_n \cdot T_n' - (\calA_{M_y}\circ f \cdot f') \circ T_y \cdot T_y'| \\
    &=:I_1+I_2+I_3.
\end{align*}
For $I_1$, we compute $T_n(z) = \frac{yz-c_nf^{-1}(z_0)}{z-c_n}$ and $T_y(z) = \frac{yz-c_yf^{-1}(z_0)}{z-c_y}$ for 
\begin{align*}
    c_n = \frac{f^{-1} \circ M_n^{-1}(z_n(1))-y}{f^{-1} \circ M_n^{-1}(z_n(1))-f^{-1}(z_0)} \rightarrow \frac{f^{-1} \circ M_y^{-1}(z_y(1))-y}{f^{-1} \circ M_y^{-1}(z_y(1))-f^{-1}(z_0)} = c_y,
\end{align*}
and thus see
\begin{align*}
    |\calA_{T_n}(z)-\calA_{T_{y}}(z)| = \bigg| \frac{-2}{z-c_n} - \frac{-2}{z-c_y} \bigg| \rightarrow 0.
\end{align*}

For $I_2$, we note that
\begin{align*}
    I_2 &\leq |\calA_f \circ T_n\cdot T_n' -\calA_f \circ T_y \cdot T_{n}'| + |\calA_f \circ T_y \cdot T_{n}'- \calA_f \circ T_y \cdot T_y'| \\
    & \leq |\calA_f \circ T_n -\calA_f \circ T_y||T_n'| + |\calA_f \circ T_y| | T_{n}'- T_y'| \to 0
\end{align*}
since $T_n \rightarrow T_y$ and $T_n' \rightarrow T_y'$. 

For the last term, we similarly see 
\begin{align*}
    I_3 \leq&  \big|(\calA_{M_n} \circ f \cdot f')\circ T_n  - (\calA_{M_y} \circ f \cdot f')\circ T_y \big| |T_{n}'|+
    |(\calA_{M_y} \circ f \cdot f')\circ T_y ||T_n' - T_y'|\\
    \leq& \big|\calA_{M_n}\circ f \cdot f'- \calA_{M_y}\circ f \cdot f'\big| \circ T_n \cdot |T_n'|\\
    & \qquad + \big| (\calA_{M_y}\circ f \cdot f') \circ T_n - (\calA_{M_y}\circ f \cdot f') \circ T_y \big| \cdot |T_n'|\\
    & \qquad +|(\calA_{M_y} \circ f \cdot f')\circ T_y ||T_n' - T_y'|\\
    \to& \,0.
\end{align*}
We conclude \eqref{Lim:NTSAPtWiseConv} holds, as claimed.  By symmetry, we likewise have $\calA_{g_n}(z) \rightarrow \calA_{g_y}(z)$.

For the dominating functions $G_n$ of Proposition \ref{Thm: DCL}, we observe
\begin{align*}
    |\calA_{f_n} - \calA_{f_y}|^2 &= |\calA_{f_n} - \calA_{f_y}|^2\chi_{\H}\\
    &\leq |\calA_{f_n} - \calA_{f_y}|^2\chi_{\H} + \abs{\calA_{g_n} - \calA_{g_y}}^2\chi_{\H^*}\\
    &\leq 2(|\calA_{f_n}|^2 + |\calA_{f_y}|^2)\chi_{\H} + 2(|\calA_{g_n}|^2 + |\calA_{g_y}|^2)\chi_{\H^*} =:G_n. 
\end{align*}
From the above point-wise convergence, it follows that
\begin{align*}
    G_n(z)  \rightarrow 4|\calA_{f_y}(z)|^2\chi_{\H}(z) + 4|\calA_{g_y}(z)|^2\chi_{\H^*}(z) =: G(z)
\end{align*}
a.e. $z \in \C$. Recalling \eqref{Eq:LEH} and the conformal invariance of Loewner energy, we have
\begin{align*}
    \int_{\C} G_n \,dA = 2\pi I^L(\gamma_n) + 2\pi I^L(\gamma_y) = 4\pi I^L(\gamma) = \int_{\C}G \,dA.
\end{align*}
We conclude by Proposition \ref{Thm: DCL} that \eqref{Lim:KContinuousNTS} holds.  


If $f(y)=\infty$, choose \Mob $S$ sending $\infty$ to a finite point and replace $f$ and $g$ with $S\circ f$ and $S \circ g$, which still has the same welding $h$.
\end{proof}

\subsection{A corollary on Weil--Petersson distance}\label{Sec: cor1}

\begin{proof}[Proof of Corollary \ref{Thm: boundedcor}]
        Since $I^{L}(\g)\leq \frac{1}{3} W_h(\infty)$ by Theorem \ref{Comparable}, it suffices to show $W_h(\infty)$ is controlled by a constant only depending on $d_{\WP}(h, \text{id})$.  Our approach to this will be to imitate the argument of \cite[proof of Lemma 6.8]{ParametrizationpWP}.  By \cite[Lem. 6.8]{ParametrizationpWP}, there exist constants $C, \tau >0$ such that the operator norm $\|\calP_h\|$ of the pullback $\calP_h$ satisfies $\|\calP_h\| \leq C$ whenever $\| \log h' \|_{\half} \leq \tau$.  Furthermore, by the analytic (and in particular continuous) nature of the correspondence $\WP(\R) \rightarrow H^\half(\R)$ given by $h \mapsto \log h'$, $\| \log h' \|_{\half} \leq \tau$ whenever $d_{\WP}(h,\text{id}) \leq \epsilon$ \cite[Thm. 2.3]{WPII}.  Write $d$ for $d_{\WP}$ to simplify notation.

    Join $h$ to id in $\WP(\R)$ with a curve of length less than $d(h,\text{id}) + 1$, say, and subdivide the curve into the smallest number of points $id=h_0, h_1, \ldots, h_n=h$ such that $d(h_j, h_{j-1}) \leq \epsilon$ for each $j$.  Note $n$ only depends on $d(h,\text{id})$.  By invariance of $d$ under right translation \cite[Ch.I \S4.2]{TTbook},
    \begin{align*}
        d(h_j \circ h^{-1}_{j-1}, \text{id}) = d(h_j, h_{j-1}) \leq \epsilon 
    \end{align*}
    for each $j$, and writing $H_j := h_j \circ h^{-1}_{j-1}$, we thus have both $\|\log H_j'\|_{\half}\leq \tau$ and $\|\calP_{H_j}\|\leq C$ for all $j$.  We thereby deduce
    \begin{align*}
        \|\log h'\|_{\half} &=\|\log (H_n \circ \cdots \circ H_1 )'\|_{\half}\\
        &\leq 
         \| \calP_{H_1} \circ \cdots \circ  \calP_{H_{n-1}} \log H_n' \|_{\half} +\cdots + \|\calP_{H_1} \log H_2'\|_{\half}+\|\log H_1'\|_{\half}\\
         &\leq  C^{n-1}\tau +\cdots + C\tau + \tau =\tau\frac{1-C^n}{1-C}
    \end{align*}
    if $C \neq 1$, while the sum is simply $n\tau$ otherwise.  Noting $d(h^{-1},\text{id}) = d(h,\text{id})$, we apply the same logic to obtain the identical bound for $\|\log (h^{-1})'\|_{\half}$, and thus  conclude 
    \begin{equation*}
        W_h(\infty)=\|\log h'\|_{\half}^2 +\|\log (h^{-1})'\|_{\half}^2\leq C_1(d(h,\text{id})).\qedhere
    \end{equation*}
\end{proof}

\section{Proof of main inequalities for Loewner energy}\label{sec: entropypart}

\subsection{Sub-additivity of the Loewner energy}\label{Sec:ProofOfSubAdditivity}

In this section we prove Theorem \ref{Thm:GeneralCompBound} by means of showing a slightly more general result that keeps track of both quasiconformal constants.  Theorem \ref{Thm:GeneralCompBound} is clearly an immediate consequence.

\begin{thm}\label{Thm:GeneralCompBoundBothKs}
    For finite-energy Jordan curves $\gamma$ and $\eta$ that are $K_\gamma$- and $K_\eta$-quasicircles, respectively,
    \begin{align*}
        I^L(\gamma \circ \eta) \leq (1+K_\eta^2)(4+K_\gamma^2)I^L(\gamma) + (1+K_\gamma^2)(4+K_\eta^2)I^L(\eta).
    \end{align*}
\end{thm}
\begin{proof}
Let $h_\gamma, h_\eta \in \ophom{\wR}$ be weldings of $\gamma$ and $\eta$, respectively.  By Theorem \ref{Comparable}, Lemma \ref{Lemma:WGeneralComp}, and then Theorem \ref{Comparable} again,
\begin{align*}
    I^L(\gamma \circ \eta) &\leq \frac{2}{3}W_{h_\gamma\circ h_\eta}(y)\\
    &\leq \frac{4}{3} (K_\eta^2+K_\eta^{-2})W_{h_\gamma}(h_\eta(y)) + \frac{4}{3}(K_\gamma^2+K_\gamma^{-2})W_{h_\eta}(y) \\
    &\leq \frac{2}{3} (K_\eta^2+K_\eta^{-2})(3+K_\gamma^2+K_\gamma^{-2})I^L(\gamma) + \frac{2}{3}(K_\gamma^2+K_\gamma^{-2})(3+K_\eta^2+K_\eta^{-2})I^L(\eta).\qedhere
\end{align*}
\end{proof}

\subsection{Estimate on growth under self-composition}\label{Sec:ProofOfEntropy}
We next turn to Theorem \ref{Thm:LEGrowth}.  We could give a version of this bound directly from Theorem \ref{Thm:GeneralCompBound}, but we obtain a slightly sharper constant by going back to Theorem \ref{Comparable} and Lemma \ref{Lemma:WGeneralComp}.  

Either approach, however, requires one additional lemma.  It is well known that the constant $K$ for quasiconformal maps is multiplicative in composition: if $f_j: \C \rightarrow \C$ are $K_j$-quasiconformal, then $f_1 \circ f_2$ is $K_1K_2$-quasiconformal.  We need a version of this for conformal welding.  While this appears not to exist in the literature, it is a simple consequence Lemma \ref{Lemma:K^2Optimal} (and thus  Smirnov's result \cite[Thm. 4]{Smirnov}, from which Lemma \ref{Lemma:K^2Optimal} follows).

    \begin{figure}
       \centering

    \scalebox{0.9}{
    \begin{tikzpicture}
        \tikzset{
            diff/.style={text=blue!80},
            dom/.style={fill=gray!15}
        }
        
        \draw[dom] (0,0) circle[radius=1];        
        \node at (0,0) {$\D$};
        \draw[-{Latex}] (-0.05,1.2) to  [in=45, out=135,looseness=20] node[above]{$w_{\mu_2}$}  (0.05,1.15);
        \node[diff] at (-1.05,0.9) {$\mu_2$};
        
        \draw[-{Latex}] (-1.75,0) -- ++(-1,0);
        \node[above] at (-2.25,0) {$f_2$};
        \draw[-{Latex}] (1.75,0) -- ++(1,0);
        \node[above] at (2.25,0) {$f_1$};
        \node[below] at (2.25,0) {$g_1$};
        \draw[-{Latex}] (5,-1.75) -- ++(0,-1);
        \node[right] at (5,-2.25) {$q$};
        \draw[-{Latex}] (1.75,-1.75) -- ++(1,-1);
        \node[anchor = north east] at (2.25,-2.25) {$q \circ g_1 \circ w_{\mu_2}$};
        
        \draw[dom] (-5.25,-0.5) to[closed, curve through = {(-4.75,-1.5) (-3.25,-0.5) (-4.25,0)  }] (-3.75,1);
        \node[right] at (-3.8,0.7) {$\gamma_2$};

        \draw[dom] (4.5,0.5) to[closed, curve through = {(5,1) (5.5,0.5) (6,-0.5) (4,-1) (3.5,-0.5)}] (4,0.5);
        \node at (5.75,0.75) {$\gamma_1$};
        \node[diff,above] at (3.8,0.65) {\small $(g_1^{-1})^* \mu_2$};
        \node at (4.75,-0.25) {$\Omega_1$};
        \node at (5.2,1.5) {$\Omega_1^*$};

        \draw[dom] (3,-3.75) to[closed, curve through = {(3.5,-3.75) (4.5,-4.25) (4,-3.25) (5,-3.25) (5.5,-4.75) (5,-5.25) (4.5,-5.75) (4,-5.25)}] (3.5,-4.25);
        \node[anchor = south west] at (5.4,-3.75) {$\gamma_3$};
        \node at (5,-4.35) {$\Omega_3$};
        \node at (5.6,-5.45) {$\Omega_3^*$};
    \end{tikzpicture}
    }
        
        \caption{\small The maps and differentials behind the construction of $\gamma_3$ in the alternative proof of Lemma \ref{prop: compositionofK}.  Beltrami differentials are in blue.  (The shape of $\gamma_3$ is a sketch only for illustrative purposes; for accurate pictures of curve composition, see \cite[Figure 10]{SharonMumford}.)}
        \label{Fig: Composition curve}
    \end{figure}
    
\begin{lemma}
\label{prop: compositionofK}
    If $h_j$ are conformal weldings of $K_j$-quasicircles $\gamma_j \subset \hat{\C}$, $j=1,2$, then $h_1\circ h_2$ is the conformal welding of a $K_1K_2$-quasicircle.
\end{lemma}

\begin{proof}
    Let $f_j$ and $g_j$ be conformal maps from $\H$ and $\H^*$, respectively, to the two components of $\hat{\C} \bs \gamma_j$, such that $h_j = g_j^{-1} \circ f_j$.  Both $f_j$ and $g_j$ have $K_j^2$-quasiconformal extensions to all of $\C$ \cite[Ch.I Lem.~6.2]{Lehtobook}.  Denoting these extensions by the same name, we thus see that $h_j$ extends globally to a $K_j^2$-quasiconformal homeomorphism of $\C$.  In particular, the composition of these extensions $H := h_1 \circ h_2$ restricts to a homeomorphism of $\R$, and is globally a  $K_1^2 K_2^2$-quasiconformal map.  By Lemma \ref{Lemma:K^2Optimal}, it thereby welds a $K_1K_2$-quasicircle.
\end{proof}

There is also an alternative, explicit argument for Lemma \ref{prop: compositionofK}.  We are grateful to Yilin Wang for teaching us how to think of composition of weldings in terms of \Teich space models, from which one obtains the following argument.  The benefit of this approach is that we explicitly construct the quasicircle $\gamma_3$ welded by $h_1\circ h_2$. 

\begin{proof}[Alternative proof of Lemma \ref{prop: compositionofK}] 
For $j=1,2$, let $\Omega_j$ and $\Omega_j^*$ be the bounded and unbounded components of $\C \bs \gamma_j$, respectively, with $f_j:\D \rightarrow \Omega_j$ and $g_j: \D^* \rightarrow \Omega_j^*$ associated conformal maps.  Recall that $f_2$ has a $K_2^2$-quasiconformal extension to all of $\C$ \cite[Ch.I Lem. 6.2]{Lehtobook}.  We call this extension $f_2$ still, and denote its complex dilatation on $\C$ by $\tilde{\mu}_2$.  Since $f_2$ is conformal on $\D$, $\tilde{\mu}_2|_\D \equiv 0$, and we proceed to replace these null values on $\D$ with the reflection across $\partial \D$ of $\tilde{\mu}_2$ outside of $\D$.  That is, for $z \in \D$ we redefine 
\begin{align}\label{Eq:BeltramiReflection}
    \tilde{\mu}_2(z) := \overline{\tilde{\mu}_2\bigg(\frac{1}{\bar{z}}\bigg)}\, \frac{z^2}{\bar{z}^2},
\end{align}
and use the same name $\tilde{\mu}_2$ for this new Beltrami differential (which still has its original values outside of $\D$).  Taking a solution $\tilde{w}$ of the Beltrami equation for $\tilde{\mu}_2$ on all of $\C$, we have that $\tilde{w}|_{\partial \D}$ is a conformal welding for $\gamma_2$, and thus, by the uniqueness of welding for quasicircles, $S \circ \tilde{w} \circ T|_{\partial \D} = h_2$ for some \Mob transformations $S,T \in \text{Aut}(\D)$ \cite[Ch.III Lem. 1.1]{Lehtobook}.  Let $\mu_2$ be the Beltrami coefficient of $S \circ \tilde{w} \circ T$ (which is a rotation of $\tilde{\mu}_2$), and set $w_{\mu_2}:=S \circ \tilde{w} \circ T$.  For this and what follows, the reader may be aided by Figure~\ref{Fig: Composition curve}.

We can now explicitly construct the curve $\gamma_3$ whose welding is $h_1 \circ h_2$.  Let $q$ be a solution of the Beltrami equation 
    \begin{align*}
        \frac{\partial_{\bar{z}} q}{\partial_z q}=
         \begin{cases}
             0 & \text{ on } \Omega_1,\\
            (g_1^{-1})^{*} \mu_2 & \text{ on } \Omega_1^*
         \end{cases}
    \end{align*}
    normalized so that $q(\gamma_1) =: \gamma_3 \subset \C$, where $(g_1^{-1})^{*} \mu_2 = \mu_2 \circ g_1^{-1} \cdot \frac{ \overline{\partial_z g_1^{-1}}}{\partial_z g_1^{-1}}$ is the pullback of differential $\mu_2$ under $g_1^{-1}$.  That is, $q$'s dilatation on $\Omega_1^*$ is that of $w_{\mu_2} \circ g_1^{-1}$ on $\Omega_1^*$.  It follows that $g_3:= q \circ g_1 \circ w_{\mu_2}^{-1}$ maps $\D^*$ conformally to the unbounded domain $\Omega_3^*$ of $\C \bs \gamma_3$.\footnote{Conformality also follows from computation, via the Beltrami coefficient composition rule  \[ \mu_{f \circ g^{-1}}(w) = \frac{\mu_{f}(z)-\mu_g(z)}{1-\mu_f(z)\overline{\mu_g(z)}}\left(\frac{\partial_{z}g}{|\partial_{z}g|}\right)^2,  \, w=g(z).
    \].}  We also have that $f_3:= q \circ f_1: \D \to \Omega_3$ is conformal on $\D$, being a composition of conformal maps there.  A welding of $\gamma_3$ is thus
    \begin{align*}
        g_3^{-1} \circ f_3 &= \left(q \circ g_1 \circ w_{\mu_2}^{-1}\right)^{-1} \circ q \circ f_1 = w_{\mu_2} \circ g_1^{-1} \circ f_1 = h_2 \circ h_1,
    \end{align*}
    as claimed.

    Lastly, to see that $\gamma_3$ is a $K_1K_2$-quasicircle, by \cite[Thm. 4(ii)]{Smirnov} it suffices to show that $g_3$ has a $K_1^2K_2^2$-quasiconformal extension to $\D$.  This, however, is immediate from the definition of $g_3$.  Indeed, since $f_2$ is $K_2^2$-quasiconformal on $\D^*$, so is $\tilde{w}$ by \eqref{Eq:BeltramiReflection}, and thus $w_{\mu_2} = S\circ \tilde{w} \circ T$ is likewise.  Furthermore, $g_1$ has a $K_1^2$-quasiconformal extension to $\D$, and $q$ is conformal on $\Omega_1$.  
\end{proof}

\begin{proof}[Proof of Theorem \ref{Thm:LEGrowth}]
    Write $h^n$ for the $n$-fold composition of $h$ with itself. By \eqref{Ineq:Main} it suffices to show 
    \begin{align}\label{Ineq:EntropyIneqProxy}
        \limsup_{n \to \infty} \frac{\log W_{h^n}(y)}{n} \leq \log(K^2+K^{-2})
    \end{align}
    for some $y$ (and some normalization of the welding).  We M\"obius re-normalize $h$ to fix $\infty$, call the new welding by the same name $h$, and choose $y=\infty$.  Lemma~\ref{Lemma:WGeneralComp} then yields
    \begin{align*}
        W_{h^2}(\infty) \leq 4(K^2+K^{-2}) W_h(\infty).
    \end{align*}
    Since $h^2$ welds a $K^2$-quasicircle by Lemma~\ref{prop: compositionofK}, applying \eqref{Ineq:WGeneralCompDiffK} and then the above inequality yields
    \begin{align*}
        W_{h^3}(\infty) = W_{h^2 \circ h}(\infty) &\leq 2(K^2+K^{-2})W_{h^2}(\infty)+2(K^4+K^{-4})W_h(\infty)\\
        &\leq \big(8(K^2+K^{-2})^2+2(K^4+K^{-4})\big)W_h(\infty)\\
        &\leq 10 (K^2+K^{-2})^2 W_h(\infty).
    \end{align*}
    Continuing by induction, one readily sees
    \[W_{h^n}(\infty) \leq 2(n+2) (K^2+K^{-2})^{n-1}W_h(\infty),\] 
    from which \eqref{Ineq:EntropyIneqProxy} directly follows.
\end{proof}

\appendix
\section{Appendix: Elementary properties of the normalized pre-Schwarzian and associated differential operators}\label{Appendix:NPS}

\subsection{The normalized pre-Schwarzian derivative}
In this section we prove the properties of the $\calB_f$ operator summarized in Table \ref{Table:NPSProps}.  We recall that for $f: \Omega \rightarrow \wC$ conformal on a domain $\Omega \subset \wC$ such that $f$ extends continuously to $\overline{\Omega}$,
\begin{align}\label{Eq:NPSDefAppendix}
    \calB_f(z,w) := \frac{f'(z)}{f(z)-f(w)} - \frac{1}{z-w} - \frac{1}{2} \frac{f''(z)}{f'(z)},
\end{align}
assuming each of $z,w,f(z)$, and $f(w)$ are finite, and that $z \neq w$.   The corresponding term vanishes when either $w$ or $f(w)$ is infinite.  By saying $f:\Omega \rightarrow \hatC$ is conformal, recall we mean that $f$ is injective and meromorphic on $\Omega$.  See \S\ref{Sec:NPS} for further discussion and other cases.

We first consider the value of $\calB_f(z,w)$ on the diagonal $z=w \in \Omega$.
\begin{lemma}\label{Lemma:NPSDiagonal}
Let $\Omega \subset \widehat{\C}$ be a domain and $f: \Omega \rightarrow \widehat{\C}$ conformal.
\begin{enumerate}[$(i)$]
    \item For $w \in \Omega \bs \{\infty\}$, \begin{align}\label{Eq:BDiagonalExpansion}
    \calB_f(z,w) = -\frac{1}{6}S_f(w)(z-w) + O(z-w)^2, \qquad z \rightarrow w.
\end{align}
\item If $\infty \in \Omega$,
\begin{align}\label{Eq:BDiagonalExpansion2}
    \calB_f(z,\infty) = \frac{1}{6}S_f(z)z + O(z^{-4}) = \frac{1}{6} \tilde{S}_f(\infty)z^{-3} + O(z^{-4}), \qquad z \rightarrow \infty,
\end{align}
where $\tilde{S}_f(\infty) := \lim\limits_{z \rightarrow \infty}z^4S_f(z) = (S_{f\circ \frac{1}{z}})(0)$.
\end{enumerate}
\end{lemma}
\begin{proof}
    $(i)$ When $f(w) \in \C$, \eqref{Eq:BDiagonalExpansion} is a routine series computation.  In fact, all cases here are; we proceed to sketch some details for completeness.  By considering the function $F(z):= f(z+w)-f(w)$ which satisfies $\calB_F(z-w,0) = \calB_f(z,w)$, we may assume $w=0=f(w)$.  Writing $f(z) = a_1z + a_2z^2 + a_3z^3 + \cdots$ for $z$ nearby zero, we find
    \begin{align*}
        \calB_f(z,0) &= \frac{f'(z)}{f(z)} - \frac{1}{z} - \frac{1}{2} \frac{f''(z)}{f'(z)}\\
        &= \frac{1}{z}\bigg( 1 + \frac{a_2}{a_1}z + \Big( 2 \frac{a_3}{a_1}-\frac{a_2^2}{a_1^2} \Big)z^2 + O(z^3) \bigg) - \frac{1}{z} - \frac{a_2}{a_1} + \bigg(2\frac{a_2^2}{a_1^2} - 3 \frac{a_3}{a_1} \bigg)z + O(z^2)\\
        &=-\bigg(\frac{a_3}{a_1}-\frac{a_2^2}{a_1^2} \bigg)z + O(z^2),
    \end{align*}
    as claimed.
    
    If $f(w) = \infty$, we may again suppose that $w=0$, and we have 
    \begin{align}\label{Eq:fLaurent}
        f(z) = \frac{a_{-1}}{z} + a_0 + a_1z + O(z^2), \qquad z \rightarrow 0,
    \end{align}
    leading to the computation
    \begin{align*}
        \calB_f(z,0) := - \frac{1}{z} - \frac{1}{2} \frac{f''(z)}{f'(z)}
        =- \frac{1}{z} + \frac{1}{z}\bigg(1 + \frac{a_1}{a_{-1}}z^2 + O(z^3) \bigg) = \frac{a_1}{a_{-1}}z + O(z^2).
    \end{align*}
    Arithmetic on the Laurent series \eqref{Eq:fLaurent} similarly yields $S_f(0) = -6a_1/a_{-1}$.

    \noindent $(ii)$ When $w=\infty$ but $f(w) \neq \infty$, 
    \begin{align*}
        f(z) = b_0 + \frac{b_1}{z} + \frac{b_2}{z^2} + \frac{b_3}{z^3} + O\Big(\frac{1}{z^4}\Big), \qquad z \rightarrow \infty,
    \end{align*}
    and we have
    \begin{align*}
        \calB_f(z,\infty) := \frac{f'(z)}{f(z)-b_0} - \frac{1}{2} \frac{f''(z)}{f'(z)} = \frac{b_3/b_1 - b_2^2/b_1^2}{z^3} + O\Big(\frac{1}{z^4}\Big), \qquad z \rightarrow \infty.
    \end{align*}
    On the other hand, computations yield
    \begin{align*}
        S_f(z) = \frac{6b_3/b_1 - 6b_2^2/b_1^2}{z^4} + O\Big(\frac{1}{z^5}\Big), \qquad z \rightarrow \infty,
    \end{align*}
    yielding the claim.

    If $f(\infty)=\infty$, then nearby $\infty$,
    \begin{align*}
        f(z) = b_{-1}z + b_0 + \frac{b_1}{z} + O\Big(\frac{1}{z^2}\Big),
    \end{align*}
    and we find
    \begin{align*}
        \calB_f(z,\infty) := - \frac{1}{2} \frac{f''(z)}{f'(z)} = \frac{-b_1/b_{-1}}{z^3} + O\Big(\frac{1}{z^4}\Big), \qquad z \rightarrow \infty.
    \end{align*}
    On the other hand, 
    \begin{align*}
        S_f(z) = \frac{-6b_1/b_{-1}}{z^4} + O\Big(\frac{1}{z^5}\Big), \qquad z \rightarrow \infty,
    \end{align*}
    which again verifies the claim.
\end{proof}

The normalized pre-Schwarzian has the following general composition rule, which parallels the standard pre-Schwarzian.  
\begin{lemma}\label{Lemma:NPSGeneralComp}
     Let $\Omega_1,\Omega_2 \subset \widehat{\C}$ be domains with locally-connected boundaries with $g: \Omega_1 \rightarrow \Omega_2$ and $f: \Omega_2 \rightarrow \widehat{\C}$ conformal. Then for all $(z,w) \in \Omega_1 \times \overline{\Omega}_1$, 
     \begin{align}\label{Eq:NPSGeneralComp}
         \calB_{f \circ g}(z,w) =\calB_{f}\big(g(z),g(w)\big)g'(z) + \calB_g(z,w).
     \end{align} 
\end{lemma}
\noindent Note that the expressions for $\calB_{f\circ g}, \calB_f,$ and $\calB_g$ on either side of \eqref{Eq:NPSGeneralComp} may have different numbers of terms, depending on whether or not each of $w, g(w),$ and $f\circ g(w)$ is infinite.  The lemma says that \eqref{Eq:NPSGeneralComp} holds in all possible combinations.  We understand \eqref{Eq:NPSGeneralComp} in a limiting sense if any of the expressions involving $z$ is infinite.

Compare the general composition rule for the point-wise welding functional $L_h$ in Lemma \ref{Lemma:LGeneralComp}.

\begin{proof}
    We first suppose that all expressions involving $z$ are finite.  Here there are technically eight cases to check, corresponding to whether or not each of $w,g(w),$ and $f\circ g(w)$ is in $\C$ or infinite.  We do two cases to give the reader a sense of the elementary computation; the rest are very similar.
    \begin{itemize}
        \item All of $w,g(w),f(g(w)) \in \C$.  We note
        \begin{align*}
            2\calB_{f \circ g}(z,w) &= \partial_z \log \bigg( \frac{\big(f(g(z))-f(g(w))\big)^2}{f'(g(z))g'(z)(z-w)^2} \bigg)\\
            &= \partial_z \log \bigg( \frac{\big(f(g(z))-f(g(w))\big)^2}{f'(g(z))(g(z)-g(w))^2} \bigg) + \partial_z \log \bigg( \frac{(g(z)-g(w))^2}{g'(z)(z-w)^2} \bigg)\\
            &= 2\calB_{f}\big(g(z),g(w)\big)g'(z) + 2\calB_g(z,w).
        \end{align*}
        \item $w=\infty$ and $g(w),f(g(w)) \in \C$. We use \eqref{Eq:NPSwInfty} and proceed as above, noting
        \begin{align*}
            2\calB_{f \circ g}(z,w) &= \partial_z \log \bigg( \frac{\big(f(g(z))-f(g(w))\big)^2}{f'(g(z))g'(z)} \bigg)\\
            &= \partial_z \log \bigg( \frac{\big(f(g(z))-f(g(w))\big)^2}{f'(g(z))(g(z)-g(w))^2} \bigg) + \partial_z \log \bigg( \frac{(g(z)-g(w))^2}{g'(z)} \bigg)\\
            &= 2\calB_{f}\big(g(z),g(w)\big)g'(z) + 2\calB_g(z,w).
        \end{align*}
    \end{itemize}
    The logic is identical for the outstanding cases: one multiplies the numerator and denominator by the appropriate factor and splits the logarithm apart.

    If any of the expressions involving $z$ is infinite, take limits in \eqref{Eq:NPSGeneralComp} from nearby finite points.
\end{proof}

\begin{lemma}\label{Lemma:BMobiusZero}
    Let $\Omega \subset \widehat{\C}$ be a domain with locally-connected boundary and $f: \Omega \rightarrow \widehat{\C}$ conformal.  There exists $w \in \overline{\Omega}$ such that $z \mapsto \calB_f(z,w)$ is the zero function in some neighborhood of $\Omega$ if and only if $f$ is a \Mob transformation.  In this case, $(z,w) \mapsto \calB_f(z,w)$ is the zero function on $\widehat{\C}^2$.
\end{lemma}
\begin{proof}
    In the case that $f(z) = \frac{az+b}{cz+d}$ is \Mob, it is routine to verify that $\calB_f$ vanishes on $\widehat{\C}^2$; we include the argument for completeness.  As in the previous lemma, there are several cases.  
    \begin{itemize}
        \item $w\in \C$.  If $f(w) \neq \infty$, we find
        \begin{align*}
            \calB_f(z,w) = \frac{cw+d}{(z-w)(cz+d)} - \frac{1}{z-w} + \frac{c}{cz+d} = 0.
        \end{align*}        
        If $f(w)=\infty$, then $w=-d/c$, and 
        \begin{align*}
            \calB_f(z,w) = 0 -\frac{1}{z+d/c} + \frac{c}{cz+d}=0.
        \end{align*}
        In either case, taking limits then covers the cases of when $z=\infty$ or $z$ a pole of $f$.
        \item $w = \infty$. If $f(\infty) = a/c \in \C$, we find
        \begin{align*}
            \calB_f(z,w) = \frac{f'(z)}{f(z)-a/c} - \frac{1}{2} \frac{f''(z)}{f'(z)} =0.
        \end{align*}
        If $f$ fixes $\infty$, $f$ is affine and 
        \begin{align*}
            \calB_f(z,w) =  - \frac{1}{2} \frac{f''(z)}{f'(z)} =0.
        \end{align*}
        Again, taking limits in either instance covers the $z=\infty$ and $f(z)=\infty$ cases.
    \end{itemize}

    Conversely, suppose there exists $w \in \overline{\Omega}$ that $\calB_f(\cdot, w)$ is the zero function in some $B_\epsilon(z_0) \subset \Omega$.  We first suppose that both $w$ and $f(w)$ are finite.  The assumption yields $\log \frac{f'(z)(z-w)^2}{(f(z)-f(w))^2}$ is some constant $C_1$ in $B_\epsilon(z_0)$, and therefore
    \begin{align*}
        -\frac{d}{dz} \frac{1}{f(z)-f(w)} = \frac{f'(z)}{(f(z)-f(w))^2}  = \frac{C_1}{(z-w)^2} = - \frac{d}{dz} \frac{C_1}{z-w}
    \end{align*}
    for $z \in B_\epsilon(z_0)$.  Integrating and re-arranging yields
    \begin{align*}
        f(z) = f(w) + \frac{z-w}{C_2z-C_2w + C_1},
    \end{align*}
    a \Mob transformation.  By the identity theorem we conclude the holomorphic function $z \mapsto \calB_f(z,w)$ is everywhere M\"{o}bius, and thus $\calB_f$ is well-defined and vanishing on all of $\widehat{\C}^2$ by the first part of the proof.

    If $w \in \C$ and $f(w) = \infty$, we apply Lemma \ref{Lemma:NPSMobiusInvariance1} and the above argument to $\calB_f(\cdot, w) = \calB_{1/f}(\cdot, w)$, concluding $1/f$, and thus $f$ itself, is M\"{o}bius.

    Next suppose $w=\infty$.  While $f(\infty)$ may be finite or infinite, by the previous paragraph, it suffices to suppose $f(\infty) \in \C$.  Writing $g(z) := f(1/z)$, by Lemma \ref{Lemma:NPSMobiusInvariance1} we have that $z \mapsto -z^2\calB_{g}(z,0)$, and thus $z \mapsto \calB_{g}(z,0)$ itself, vanishes on some ball.  Arguing as above we conclude $f\circ 1/z$ is M\"{o}bius.    
\end{proof}
\noindent Note that integrating as in the preceding proof shows that the property \eqref{Eq:MobDifferenceQuotient} characterizes $\PSL_2(\C)$.

The most important property in this section is the following, which shows how $\calB_f$ interacts with pre- and post-composition by elements of $\PSL_2(\C)$.
\begin{lemma}\label{Lemma:NPSMobiusInvariance1}
     Let $\Omega \subset \widehat{\C}$ be a domain with locally-connected boundary and $f: \Omega \rightarrow \widehat{\C}$ conformal.  If $S$ and $T$ are \Mob transformations, 
     \begin{align}\label{Eq:NPSMobiusPrePost}
         \calB_{S\circ f\circ T}(z,w) =\calB_{f}\big(T(z),T(w)\big)T'(z)
     \end{align} 
     for all $(z,w) \in T^{-1}(\Omega) \times T^{-1}(\overline{\Omega})$.
\end{lemma}
\noindent We recall that the assumption on $\partial \Omega$ is only so that $\calB$ is well defined for all $w \in \partial \Omega$. We understand \eqref{Eq:NPSMobiusPrePost} in a limiting sense if any of the expressions involving $z$ is infinite.

Thus, like the Schwarzian, $\calB_f$ is invariant under post-composition by \Mob maps. While the Schwarzian is a $(2,0)$-differential with respect to \Mob precomposition, we see the normalized pre-Schwarzian behaves as a $(1,0)$-differential; neither property holds for the traditional pre-Schwarzian.

\begin{proof}
    First assume all quantities involving $z$ are finite.  By Lemmas \ref{Lemma:NPSGeneralComp} and \ref{Lemma:BMobiusZero},
    \begin{align*}
        \calB_{S\circ f\circ T}(z,w) &= \calB_S\big( f\circ T(z), f\circ T(w) \big)(f \circ T)'(z) + \calB_f\big( T(z), T(w) \big)T'(z) + \calB_T(z,w)\\
        &= \calB_f\big( T(z), T(w) \big)T'(z).
    \end{align*}
    If any of the expressions involving $z$ is infinite, take limits from a nearby finite point.
\end{proof}

\begin{cor}\label{Lemma:BIntegralInvariance}
    Let $\Omega \subset \widehat{\C}$ be a domain with locally-connected boundary and $f: \Omega \rightarrow \widehat{\C}$ conformal. Then for any \Mob transformation $T$ and $w \in \overline{\Omega}$,
    \begin{align*}
        \int_\Omega |\calB_f(z,w)|^2 \dd A(z) = \int_{T^{-1}(\Omega)}|\calB_{f \circ T}(z, T^{-1}(w))|^2 \dd A(z).
    \end{align*}
\end{cor}
\noindent Part of the content of the lemma is that this formula holds whether or not $w, f(w)$, and/or $T^{-1}(W)$ is infinite.  That is, the expressions for $\calB_f(\cdot,w)$ and $\calB_{f\circ T}(\cdot, T^{-1}(w))$ in the integrals may have a different number of terms, depending on which of \eqref{Def:NPS}, \eqref{Eq:NPSfInfty}, \eqref{Eq:NPSwInfty}, or \eqref{Eq:NPSfwInfty} applies.
\begin{proof}
    By Lemma \ref{Lemma:NPSMobiusInvariance1} and a change of variables,
    \begin{align*}
        \int_{T^{-1}(\Omega)}|\calB_{f \circ T}(z, T^{-1}(w))|^2 \dd A(z) &= \int_{T^{-1}(\Omega)}|\calB_{f}(T(z), w)|^2|T'(z)|^2 \dd A(z)\\
        &= \int_\Omega |\calB_f(z,w)|^2 \dd A(z). \qedhere
    \end{align*}
\end{proof}

In the course of discussing the Schwarzian derivative, Lehto notes that, ``[it] can be defined in any domain $A$ for every function $f$ meromorphic and locally injective in $A$, and it is a holomorphic function in $A$'' \cite[Page 52]{Lehtobook}.  Lemma \ref{Lemma:NPSMobiusInvariance1} implies the same holds for $\calB_f$.  That is, even though \emph{a priori} $\calB_f(\cdot,w) \in \mathcal{M}(\Omega)$, the collection the meromorphic functions on $\Omega$, it never assumes the value $\infty$, even when $f$ itself does, and even when $z = \infty \in \Omega$. Let $\mathcal{H}(\Omega)$ be the collection of holomorphic (i.e. analytic, $\C$-valued) maps on $\Omega$.  Recall that the Schwarzian vanishes to at least order four at $\infty$.

\begin{cor}\label{Lemma:BFinite}
    Let $\Omega \subset \widehat{\C}$ be a domain with locally-connected boundary and $f: \Omega \rightarrow \widehat{\C}$ conformal.  For any $w \in \overline{\Omega}$, $\calB_f(\cdot, w) \in \calH(\Omega)$.  In particular, when $\infty \in \Omega$, $\calB_f(\cdot,w)$ vanishes to at least order two at $z=\infty$.
\end{cor}

\noindent  The corollary says that, at a pole $z$ of $f$, the singularity from the $f'(z)/(f(z)-f(w))$ term of $\calB_f$ cancels the singularity from the pre-Schwarzian term. 
\begin{remark}
    If we broadened the definition of the normalized pre-Schwarzian to include merely locally-injective functions or general meromorphic maps $f$, $\calB_f$ could have poles. The lemma precludes this possibility for injective $f$.
\end{remark}
\begin{proof}
    Clearly $\calB_f(\cdot, w) \in \mathcal{M}(\Omega)$, so we just need to show $\calB_f(
    \cdot,w): \Omega \rightarrow \C$.  This is the case, by definition \eqref{Eq:NPSDefAppendix} (or appropriate variant when $w$ and/or $f(w)$ is infinite) when both $z$ and $f(z)$ are finite.  Furthermore, since $\calB_{\frac{1}{z}\circ f}(z,w) = \calB_{f}(z,w)$ by Lemma \ref{Lemma:NPSMobiusInvariance1}, we may always assume $f(z) \in \C$.  To handle the $z=\infty$ case, we consider $\tilde{f}:= f \circ \frac{1}{z}$ and apply Lemma \ref{Lemma:NPSMobiusInvariance1} to see 
    \begin{align*}
        \calB_f(\infty,w) := \lim_{z \rightarrow 0} \calB_f(1/z,w) =  \lim_{z \rightarrow 0} -z^2 \calB_{\tilde f}(z,1/w) =0
    \end{align*}
    since $\calB_{\tilde f}(0,1/w) \in \C$ by the previous sentences.  In particular, $\calB_f(\infty,w)$ is zero to at least order two.
\end{proof}

\subsection{Two associated operators}

We lastly prove \Mob covariance for the $\calB^*$ and $\calC$ operators introduced in \S\ref{Sec:B*C}.  Recall the setting: $\gamma \subset \widehat{\C}$ is any Jordan curve, $D \subsetneq \widehat{\C}$ any disk, and $f:D \rightarrow \Omega$ and $g: D^* \rightarrow \Omega^*$ any conformal maps to the two complementary components of $\widehat{\C}\bs \gamma = \Omega \cup \Omega^*$.  Let $u \in \overline{D}$ and any $v \in \overline{D^*}$ be arbitrary.  Writing $u^* := j(u)= j_D(u)$ for the Schwarz reflection of $u$ across $\partial D$, we recall that
\begin{align*}
   \calB_{f,g,D}^*(z,u,v) := \frac{f'(z)}{f(z)-g(v)} - \frac{1}{z-u^*} - \frac{1}{2}\frac{f''(z)}{f'(z)}
\end{align*}
and
\begin{align}
    \calC_{f,g,D}(z,u,v) :=& \frac{f'(z)}{f(z)-f(u)} - \frac{f'(z)}{f(z)-g(v)} - \bigg(\frac{1}{z-u} - \frac{1}{z-u^*}  \bigg) \notag\\
    =& \calB_f(z,u) - \calB^*_{f,g}(z,u,v)\label{Eq:CalCDefAppendix}
\end{align}
when all expressions appearing are elements of $\C$.  When $u,u^*,v,f(u)$ or $g(v)$ is infinite, the corresponding term vanishes.  See \S\ref{Sec:B*C} for details.

\begin{lemma}\label{Lemma:NPS*MobiusInvariance}
    Let $D, D^*, f,g,u,$ and $v$ be as above.  Then for any \Mob transformation $S$ and any \Mob transformations $T_u$ and $T_v$ which map the same disk $E$ onto $D$ (and thus also $E^*$ to $D^*$),
    \begin{align}
        \calB^*_{S\circ f \circ T_u, S \circ g \circ T_v, E} (z,u,v) &= \calB^*_{f,g,D}\big(T_u(z),T_u(u), T_v(v)\big)T_u'(z) \quad \text{and}\label{Eq:NPS*MobiusInvariance}\\
        \calC_{S\circ f \circ T_u, S \circ g \circ T_v, E} (z,u,v) &= \calC_{f,g,D}\big(T_u(z),T_u(u), T_v(v)\big)T_u'(z) \label{Eq:CMobiusInvariance}
    \end{align}
    for all $(z,u,v) \in E \times \overline{E} \times \overline{E^*}$.
\end{lemma}
We understand \eqref{Eq:NPS*MobiusInvariance} and \eqref{Eq:CMobiusInvariance} in a limiting sense if any of the expressions involving $z$ is infinite.
\begin{proof}
    We first prove \eqref{Eq:NPS*MobiusInvariance}, and then use it and Lemma \ref{Lemma:NPSMobiusInvariance1} to show \eqref{Eq:CMobiusInvariance}.
    
    Beginning with  \eqref{Eq:NPS*MobiusInvariance}, we start by showing
    \begin{align}\label{Eq:NPS*MobiusInvarianceA}
        \calB^*_{S\circ f, S \circ g} (z,u,v) = \calB^*_{f,g}(z,u,v).
    \end{align}
    This is clear when $S$ is affine, and so for full \Mob invariance it suffices to consider $S(z)=1/z$.  First assuming each element appearing in the expression is finite, we compute
    \begin{align*}
        \calB^*_{S\circ f, S \circ g} (z,u,v) &= \frac{-\frac{f'(z)}{f(z)^2}}{\frac{1}{f(z)} - \frac{1}{g(v)}} - \frac{1}{z-u^*} + \frac{f'(z)}{f(z)} - \frac{1}{2}\calA_f(z)\\
        &= \frac{f'(z)}{f(z)-g(v)}-\frac{1}{z-u^*} - \frac{1}{2}\calA_f(z) = \calB^*_{f,g}(z,u,v).
    \end{align*}
    If any of the expressions involving $u$ or $v$ is infinite, take limits using the above finite case.  Then further  limits in $z$, if necessary (when $z$ or $f(z)$ is $\infty$), yields \eqref{Eq:NPS*MobiusInvarianceA} in all instances. 

    We secondly show
    \begin{align}\label{Eq:NPS*MobiusInvarianceB}
        \calB^*_{f\circ T_u, g\circ T_v,E} (T_u^{-1}(z),T_u^{-1}(u),T_v^{-1}(v)) = \calB^*_{f,g,D}(z,u,v)T_u'(T_u^{-1}(z))
    \end{align}
    for any $(z,u,v) \in D \times \overline{D} \times \overline{D^*}$, which completes the proof of \eqref{Eq:NPS*MobiusInvariance}.  Again, we first suppose each element appearing on either side is finite, and compute
    \begin{align}
        \calB^*_{f\circ T_u, g\circ T_v} (&T_u^{-1}(z),T_u^{-1}(u),T_v^{-1}(v)) \notag\\
        &= \bigg(  \frac{f'(z)}{f(z)-g(v)} - \frac{(T_u^{-1})'(z)}{T_u^{-1}(z)-(T_u^{-1}(u))^*} -\frac{1}{2} \calA_f(z) - \frac{\calA_{T_u}(T_u^{-1}(z))}{2T_u'(T_u^{-1}(z))} \bigg)T_u'(T_u^{-1}(z)). \label{Eq:NPS*MobiusLastStep}
    \end{align}
    By the pre-Schwarzian chain rule (see the first entry of Table \ref{Table:ABS}),
    \begin{align*}
        - \frac{\calA_{T_u}(T_u^{-1}(z))}{2T_u'(T_u^{-1}(z))} = \frac{1}{2}\calA_{T_u^{-1}}(z),
    \end{align*}
    while furthermore 
    \begin{align*}
        (T_u^{-1}(u))^* = j_{E}(T_u^{-1}(u)) = T_u^{-1}(j_D(u)) = T_u^{-1}(u^*)
    \end{align*}
    since $j_{E} \circ T_u^{-1} \circ j_D = T_u^{-1}$ on $\partial D$ and thus everywhere. Hence \eqref{Eq:NPS*MobiusLastStep} says
    \begin{align*}
        \calB^*_{f\circ T_u, g\circ T_v} (&T_u^{-1}(z),T_u^{-1}(u),T_v^{-1}(v))\\
        &= \bigg(  \frac{f'(z)}{f(z)-g(v)} - \frac{(T_u^{-1})'(z)}{T_u^{-1}(z)-T_u^{-1}(u^*)} +\frac{1}{2}\calA_{T_u^{-1}}(z) - \frac{1}{2} \calA_f(z) \bigg)T_u'(T_u^{-1}(z))\\
        &= \bigg(  \frac{f'(z)}{f(z)-g(v)} -\frac{1}{z-u^*} - \frac{1}{2} \calA_f(z) \bigg)T_u'(T_u^{-1}(z))\\
        &= \calB^*_{f,g,D}(z,u,v)T_u'(T_u^{-1}(z)),
    \end{align*}
    where we see the second equality from Lemma \ref{Lemma:BMobiusZero}, for instance. Thus \eqref{Eq:NPS*MobiusInvarianceB} holds in the finite case, as claimed, and taking limits again handles the infinite cases.

    The covariance \eqref{Eq:CMobiusInvariance} of the $\calC$ operator then immediately follows by \eqref{Eq:NPSMobiusPrePost} and  \eqref{Eq:NPS*MobiusInvariance} applied to \eqref{Eq:CalCDefAppendix}.
\end{proof}

\bibliography{Weld}
\bibliographystyle{alpha}
\end{document}